\documentclass{article}
\usepackage{abstract}
\usepackage{arxiv}

\usepackage[utf8]{inputenc} 
\usepackage[T1]{fontenc}    
\usepackage{hyperref}       
\usepackage{url}            
\usepackage{booktabs}       
\usepackage{amsfonts}       
\usepackage{nicefrac}       
\usepackage{microtype}      
\usepackage{lipsum}		
\usepackage{graphicx}
\usepackage[numbers]{natbib}
\usepackage{doi}
\usepackage{amsmath}
\usepackage{amsthm}
\usepackage{bbm}
\usepackage{caption}
\usepackage{subcaption}
\DeclareMathOperator*{\argmax}{arg\,max}
\DeclareMathOperator*{\argmin}{arg\,min}
\usepackage{multirow}
\usepackage{longtable}
\usepackage{float}
\usepackage{xcolor} 
\usepackage[ruled,vlined]{algorithm2e}
\usepackage{comment}

\definecolor{ForestGreen}{RGB}{34,139,34} 

\title{Optimal local interventions in the two-dimensional Abelian Sandpile model}

\date{\today}	

\author{ \href{https://orcid.org/0000-0000-0000-0000}{\includegraphics[scale=0.06]{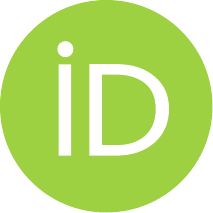}\hspace{1mm}Maike C. de Jongh}$^{a}$\\
	\texttt{m.c.dejongh@utwente.nl} \\
	\And
	\href{https://orcid.org/0000-0000-0000-0000}{\includegraphics[scale=0.06]{orcid.pdf}\hspace{1mm}Richard J. Boucherie}$^{a}$ \\
	\texttt{r.j.boucherie@utwente.nl} \\[0.2in]
	\mbox{}$^{a}$Department of Applied Mathematics\\
	University of Twente\\
	P.O. Box 217, NL-7500 AE Enschede \\[0.1in]
	\mbox{}$^{b}$Centrum Wiskunde \& Informatica\\
	P.O. Box 94079, NL-1090 GB Amsterdam
   	\And
    \href{https://orcid.org/0000-0000-0000-0000}{\includegraphics[scale=0.06]{orcid.pdf}\hspace{1mm}M.N.M. van Lieshout}$^{b,a}$\\ 
	\texttt{m.n.m.van.lieshout@cwi.nl} \\ 
}

\renewcommand{\shorttitle}{Optimal local interventions in the two-dimensional Abelian sandpile model}

\hypersetup{
pdftitle={A template for the arxiv style},
pdfsubject={q-bio.NC, q-bio.QM},
pdfauthor={David S.~Hippocampus, Elias D.~Striatum},
pdfkeywords={First keyword, Second keyword, More},
}

\allowdisplaybreaks

\begin{document}

\theoremstyle{plain}
\newtheorem{axiom}{Axiom}
\newtheorem{claim}[axiom]{Claim}
\newtheorem{theorem}{Theorem}[section]
\newtheorem{lemma}[theorem]{Lemma}
\newtheorem{corollary}[theorem]{Corollary}

\theoremstyle{remark}
\newtheorem{definition}[theorem]{Definition}
\newtheorem*{example}{Example}
\newtheorem*{fact}{Fact}
\newtheorem{assumption}{Assumption}
\newtheorem{remark}{Remark}

\maketitle
\begin{abstract}
The Abelian sandpile model serves as a canonical example of self-organized criticality. This critical behavior manifests itself through large cascading events triggered by small perturbations. Such large-scale events, known as avalanches, are often regarded as stylized representations of catastrophic phenomena, such as earthquakes or forest fires. Motivated by this perspective, we study strategies to reduce avalanche sizes. We provide a first rigorous analysis of the impact of interventions in the Abelian sandpile model, considering a setting in which an external controller can perturb a configuration by removing sand grains at selected locations. We first develop and formalize an extended method to compute the expected size of an avalanche originating from a connected component of critical vertices, i.e., vertices at maximum height. Using this method, we characterize the structure of avalanches starting from square components and explicitly analyze the effect of interventions in such components. Our results show that the optimal intervention locations strike an interesting balance between reduction of largest avalanche sizes and increasing the number of mitigated avalanches.  
\end{abstract}


\section{Introduction}
\label{Introduction}
The Abelian sandpile model was proposed by \citet{Bak1, Bak2} to study systems driven by small external forces that organize themselves by means of energy-dissipating avalanches. In this model, grains of sand are successively dropped on randomly chosen vertices of a graph. Whenever a vertex accumulates too many grains, it topples and transfers them to its neighbours, which may trigger a cascade involving many other vertices. An interesting aspect of the model is that critical behaviour emerges without the need to tune any external parameter. This phenomenon, known as self-organized criticality, manifests itself in a range of real-world systems, including earthquakes \citep{Scholz}, forest fires \citep{Drossel}, financial markets \citep{Bartolozzi}, and the evolution of species \citep{Bak3}. In many of these applications, large avalanches may have destructive consequences. For this reason, a growing body of literature has focused on the use of intervention heuristics for controlling the size of avalanches. Several studies provide an empirical investigation of avalanche propagation in cases where an external controller can influence the landing position of the next sand grain \citep{Cajueiro1, Falk, Noel, Parsaeifard, Turalska}. In contrast, \citet{Qi} point out that in many cases the origin of an avalanche is uncontrollable and therefore focus on control strategies that intervene once an avalanche has started. In particular, they study a situation in which vertices can be prevented from toppling, even when the number of sand grains exceeds their capacity. Other intervention strategies that have been explored mitigate the risk of large-scale avalanches by deliberately triggering smaller ones in a controlled way \citep{Cajueiro2, Cajueiro3, Hou}. Another line of research explores the impact of sand grain absorption on avalanche sizes \citep{Scala}. 

In contrast to the studies mentioned above, which provide numerical results on control strategies, we present a rigorous analysis of the impact of strategic interventions. In particular, we consider a scenario in which an external controller can remove sand grains from selected vertices. In our analysis, we focus on connected components of vertices that hold the maximum amount of sand grains, called \textit{generators}, and investigate how removing sand grains affects avalanches originating from these generators. Our contribution is twofold. First, we formalize and extend the method proposed by \citet{Dorso} for computing the expected size of an avalanche starting from a given generator. We identify the cases in which their method does not apply and propose appropriate modifications. In addition, we provide a formal justification that the resulting algorithm indeed yields the correct result for all possible generators. Second, we consider generators in the shape of squares and provide rigorous results on the effect of removing sand grains, identifying optimal targets for interventions. The motivation for analyzing such generators is that, when surrounded by noncritical vertices, they allow for a local analysis. The boundary of noncritical vertices prevents the propagation of avalanches beyond the square. Consequently, the internal avalanche dynamics are unaffected by the remainder of the configuration. This enables us to characterize the structure of avalanches originating from the square. Moreover, the size of avalanches in this setting provides a natural upper bound for avalanches in any square-shaped region surrounded by noncritical vertices.

The structure of this paper is as follows. In Section \ref{Definition of the Abelian sandpile model}, we give a definition of the Abelian sandpile model. Section~\ref{Avalanche development} presents our adaptation of the method of \citet{Dorso} for computing the expected avalanche size corresponding to a particular generator. This method makes use of the decomposition of an avalanche into a sequence of waves, introduced by \citet{Ivashkevich}. We explain how our extension ensures that the method becomes applicable in all cases and we provide a proof of correctness. In Section \ref{Reducing avalanche sizes through strategic interventions} we explore the impact of strategic interventions on expected avalanche sizes. We derive rigorous results on the impact of removing sand grains from vertices in square-shaped generators by studying the structure of the waves that constitute an avalanche. Also, we identify the set of vertices that are the optimal targets for the removal of sand grains in this case. Some proofs are deferred to the Appendix. 
Finally, we reflect on our results in Section \ref{Conclusion} and indicate some directions for future research.

\section{Definition of the Abelian sandpile model}
\label{Definition of the Abelian sandpile model}
We consider the Abelian sandpile model on a graph that is constructed from a finite, two-dimensional box $V =~\{1,\ldots, L\}^2 \cap \mathbb{Z}^2$ in the following way. First, all vertices in the set $V^c = \mathbb{Z}^2 \setminus V$ are identified with a single vertex $s$, called the \textit{sink}. Next, all self-loops at $s$ are removed. We refer to the resulting graph as the \textit{wired lattice} of size $L$. By construction, the corner vertices of the box are connected to the sink by two edges, whereas each of the remaining boundary vertices has exactly one edge connecting it to the sink. A \textit{sandpile}, denoted by a map $\eta: V \rightarrow \mathbb{Z}$, is a configuration of indistinguishable particles, called \textit{sand grains}, on this wired lattice.
Let the set of all sandpiles on~$V$ be denoted by $\mathcal{X}$. We define the \textit{mass} $m(\eta)$ of a sandpile $\eta \in \mathcal{X}$ as the total number of sand grains in the sandpile~$\eta$. Furthermore, we call a vertex $v \in V$ \textit{stable} in a sandpile $\eta \in \mathcal{X}$ if $0 \leq \eta(v) < 4$. If all vertices $v \in V$ are stable in $\eta$, then $\eta$ is called a \textit{stable sandpile}. We denote the set of stable sandpiles on $V$ by $\Omega$. A vertex $v \in V$ is \textit{critical in} $\eta \in \Omega$ if it contains three sand grains, i.e., if $\eta(v) = 3$. A vertex $v \in V$ that satisfies $\eta(v) < 3$ is called \textit{noncritical in $\eta$}.

The \textit{sandpile Markov chain} evolves according to the following dynamics: each time step, a new sand grain is dropped on a vertex $v \in V$ selected uniformly at random. If this new sand grain causes the vertex $v$ to become unstable, then each of the sand grains located at $v$ will be sent to one of its horizontal and vertical neighbours, an operation called \textit{toppling}. Unstable vertices now continue to topple until a stable sandpile is reached, a process called \textit{relaxation}. Sand grains that end up in the sink during the relaxation are discarded. To formalize these dynamics, we label the vertices as $V = \{v_1, v_2, \ldots, v_{L^2}\}$ and write $v_i \sim v_j$ if vertices $v_i$ and $v_j$ are neighbours. Then, we define the $L^2 \times L^2$ matrix $\Delta$, of which the $i$th row/column corresponds to vertex $v_i$, as 
\begin{equation*}
    \Delta_{ij} = \begin{cases}
        4, &\text{if } i = j, \\
        -1, &\text{if } i \neq j, \quad v_i \sim v_j, \\
        0, &\text{otherwise.}
    \end{cases}
\end{equation*}
We now define several operators on sandpiles $\eta \in \mathcal{X}$. 
First, let $\alpha_i: \mathcal{X} \rightarrow \mathcal{X}$ denote an addition operator, which adds a sand grain to vertex $v_i \in V$, i.e.,
    \begin{equation*}
        (\alpha_i \eta)(v_j) = \begin{cases}
            \eta(v_j) + 1, &\text{if } i = j, \\
            \eta(v_j), &\text{otherwise}.
        \end{cases}
    \end{equation*}
Second, we define a removal operator $\gamma_i: \mathcal{X} \rightarrow \mathcal{X}$, which removes all sand grains from vertex $v_i \in V$, i.e.,
    \begin{equation*}
        (\gamma_i \eta)(v_j) = \begin{cases}
            0, &\text{if } i = j, \\
            \eta(v_j), &\text{otherwise.}
        \end{cases}
    \end{equation*}
Third, let $\tau_i: \mathcal{X} \rightarrow \mathcal{X}$ denote a toppling operator, which represents the procedure of \textit{toppling} a vertex $v_i$, i.e.,
    \begin{equation*}
        (\tau_i \eta)(v_j) = \eta(v_j) - \Delta_{ij}.
    \end{equation*}
    We call a toppling of a vertex $v_i \in V$ \textit{legal} with respect to a sandpile $\eta \in \mathcal{X}$ if $\eta(v_i) > 3$. \citet{Dhar} showed that toppling operators commute. It follows from similar reasoning that toppling operators commute with addition operators. 
    
Fourth, we define an identity operator $I: \mathcal{X} \rightarrow \mathcal{X}$, which maps a sandpile configuration to itself, i.e., $I \eta = \eta$.

Finally , let $R$ denote a relaxation operator, which takes as input a sandpile $\eta \in \mathcal{X}$ and outputs the stable sandpile $\eta' = R \eta$ that results from toppling unstable vertices until a stable sandpile is reached. As shown by \citet{Dhar}, each sequence of topplings of unstable vertices is finite and results in the same unique stable sandpile, which ensures that the operator $R$ is well-defined. Hence, 
    \begin{equation*}
        R\eta = \tau_{i_k} \tau_{i_{k-1}} \cdots \tau_{i_1} \eta,
    \end{equation*}
    where $(\tau_{i_1}, \tau_{i_2}, \ldots, \tau_{i_k})$, $k = 0, 1, \ldots$, is any sequence of legal topplings such that $\tau_{i_k} \tau_{i_{k-1}} \cdots \tau_{i_1} \eta$ is a stable sandpile configuration.  

The dynamics of the sandpile Markov chain can now be expressed in terms of the following transition probability kernel:
\begin{equation*}
    \mathbb{P}(\eta, \eta') = \dfrac{1}{L^2} \sum\limits_{i = 1}^{L^2}  \mathbbm{1}[\eta' = R\alpha_i \eta], \quad \eta, \eta' \in \Omega.
\end{equation*}
Our main variable of interest is the \textit{avalanche size}. \citet{Dhar} showed that the number of times a vertex $v_j \in V$ topples during the avalanche induced by dropping a grain at vertex $v_i \in V$ onto a sandpile $\eta$ is the same for each sequence of legal topplings leading from $\alpha_i \eta$ to $R \alpha_i \eta$. Let this number be denoted by $n(v_i, v_j; \eta)$. We define the \textit{size} $x(\eta, v_i)$ \textit{of the avalanche} caused by dropping a sand grain at vertex $v_i \in V$ onto sandpile $\eta$ as the total number of topplings that occur during the relaxation of the pile $\alpha_i \eta$, i.e., 
\begin{equation*}
    x(\eta, v_i) := \sum\limits_{j = 1}^{L^2} n(v_i, v_j; \eta).
\end{equation*}
Let the (random) size of the next avalanche that will occur given a sandpile $\eta \in \Omega$ be denoted by $X(\eta)$.
Note that 
\begin{equation*}
    \mathbb{P}(X(\eta) = k) = \dfrac{1}{L^2} \sum\limits_{i=1}^{L^2}  \mathbbm{1}[x(\eta, v_i) = k], \quad k = 0, 1, \ldots.
\end{equation*}
In this work, we are particularly interested in the avalanche size conditional on the event that the next sand grain lands in a subset $A \subseteq V$. 

\section{Avalanche development}
\label{Avalanche development}
In order to analyze the development of avalanches and investigate prevention strategies, we build on the approach for computing the expected avalanche size proposed by \citet{Dorso}. Their method relies on the representation of an avalanche as a sequence of waves, a concept introduced by \citet[p.~350]{Ivashkevich}. Such a representation is established by choosing a specific type of sequence of legal topplings to relax an unstable sandpile, namely as follows. Suppose that $\eta \in \mathcal{X}$ is a sandpile that satisfies $\eta(v_i) = 3$ for some $v_i \in V$ and $\eta(v_j) \leq 3$ for all $v_j \in V$, $v_j \neq v_i$. We now drop a new sand grain at vertex $v_i$ and choose any sequence of legal topplings that begins by toppling vertex $v_i$ and then proceeds to topple unstable vertices in the set $V \setminus \{v_i\}$ until all vertices are stable, except possibly for $v_i$. Hence, let $(\tau_{i_1}, \tau_{i_2}, \ldots, \tau_{i_k})$ denote a sequence of topplings such that $i_1 = i$ and $i_j \neq i$ for all $j = 2, \ldots, k$, the toppling $\tau_{i_j}$, $j = 2, \ldots, k$, is legal for the sandpile $\tau_{i_{j-1}} \cdots \tau_{i_1} \alpha_i \eta$ and $\tau_{i_k} \cdots \tau_{i_1} \alpha_i \eta(v_j)  \leq 3$ for all $j \neq i$. \citet{Ivashkevich} show that $\{i_1, i_2, \ldots, i_k\}$ is a set of $k$ unique vertices. Hence, no vertex topples twice during the sequence of topplings $(\tau_{i_1}, \tau_{i_2}, \ldots, \tau_{i_k})$. The set of vertices $W_{i1}(\eta) := \{i_1, i_2, \ldots, i_k\}$ is called the \textit{first wave} of the avalanche generated by dropping a sand grain at vertex $v_i$ onto $\eta$. Let $R^{\eta}_{i1}$ denote the operator that yields the sandpile that results from toppling the vertices in the first wave, i.e.,
\begin{equation*}
    R^{\eta}_{i1} \alpha_i \eta := \tau_{i_k} \cdots \tau_{i_1}\alpha_i \eta.
\end{equation*}
If $R^{\eta}_{i1} \alpha_i \eta(v_i) > 3$, then relaxing this sandpile requires toppling vertex $v_i$ a second time. Repeating the same procedure yields the \textit{second wave} $W_{i2}(\eta)$ and corresponding operator $R^{\eta}_{i2}$. Since any sandpile can be stabilized through a finite number of topplings \citep{Dhar}, continuing this process yields a finite number of waves $W_{i1}(\eta), \ldots, W_{in}(\eta)$ and corresponding operators $R^{\eta}_{i1}, \ldots, R^{\eta}_{in}$, resulting in a stable sandpile $R^{\eta}_{in} \cdots R^{\eta}_{i1} \alpha_i \eta$. 

We proceed to discuss the key ideas of the method proposed by \citet{Dorso}, identify its limitations, and develop an extension that works in full generality, along with a formal justification. Suppose that dropping a sand grain at a critical vertex $v_i \in V$ onto a stable sandpile $\eta \in \Omega$ yields an avalanche that can be represented as a sequence of $n(\eta, v_i)$ waves. Let $w_{ij}(\eta)$ denote the size of the $j$th wave, i.e., $w_{ij}(\eta) := |W_{ij}(\eta)|$, $j = 1, \ldots, n(\eta, v_i)$. Note that the size of the avalanche generated by dropping a sand grain at $v_i$ onto $\eta$ is given by 
\begin{equation}\label{sum_of_waves}
    x(\eta, v_i) = \sum\limits_{j = 1}^{n(\eta, v_i)} w_{ij}(\eta).
\end{equation}
Let $A \subseteq V$ denote the connected component of critical vertices in $\eta$ that contains $v_i$. We refer to such a connected component as a \textit{generator}. Note that each vertex in $A$ will topple during the first wave, i.e., $A \subseteq W_{i1}(\eta)$. To see this, suppose that some vertex $v \in A$ does not topple during the first wave. Since $\eta(v) = 3$, this implies that none of the neighbours of $v$ topples during the first wave. Continuing this argument and using the fact that $v$ and $v_i$ are in the same connected component of critical vertices, this implies that $v_i$ does not topple in the first wave, which yields a contradiction.
\citet{Dorso} claim that the first wave of the avalanche caused by dropping a sand grain at vertex $v \in A$ is the same for each $v \in A$, a statement we formally prove in Lemma \ref{first_wave}. Let $W_A(\eta)$ denote the first wave of an avalanche arising from dropping a sand grain on some vertex $v \in A$ and let $w_A(\eta)$ denote its size. These are guaranteed to be well-defined by Lemma \ref{first_wave}. Now, let $(x_1, \ldots, x_k)$ be a permutation of the elements of $W_A(\eta)$. The key idea underlying the method of \citet{Dorso} is the fact that toppling operators and addition operators commute and therefore $\tau_{x_k} \cdots \tau_{x_1} \alpha_i \eta = \alpha_i \tau_{x_k} \cdots \tau_{x_1} \eta$. They argue that it suffices to study the sandpile configuration $\eta^A$, defined as $\eta^A = \tau_{x_k} \cdots \tau_{x_1} \eta$. By the fact that topplings commute, $\eta^A$ does not depend on the specific choice of permutation and is therefore well-defined. Note that the topplings in the sequence $(\tau_{x_1}, \ldots, \tau_{x_k})$ are not necessarily legal with respect to $\eta$ and it is possible that there exists a $v \in V$ such that $\eta^A(v) < 0$. An implicit assumption in the approach of \citet{Dorso} is that the sizes of the second and higher order waves are equal as well for all vertices $v \in A$ at which the next sand grain may land, allowing for expressions such as \citep[expression (20)]{Dorso}. This is the case if the set $\{v \in A|\eta^A(v) = 3\}$ is connected. In general, however, this is not guaranteed. Figure \ref{Dorso_Dadamia_failure} displays a sandpile $\eta$ and generator $A$ (outlined in bold in the left panel) for which the set $\{v \in A| \eta^A(v) = 3\}$ splits into two components (shown in the right panel). If a sand grain lands on a vertex in the upper component, it produces a second wave of size 2, whereas hitting the lower component generates a second wave of size 1.

\begin{figure}
   \begin{subfigure}[b]{0.5\textwidth}
        \centering
        \includegraphics[width=0.6\textwidth]{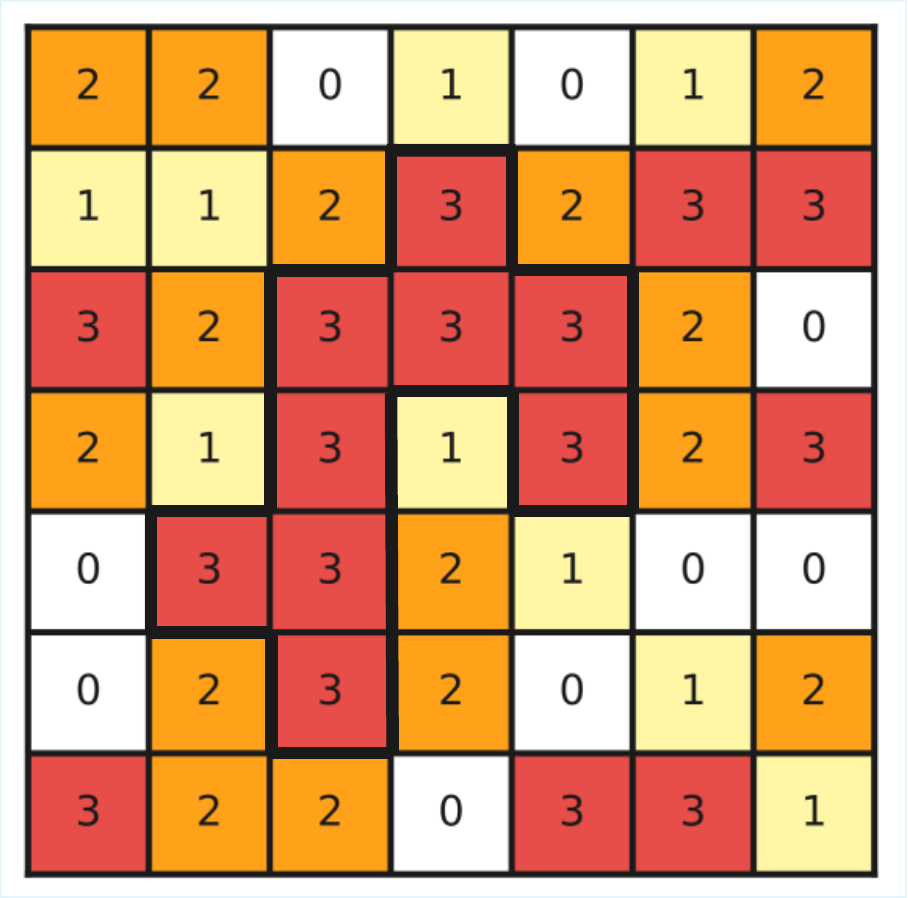}
        \label{DorsoDadamia_failure_example_1}
    \end{subfigure}
    \begin{subfigure}[b]{0.5\textwidth}  
        \centering 
        \includegraphics[width=0.6\textwidth]{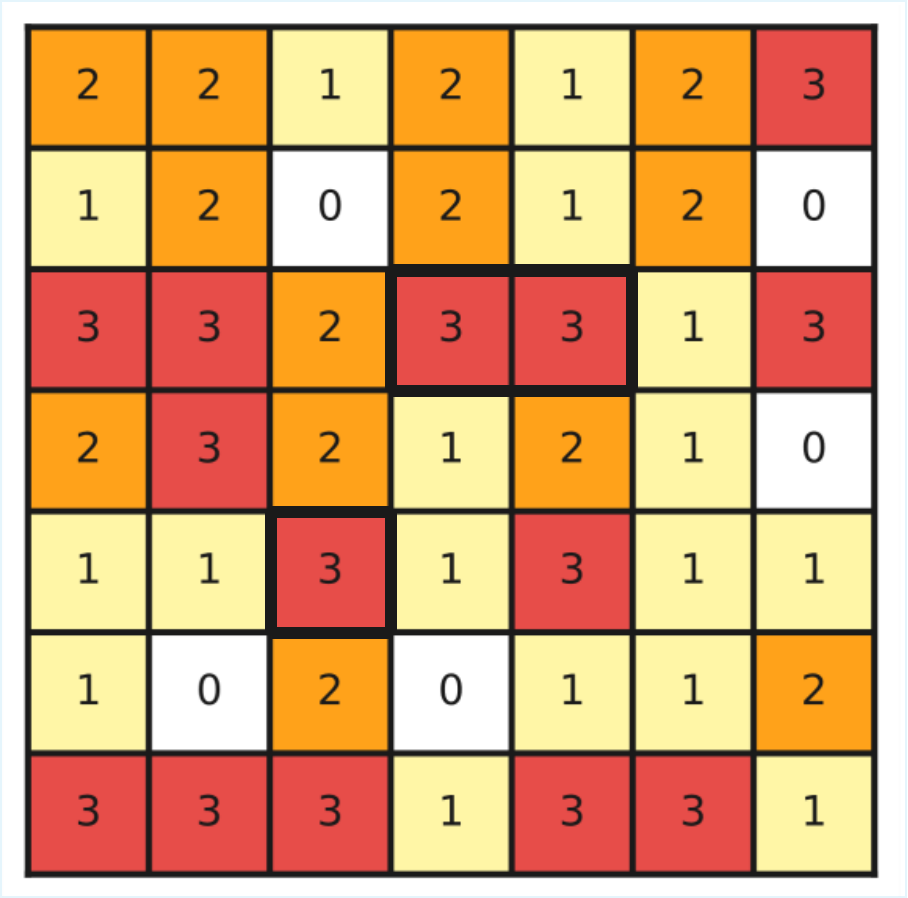}
        \label{DorsoDadamia_failure_example_2}
    \end{subfigure}
    \caption{Example to which the method of \citet{Dorso} does not apply. The generator outlined in bold splits into two components after the first wave, each producing second waves of different sizes.} 
    \label{Dorso_Dadamia_failure}
    \end{figure}

We now propose a modification of the method of \citet{Dorso}, which applies in full generality, along with a formal justification. We start by proving that the first wave of the avalanche produced by dropping a sand grain at a vertex $v \in A$ is the same for each $v \in A$.
\newpage

\begin{lemma}\label{first_wave}
    Given a stable sandpile configuration $\eta \in \Omega$ and a generator $A = \{v_{i_1}, v_{i_2}, \ldots, v_{i_k}\} \subseteq V$, we have $W_{i_{\ell}1}(\eta) = W_{i_m1}(\eta)$ for all $\ell, m = 1, \ldots, k$.  
\end{lemma}
\begin{proof}
    Consider $v_{i_{\ell}}, v_{i_m} \in A$. Let $W_{i_{\ell}1}(\eta) = A \cup \tilde{W}_{i_{\ell}1}(\eta)$ and $W_{i_m1}(\eta) = A \cup \tilde{W}_{i_m1}(\eta)$, where $\tilde{W}_{i_{\ell}1}(\eta), \tilde{W}_{i_m1}(\eta) \subseteq V \setminus A$. Suppose first that $\tilde{W}_{i_{\ell}1}(\eta) = \emptyset$. We show that this implies that $\tilde{W}_{i_m1}(\eta) = \emptyset$. Since each vertex in $A$ will topple exactly once during the first wave and $\tilde{W}_{i_{\ell}1}(\eta) = \emptyset$ \citep{Ivashkevich}, there exists a sequence of legal topplings $(\tau_{x_1}, \tau_{x_2}, \ldots, \tau_{x_k})$ acting on $\alpha_{i_{\ell}}\eta$, where $(x_1, \ldots, x_k)$ is a permutation of the set $\{i_1, \ldots, i_k\}$, such that $\eta' = \tau_{x_k} \cdots \tau_{x_1} \alpha_{i_{\ell}} \eta$ is a sandpile that satisfies $\eta'(v) \leq 3$ for all $v \neq v_{i_{\ell}}$. Note that $\tau_{x_k} \cdots \tau_{x_1} \alpha_{i_{\ell}} \eta = \alpha_{i_{\ell}} \tau_{x_k} \cdots \tau_{x_1} \eta$, by the fact that toppling operators and addition operators commute. Since each vertex in $A$ toppled exactly once, we have $\eta'(v_{i_{\ell}}) \leq 4$. Also, since $A$ is a connected component there exists a sequence of legal topplings $(\tau_{y_1}, \tau_{y_2}, \ldots, \tau_{y_k})$ acting on $\alpha_{i_m} \eta$, where $(y_1, \ldots, y_k)$ is again a permutation of the set $\{i_1, \ldots, i_k\}$. Let $\tilde{\eta} = \tau_{y_k} \cdots \tau_{y_1} \alpha_{i_m} \eta = \alpha_{i_m} \tau_{y_k} \cdots \tau_{y_1} \eta$. Since topplings commute, we have $\tau_{x_k} \cdots \tau_{x_1} \eta = \tau_{y_k} \cdots \tau_{y_1} \eta$. It now follows from the fact that $\eta'$ satisfies $\eta'(v) \leq 3$ for all $v \neq v_{i_{\ell}}$ and $\eta'(v_{i_{\ell}}) \leq 4$ that $\tilde{\eta}(v) \leq 3$ for all $v \neq v_{i_m}$. Thus, $\tilde{W}_{i_m1}(\eta) = \emptyset$. 

    Now, suppose that $\tilde{W}_{i_{\ell}1}(\eta) = \{v_{j_1}, v_{j_2}, \ldots, v_{j_n}\}$. Since $A$ is a connected component of critical vertices, there exists a sequence of legal topplings $(\tau_{x_1}, \tau_{x_2}, \ldots, \tau_{x_k})$ acting on $\alpha_{i_{\ell}}\eta$, where $(x_1, \ldots, x_k)$ is a permutation of the set $\{i_1, \ldots, i_k\}$. Hence, there exists a sequence of legal topplings $(\tau_{x_1}, \tau_{x_2}, \ldots, \tau_{x_k}, \tau_{\tilde{x}_1}, \tau_{\tilde{x}_2}, \ldots, \tau_{\tilde{x}_n})$ acting on $\alpha_{i_{\ell}} \eta$, where $(\tilde{x}_1, \tilde{x}_2, \ldots, \tilde{x}_n)$ is a permutation of the set $\{j_1, \ldots, j_n\}$, such that $\eta' := \tau_{\tilde{x}_n} \cdots \tau_{\tilde{x}_1} \tau_{x_k} \cdots \tau_{x_1} \alpha_{i_{\ell}} \eta$ is a sandpile configuration that satisfies $\eta'(v) \leq 3$ for all $v \neq v_{i_{\ell}}$. Since each vertex in $W_{i_{\ell}1}$ topples only once, we obtain $\eta'(v_{i_{\ell}}) \leq 4$. Again using the fact that $A$ is a connected component of critical vertices, we can construct a sequence of legal topplings $(\tau_{y_1}, \tau_{y_2}, \ldots, \tau_{y_k})$ acting on $\alpha_{i_m} \eta$, where $(y_1, \ldots, y_k)$ is a permutation of the set $\{i_1, \ldots, i_k\}$. Now, note that $\eta' = \tau_{\tilde{x}_n} \cdots \tau_{\tilde{x}_1} \tau_{x_k} \cdots \tau_{x_1} \alpha_{i_{\ell}} \eta = \tau_{\tilde{x}_n} \cdots \tau_{\tilde{x}_1} \alpha_{i_{\ell}} \tau_{x_k} \cdots \tau_{x_1} \eta$. By the fact that topplings commute, the sandpile configurations $\alpha_{i_{\ell}} \tau_{x_k} \cdots \tau_{x_1} \eta$ and $\alpha_{i_m} \tau_{y_k} \cdots \tau_{y_1}\eta$ only differ at vertices $v_{i_{\ell}}$ and $v_{i_m}$. Since $\{v_{j_1}, \ldots, v_{j_n}\} \subseteq V \setminus A$ and thus $\{v_{j_1}, \ldots, v_{j_n}\} \cap \{v_{i_{\ell}}, v_{i_m}\} = \emptyset$, it follows that $(\tau_{\tilde{x_1}}, \tau_{\tilde{x_2}}, \ldots, \tau_{\tilde{x}_n})$ is a sequence of legal topplings for the sandpile configuration $\alpha_{i_m} \tau_{y_k} \cdots \tau_{y_1} \eta$. Let $\tilde{\eta} := \tau_{\tilde{x}_n} \cdots \tau_{\tilde{x}_1} \alpha_{i_m} \tau_{y_k} \cdots \tau_{y_1} \eta =  \tau_{\tilde{x}_n} \cdots \tau_{\tilde{x}_1} \tau_{y_k} \cdots \tau_{y_1} \alpha_{i_m} \eta$. Note that $\eta' = \alpha_{i_{\ell}} \tau_{\tilde{x}_n} \cdots \tau_{\tilde{x}_1} \tau_{x_k} \cdots \tau_{x_1} \eta$ and $\tilde{\eta} = \alpha_{i_m} \tau_{\tilde{x}_n} \cdots \tau_{\tilde{x}_1} \tau_{y_k} \cdots \tau_{y_1} \eta$. From the facts that topplings commute, $\eta'(v) \leq 3$ for all $v \neq v_{i_{\ell}}$ and $\eta'(v_{i_{\ell}}) \leq 4$, it now follows that $\tilde{\eta}(v) \leq 3$ for all $v \neq v_{i_m}$ and $\tilde{\eta}(v_{i_m}) \leq 4$. Thus, $\tilde{W}_{i_m1}(\eta) = \tilde{W}_{i_{\ell}1}(\eta)$. 
\end{proof}

Given a generator $A \subseteq V$ in a sandpile configuration $\eta \in \Omega$, we proceed to establish a recursive expression for computing $\mathbb{E}[X(\eta)|Y \in A]$, where $Y$ denotes the random vertex at which the next sand grain lands. Let $W_A(\eta)$ denote the first wave of an avalanche arising from dropping a sand grain on some vertex $v \in A$ and let $w_A(\eta)$ denote its size. These are guaranteed to be well-defined by Lemma \ref{first_wave}. Now, let $(x_1, \ldots, x_k)$ be a permutation of the elements of $W_A(\eta)$ and let the sandpile configuration $\eta^A$ be defined as $\eta^A = \tau_{x_k} \cdots \tau_{x_1} \eta$. By the fact that topplings commute, $\eta^A$ does not depend on the specific choice of the permutation and is therefore well-defined. Note that the topplings in the sequence $(\tau_{x_1}, \ldots, \tau_{x_k})$ are not necessarily legal with respect to $\eta$ and it is possible that there exists a $v \in V$ such that $\eta^A(v) < 0$. We now define the set $\tilde{A}$ as
\begin{equation*}
    \tilde{A} = \{v \in A| \eta^A(v) < 3\}.    
\end{equation*}
Furthermore, if the set $A' := \{v \in A|\eta^A(v) = 3\}$ is nonempty, let $A_1, \ldots, A_{\ell}$ denote its connected components. Lemma \ref{exp_av_size_rec_thm} now provides a recursive expression for computing $\mathbb{E}[X(\eta)|Y \in A]$.

\begin{lemma}\label{exp_av_size_rec_thm}
  Let $\eta \in \Omega$ denote a stable sandpile configuration and let $A \subseteq V$ denote a generator in $\eta$. Let $Y$ denote the random vertex at which the next sand grain lands. The expected value of the size of the next avalanche corresponding to $A$ satisfies
  \begin{equation}\label{exp_av_size_rec_eq}
      \mathbb{E}[X(\eta)|Y \in A] = \begin{cases}
          w_A(\eta) + \sum\limits_{j=1}^{\ell} \dfrac{|A_j|}{|A|} \mathbb{E}[X(\eta^A)|Y \in A_j], &\text{if } \{v \in A|\eta^A(v) = 3\} \neq \emptyset, \\
          w_A(\eta), &\text{otherwise}. 
      \end{cases}
  \end{equation}
\end{lemma}
\begin{proof}
    Let $\tilde{H}$ denote the subset of vertices in $A$ that, when hit by a new sand grain, generate an avalanche consisting of a single wave. Let $H' = A \setminus \tilde{H}$. If $H' \neq \emptyset$, let $H_1, \ldots, H_{\ell}$ denote its connected components. By conditioning on $Y$, we obtain
    \begin{align*}
        \mathbb{E}[X(\eta)|Y \in A] &= \dfrac{|\tilde{H}|}{|A|} \mathbb{E}[X(\eta)|Y \in \tilde{H}]  + \sum\limits_{j=1}^{\ell} \dfrac{|H_j|}{|A|} \mathbb{E}[X(\eta)|Y \in H_j] \\
        &= \dfrac{|\tilde{H}|}{|A|} w_A(\eta) + \sum\limits_{j=1}^{\ell} \dfrac{|H_j|}{|A|}\mathbb{E}[X(\eta)|Y \in H_j],
    \end{align*}
    if $H' \neq \emptyset$ and
    \begin{equation*}
        \mathbb{E}[X(\eta)|Y \in A] = w_A(\eta),
    \end{equation*}
    otherwise. Hence, to establish expression (\ref{exp_av_size_rec_eq}), it suffices to show that $\tilde{H} = \tilde{A}$, $H_j = A_j$ for all $j = 1, \ldots, \ell$ and
    \begin{equation}\label{exp_av_size_rec_help}
        \mathbb{E}[X(\eta)|Y \in A_j] = w_A(\eta) + \mathbb{E}[X(\eta^A)|Y \in A_j].
    \end{equation}
    We start by showing that $\tilde{H} = \tilde{A}$. First, consider a vertex $v_i \in \tilde{H}$. Let $(x_1, \ldots, x_k)$ be a permutation of $W_A(\eta)$. Since an avalanche caused by dropping a sand grain at $v_i$ consists of only one wave, we have $\tau_{x_k} \cdots \tau_{x_1} \alpha_i \eta(v_i) \leq 3$. Since $\tau_{x_k} \cdots \tau_{x_1} \alpha_i \eta = \alpha_i \tau_{x_k} \cdots \tau_{x_1} \eta$, this implies that $\tau_{x_k} \cdots \tau_{x_1} \eta(v_i) = \eta^A(v_i) < 3$. Hence, $v_i \in \tilde{A}$. Now, suppose that $v_i \in \tilde{A}$. This implies that $\eta^A(v_i) < 3$. Therefore, we have $\tau_{x_k} \cdots \tau_{x_1} \alpha_i \eta(v_i) \leq 3$ for each permutation $(x_1, \ldots, x_k)$ of $W_A(\eta)$. It follows that an avalanche generated by dropping a sand grain at vertex $v_i$ consists of only a single wave, and thus $v_i \in \tilde{H}$. Thus, we have $\tilde{H} = \tilde{A}$.

    We proceed to prove that $H_j = A_j$ for all $j = 1, \ldots, \ell$. Let $A' = \bigcup_{j = 1}^{\ell} A_j$. Note that it suffices to show that $H' = A'$. First, consider $v_i \in H'$. Let $(x_1, \ldots, x_k)$ be a permutation of $W_A(\eta)$. Since the avalanche that is generated after a new sand grain drops at $v_i$ can be decomposed into more than one wave, we have $\tau_{x_k} \cdots \tau_{x_1} \alpha_i \eta(v_i) > 3$. Since each vertex in $W_A$ topples only once, this implies $\tau_{x_k} \cdots \tau_{x_1} \alpha_i \eta (v_i) = 4$. Hence, $\tau_{x_k} \cdots \tau_{x_1} \eta(v_i) = \eta^A(v_i) = 3$. It follows that $v_i \in A'$. Now, suppose that $v_i \in A'$. Hence, $\eta^A(v_i) = 3$. This implies that $\tau_{x_k} \cdots \tau_{x_1} \alpha_i \eta (v_i) = 4$ for each permutation $(x_1, \ldots, x_k)$ of $W_A(\eta)$. Therefore, if a new sand grain lands at $v_i$, the avalanche that is generated consists of multiple waves. Hence, $v_i \in H'$. Thus, we obtain $H' = A'$. 

    Finally, we show the validity of expression (\ref{exp_av_size_rec_help}). Through further conditioning on $Y$, we obtain
    \begin{equation*}
        \mathbb{E}[X(\eta)|Y \in A_j] =  \dfrac{1}{|A_j|}\sum\limits_{v_i \in A_j} x(\eta, v_i) . 
    \end{equation*}
    Suppose that the avalanche caused by dropping a sand grain at vertex $v_i$ consists of $k_i$ waves. Invoking expression (\ref{sum_of_waves}), we obtain
    \begin{equation*}
        \mathbb{E}[X(\eta)|Y \in A_j] = \dfrac{1}{|A_j|} \sum\limits_{v_i \in A_j} \left[w_A(\eta) + \sum\limits_{n=2}^{k_i} w_{in}(\eta) \right].
    \end{equation*}
    Let $(x_1, \ldots, x_k)$ denote a permutation of $W_A(\eta)$. Since $\tau_{x_k} \cdots \tau_{x_1} \alpha_i \eta = \alpha_i \tau_{x_k} \cdots \tau_{x_1} \eta = \alpha_i \eta^A$, it follows that $w_{in}(\eta) = w_{i(n-1)}(\eta^A)$ for all $n = 2, \ldots, k_i$. Hence, we obtain
    \begin{align*}
        \mathbb{E}[X(\eta)|Y \in A_j] &= \dfrac{1}{|A_j|} \sum\limits_{v_i \in A_j} \left[w_A(\eta) + \sum\limits_{n=1}^{k_i-1}w_{in}(\eta^A)\right] = w_A(\eta) + \dfrac{1}{|A_j|} \sum\limits_{v_i \in A_j} \sum\limits_{n=1}^{k_i-1} w_{in}(\eta^A) \\
        &= w_A(\eta) + \dfrac{1}{|A_j|} \sum\limits_{v_i \in A_j} x(\eta^A, v_i) = w_A(\eta) + \mathbb{E}[X(\eta^A)|Y \in A_j].
    \end{align*}
    This concludes the proof.
\end{proof}

Based on expression (\ref{exp_av_size_rec_eq}), we recursively define the \textit{depth} $\rho(\eta, A)$ of a generator $A$ in $\eta$ as
\begin{equation}\label{depth_def}
    \rho(\eta, A) = \begin{cases}
        1 + \max\limits_{j = 1, \ldots, \ell}\{\rho(\eta^A, A_j)\}, &\text{if } \{v \in A| \eta^A(v) = 3\} \neq \emptyset,\\
        1, &\text{otherwise.}
    \end{cases}
\end{equation}
Here, recall that the sets $A_j$, $j = 1, \ldots, \ell$, are the connected components of $\{v \in A| \eta^A(v) = 3\} \neq \emptyset$. The depth of a generator $A$ corresponds to the maximum number of topplings at a single vertex $v \in A$ during an avalanche originating from this generator. 

Algorithm \ref{Avalanche_analysis} now provides a method to compute the expected avalanche size based on the recursive expression presented in Lemma \ref{exp_av_size_rec_thm}. Theorem \ref{rec_to_alg} offers theoretical justification for the algorithm.

\begin{algorithm}
 \SetKwInput{Input}{Input}
    \SetKwInOut{Output}{Output}
  \SetAlgoLined
  \Input{A sandpile configuration $\eta \in \Omega$ on the $L\times L$ wired lattice $V$ and a generator $A \subseteq V$.}
  
  $\mathbf{1}.$ Initialize a directed graph $(V, E)$ with $E = \emptyset$. Let $\text{indeg}(v)$ and $\text{outdeg}(v)$, $v \in V$, denote the indegree and the outdegree of vertex $v$ in this graph. \\
  $\mathbf{2}.$ For each $v \in A$, add an edge to $E$ from $v$ to each of the neighbours (w.r.t. the lattice) of $v$;\\
  $\mathbf{3}.$ Set $w' = |A|$; \\
  $\mathbf{4}.$ \While{there is a $v \notin A$ that satisfies $\eta(v) + \text{indeg}(v) - \text{outdeg}(v) \geq 4$}{Select an arbitrary vertex $v \notin A$ that satisfies $\eta(v) + \text{indeg}(v) - \text{outdeg}(v) \geq 4$, add edges to $E$ from $v$ to each of the neighbours (w.r.t. the lattice) of $v$ and increment $w'$ by 1;}
  $\mathbf{5}.$ For each $v \in V$, set $\eta'(v) = \eta(v) + \text{indeg}(v) - \text{outdeg}(v)$;\\
  $\mathbf{6.a}.$ \If{the set $\{v \in A| \eta'(v) = 3\}$ is nonempty,}{Let $a_1, a_2, \ldots, a_{\ell}$ denote its connected components (w.r.t. the lattice). \\
    Let $z_i$ be the output of Algorithm \ref{Avalanche_analysis} for the sandpile configuration $\eta'$ and the generator $a_i$, $i = 1, \ldots, \ell$.\\
  Set $w = w' + \sum\limits_{i = 1}^{\ell} \dfrac{|a_i|}{|A|} z_i$.}
  $\mathbf{6.b}$ \Else{Set $w = w'$.} 
  $\mathbf{7.}$ \Output{Wave size $w$.} 
  \caption{Analysis of avalanche development through decomposition into waves.}
  \label{Avalanche_analysis}
\end{algorithm}

\begin{theorem}\label{rec_to_alg}
    Given a stable sandpile configuration $\eta \in \Omega$ and a generator $A = \{v_{i_1}, v_{i_2}, \ldots, v_{i_k}\} \subseteq V$, Algorithm \ref{Avalanche_analysis} terminates and yields $\mathbb{E}[X(\eta)|Y \in A]$. 
\end{theorem}
\begin{proof}
    First, we show that the quantity $w'$ in Algorithm \ref{Avalanche_analysis} equals $w_A(\eta)$, the size of the first wave. To this end, we start by showing that each vertex $v \in V$ is selected as the tail of new outgoing edges in steps 2-4 at most once. This is obviously true for each vertex $v \in A$. Suppose that some vertex $v \notin A$ is selected twice in step 4. Without loss of generality, let $v$ be the first vertex that is selected for the second time. At this point, we have $\text{outdeg}(v) = 4$. Since $\eta(v) < 4$, this implies that $\text{indeg}(v) > 4$. Hence, some neighbour of $v$ was selected more than once as well, which is a contradiction.  This immediately yields that step 4 terminates, since the number of vertices is finite. 
    
    Now, we prove that a vertex $v \in V$ has outgoing edges if and only if it is part of the first wave $W_A(\eta)$. This naturally holds for all vertices $v \in A$. Now, let $v_{j_1}, v_{j_2}, \ldots, v_{j_m}$ denote an arbitrary total sequence of vertices selected in step~4. Also, let $(x_1, \ldots, x_k)$ be a permutation of $\{i_1, \ldots, i_k\}$. Observe that at the start of step 4, the quantity $\eta(v) + \text{indeg}(v) - \text{outdeg}(v)$ for $v \notin A$ is equivalent to $\tau_{x_k} \cdots \tau_{x_1} \eta(v)$. Hence, the fact that $v_{j_1}$ is selected first in step 4 implies that $\tau_{j_1}$ is a legal toppling in the sandpile $\tau_{x_k} \cdots \tau_{x_1} \eta$. Similarly, at the point that vertex $v_{j_i}$, $i = 2, \ldots, m$ is selected in step~4, the quantity $\eta(v_{j_i}) + \text{indeg}(v_{j_i}) - \text{outdeg}(v_{j_i})$ is equivalent to the number of sand grains that vertex $v_{j_i}$ holds in the sandpile configuration $\tau_{j_{i-1}} \cdots \tau_{j_1} \tau_{x_k} \cdots \tau_{x_1} \eta$. Hence, $\tau_{j_i}$ is a legal toppling in the configuration $\tau_{j_{i-1}} \cdots \tau_{j_1} \tau_{x_k} \cdots \tau_{x_1} \eta$. At the end of step 4, the quantity $\eta(v) + \text{indeg}(v) - \text{outdeg}(v)$ is equal to $\tau_{j_m} \cdots \tau_{j_1} \tau_{x_k} \cdots \tau_{x_1} \eta(v)$ for all $v \notin A$. Since at this point $\eta(v) + \text{indeg}(v) - \text{outdeg}(v) < 4$ for all $v \notin A$, there are no legal topplings in the sandpile configuration $\tau_{j_m} \cdots \tau_{j_1} \tau_{x_k} \cdots \tau_{x_1} \eta$. Hence, we have $W_A(\eta) = A \cup \{v_{j_1}, v_{j_2}, \ldots, v_{j_m}\}$. Since this argument can be applied to each sequence of vertices selected in step 4, it follows immediately that any such sequence is a permutation of the set $\{v_{j_1}, v_{j_2}, \ldots, v_{j_m}\}$. Hence, the quantity $w'$ in Algorithm \ref{Avalanche_analysis} equals $w_A(\eta)$.

    We proceed to show that $\eta' = \eta^A$. Let $\{v_{j_1}, \ldots, v_{j_m}\} = W_A(\eta) \setminus A$ and let $(x_1, \ldots, x_k)$ and $(y_1, \ldots, y_m)$ be permutations of $\{i_1, \ldots, i_k\}$ and $\{j_1, \ldots, j_m\}$. It follows from the arguments above that at the start of step 5, the indegree of a vertex $v \in V$ equals the number of neighbours of $v$ that topple in the first wave and the outdegree of a vertex $v \in V$ equals 0 if $v \notin W_A(\eta)$ and 4 otherwise. This implies that $\eta'(v) = \tau_{y_m} \cdots \tau_{y_1} \tau_{x_k} \cdots \tau_{x_1}\eta(v) = \eta^A(v)$. It now follows that $a_i = A_i$ for all $i = 1, \ldots, \ell$. 

    By the facts that $w' = w_A(\eta)$ and $a_i = A_i$ for all $i = 1, \ldots, \ell$, it follows that step 6 simply computes the recursion in expression (\ref{exp_av_size_rec_eq}).
    
    The facts that each sandpile can be stabilized by means of a finite number of topplings \citep{Dhar} and $w_A(\eta)$ is always positive ensures that the recursion in expression (\ref{exp_av_size_rec_eq}) has finite depth. This implies that Algorithm \ref{Avalanche_analysis} terminates in a finite number of steps.  
\end{proof}

\color{black}

\section{Reducing avalanche sizes through strategic interventions}
\label{Reducing avalanche sizes through strategic interventions}

This section explores the impact of removing sand grains from critical vertices in a sandpile $\eta \in \Omega$. We define the \textit{stability level} $\lambda(\eta, A)$ of a set $A \subseteq V$ in $\eta$ as the quantity
\begin{equation}\label{stability_level}
    \lambda(\eta, A) = \min_{v_i \in A}\dfrac{\mathbb{E}[X(\gamma_i \eta)|Y \in A]}{\mathbb{E}[X(\eta)|Y \in A]},
\end{equation} 
where $Y$ denotes the vertex at which the next sand grain lands.
The stability level is a measure of the impact of the removal of sand grains from a vertex in $A$ on the expected avalanche size. We call the set of vertices in $A$ for which the minimum in expression (\ref{stability_level}) is attained the set of \textit{cornerstone vertices} of $A$, denoted by $B^*_A(\eta)$.

We consider square-shaped generators of size $N \times N$. For such generators, avalanches exhibit a tractable structure, which enables an analytical computation of the expected avalanche size. Moreover, avalanches originating from the square do not propagate beyond its boundary, which allows for a local analysis of avalanche dynamics and control strategies. We explicitly compute the expected avalanche size after interventions at different locations within the square and characterize the optimal targets for sand grain removal. 

Given a set $A \subset V$ that has the shape of an $N \times N$ square, where $N \geq 3$, let $\delta^{\text{in}}(A)$, $C(A)$, $\text{int}(A)$ and $\delta^{\text{out}}(A)$ denote the inner boundary, the corners, the interior and the outer boundary of $A$, i.e., 
    \begin{align*}
        \delta^{\text{in}}(A) &= \{v \in A|\quad \exists \text{ exactly one } w \notin A \text{ such that } v \sim w\}, \\
        C(A) &= \{v \in A|\quad \exists \text{ exactly two } w \notin A \text{ such that } v \sim w\}, \\
        \text{int}(A) &= \{v \in A| \quad w \in A \text{ for all } w \in V \text{ such that } v \sim w\}, \\
        \delta^{\text{out}}(A) &= \{v \notin A|\quad \exists w \in A \text{ such that } v \sim w\}.
    \end{align*} 
Additionally, we define the rings $R_1(A), R_2(A), \ldots, R_{\lceil N/2 \rceil}(A)$ of $A$ as
\begin{align*}
    &R_{\lceil N/2 \rceil}(A) = \delta^{\text{in}}(A) \cup C(A), \\
    &R_j(A) = \delta^{\text{in}}(A \setminus (\cup_{k = j+1}^{\lceil N/2 \rceil}R_k(A))) \cup C(A \setminus (\cup_{k = j+1}^{\lceil N/2 \rceil}R_k(A))), \quad \text{for } j = 2, \ldots, \lceil N/2 \rceil -1,\\
    &R_1(A) = A \setminus (\cup_{k = 2}^{\lceil N/2 \rceil}R_k(A)).
\end{align*}
Let $A^{(N)}$ denote a generator that has the shape of an $N \times N$ square. Theorem \ref{thm_exp_av_size_square} provides the expected avalanche size corresponding to this type of generator. Theorem \ref{thm_depth_square} gives the depth of such a generator. 

\begin{theorem}\label{thm_exp_av_size_square}
Consider a generator $A^{(N)}$ in a sandpile $\eta \in \Omega$ that has the form of an $N\times N$ square. The expected avalanche size corresponding to this generator is given by
\begin{equation}\label{exp_av_size_square}
\mathbb{E}[X(\eta)|Y \in A^{(N)}] = (3N^4+15N^3+20N^2-8)/(30N). 
\end{equation}
\end{theorem}
\begin{proof}
    We compute the expected avalanche size corresponding to the generator $A^{(N)}$ by induction over $N$. We start by showing the validity of expression (\ref{exp_av_size_square}) for $N = 1$ and $N = 2$. First, consider $N = 1$. Let the unique vertex in $A^{(1)}$ be denoted by $v_i$. To find the expected avalanche size corresponding to $A^{(1)}$, we use Algorithm \ref{Avalanche_analysis}. After initializing the directed graph $(V, E)$ with $E = \emptyset$, we add an edge to $E$ from $v_i$ to each of its neighbours and set $w' = 1$. Now, consider $v \in V$, $v \neq v_i$. If $v$ is not a neighbour of $v_i$ (w.r.t. the lattice), we have $\text{indeg}(v) = 0$ and thus $\eta(v) + \text{indeg}(v) < 4$. If, on the other hand, $v$ is a neighbour of $v_i$ (w.r.t. the lattice), we have $\text{indeg}(v) = 1$. Also, the fact that $v \notin A^{(1)}$ implies $\eta(v) < 3$. Hence, $\eta(v) + \text{indeg}(v) < 4$. Therefore, step 4 of the algorithm terminates and we obtain $w_{A^{(1)}}(\eta) = 1$ and $A^{(1)}_1 := \{v \in A^{(1)}|\eta^{A^{(1)}}(v) = 3\} = \emptyset$. Thus, we have $\mathbb{E}[X(\eta)|Y \in A^{(1)}] = 1$, which is consistent with expression~(\ref{exp_av_size_square}). 
    
    We proceed to consider $N = 2$. Again, we initialize the directed graph $(V, E)$ with $E = \emptyset$ and we augment $E$ with an edge from $v$ to each of its neighbours for each $v \in A^{(2)}$. Note that, at this point, we have $\eta(v) + \text{indeg}(v) - \text{outdeg}(v) < 4$ for each $v \notin A^{(2)}$. Hence, step 4 of Algorithm \ref{Avalanche_analysis} terminates and we obtain $w_{A^{(2)}}(\eta) = 4$ and $A^{(2)}_1 := \{v \in A^{(2)}|\eta^{A^{(2)}}(v) = 3\} = \emptyset$. The directed graph and the sandpile configuration $\eta^{A^{(2)}}$ are shown in Figure \ref{Figs_full_square_small_even}. We obtain $\mathbb{E}[X(\eta)|Y \in A^{(2)}] = 4$, which is consistent with expression (\ref{exp_av_size_square}). 

     \begin{figure}[H]
    \centering
    \captionsetup[subfigure]{justification=centering, labelformat=empty, singlelinecheck=false, width=0.6\linewidth}

    \begin{subfigure}[b]{0.45\textwidth}
        \centering
        \includegraphics[width=0.7\textwidth]{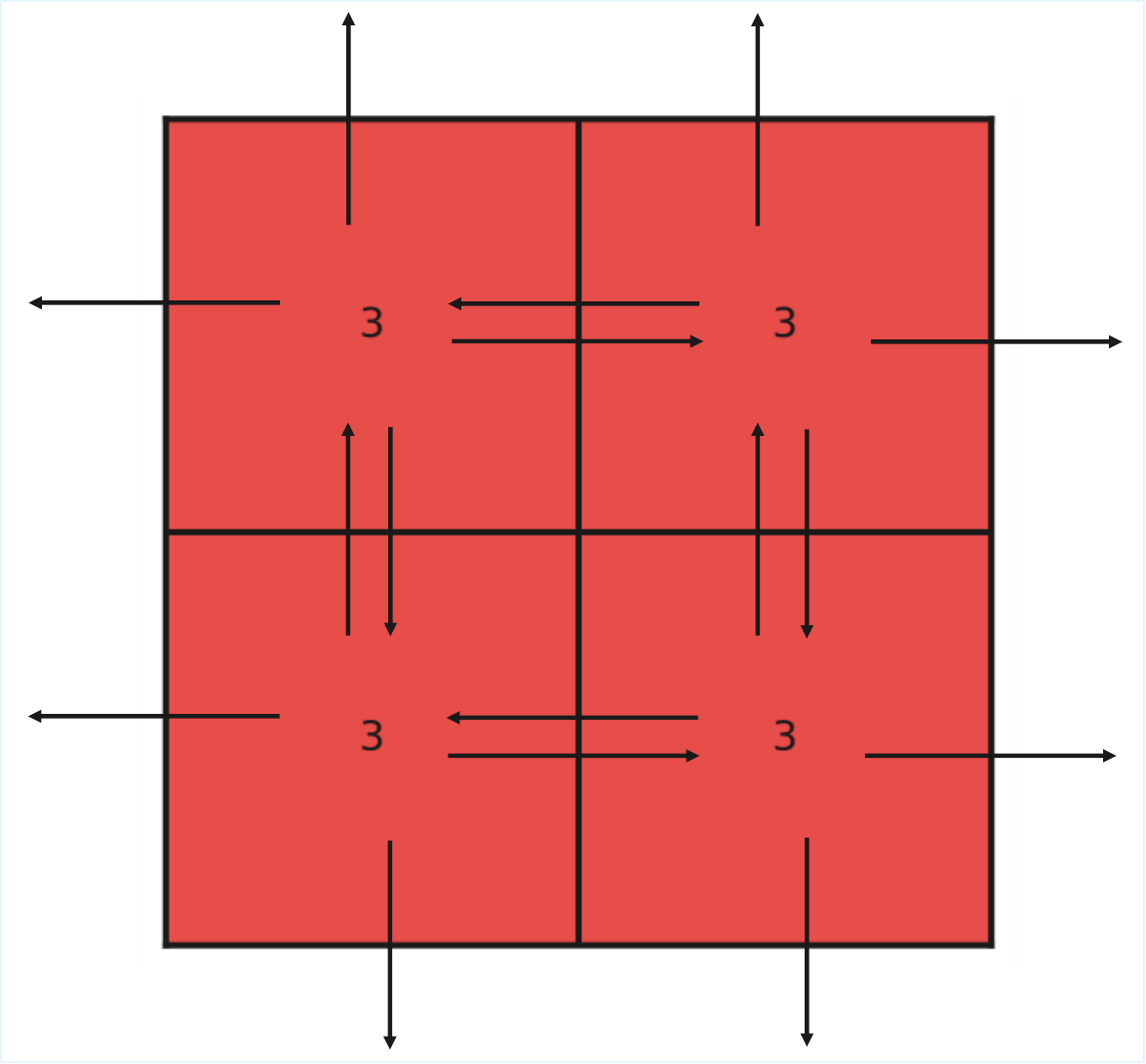}
        \caption[]%
        {{\small }}    
    \end{subfigure}
    \begin{subfigure}[b]{0.45\textwidth}  
        \centering 
        \includegraphics[width=0.7\textwidth]{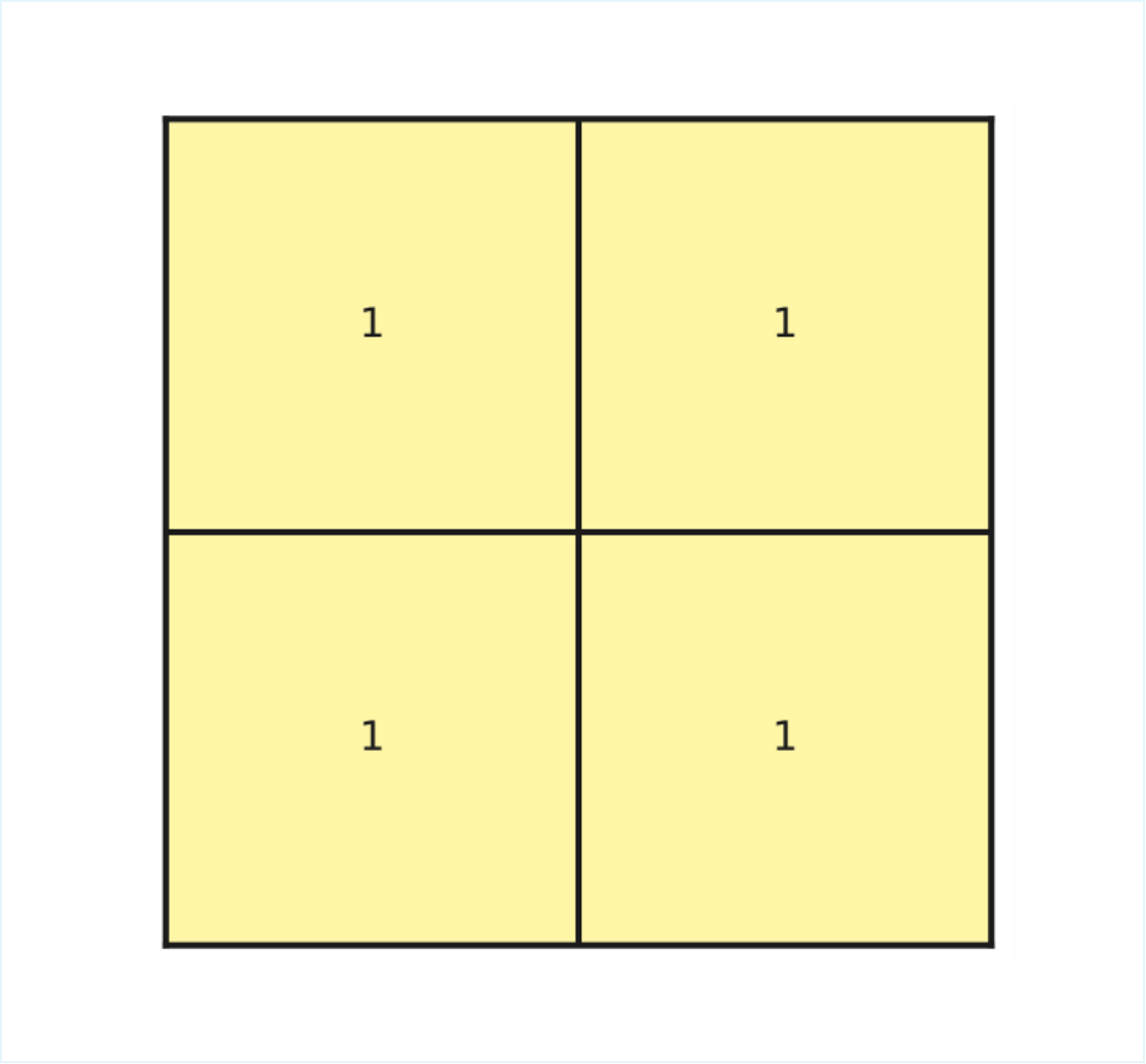}
        \caption[]%
        {{\small}}    
    \end{subfigure}
    \caption{Directed graph constructed in Algorithm 1 applied to a square generator (left) with $N = 2$ and sandpile configuration $\eta^{A^{(2)}}$ (right).} 
    \label{Figs_full_square_small_even}
\end{figure}

    Now, suppose that expression (\ref{exp_av_size_square}) holds for $N = k-2$ for some $k = 3, 4, \ldots$. We show that it continues to hold for $N = k$. We again use Algorithm \ref{Avalanche_analysis} to find the size $w_{A^{(k)}}(\eta)$ of the first wave, the sets $A^{(k)}_1, \ldots, A^{(k)}_{\ell}$ and the sandpile $\eta^{A^{(k)}}$. Again, we initialize a directed graph $(V, E)$ with $E = \emptyset$, add edges to $E$ from $v$ to each of the neighbours (w.r.t. the lattice) of $v$ for each $v \in A^{(k)}$ and set $w' = |A^{(k)}| = k^2$. Note that for each $v \notin A^{(k)} \cup \delta^{\text{out}}(A^{(k)})$, we have $\text{indeg}(v) = \text{outdeg}(v) = 0$ and that for $v \in \delta^{\text{out}}(A^{(k)})$, we have $\text{indeg}(v) = 1$ and $\text{outdeg}(v) = 0$. Furthermore, since $v \in \delta^{\text{out}}(A^{(k)})$ is not part of $A^{(k)}$, we have $\eta(v) < 3$. This implies that $\eta(v) + \text{indeg}(v) < 4$ for all $v \notin A^{(k)}$ and step~4 of the algorithm terminates at this point, resulting in $w_{A^{(k)}}(\eta) = k^2$. Now, note that $\eta^{A^{(k)}}(v) = 3$ if and only if $\text{indeg}(v) = 4$. This implies that $\eta^{A^{(k)}}(v) = 3$ if and only if $v \in \text{int}(A^{(k)})$. Hence, the set $\{v \in A^{(k)}| \eta^{A^{(k)}}(v) = 3\}$ consists of a single connected component $A^{(k)}_1$, which has the form of a square of size $k-2 \times k-2$. The directed graph constructed in Algorithm 1 and the sandpile configuration $\eta^{A^{(k)}}$ are depicted in Figure \ref{Figs_full_square}. 
    
     \begin{figure}[H]
    \centering
    \captionsetup[subfigure]{justification=centering, labelformat=empty, singlelinecheck=false, width=0.6\linewidth}

    \begin{subfigure}[b]{0.45\textwidth}
        \centering
        \includegraphics[width=0.7\textwidth]{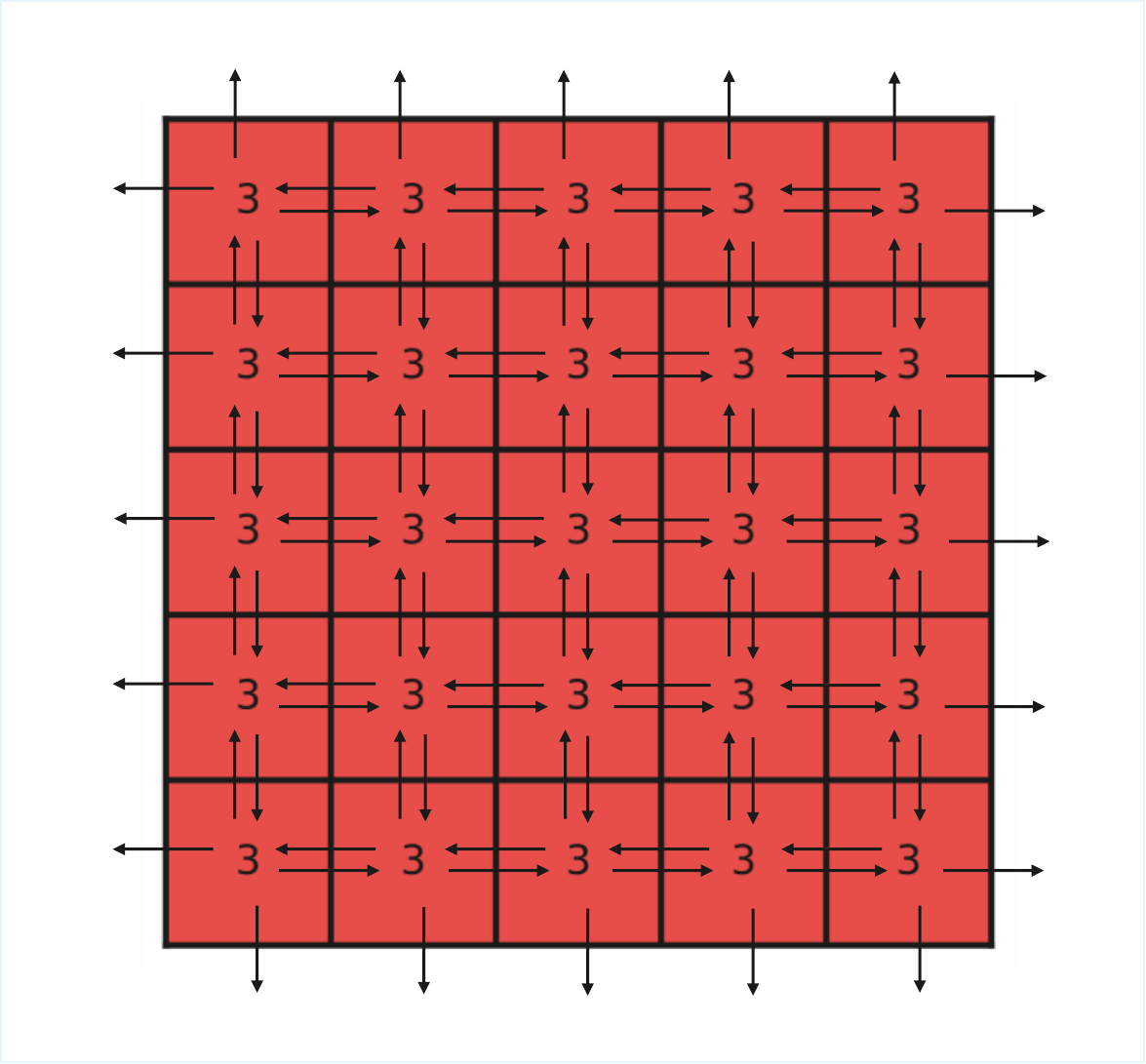}
        \caption[]%
        {{\small }}    
    \end{subfigure}
    \begin{subfigure}[b]{0.45\textwidth}  
        \centering 
        \includegraphics[width=0.7\textwidth]{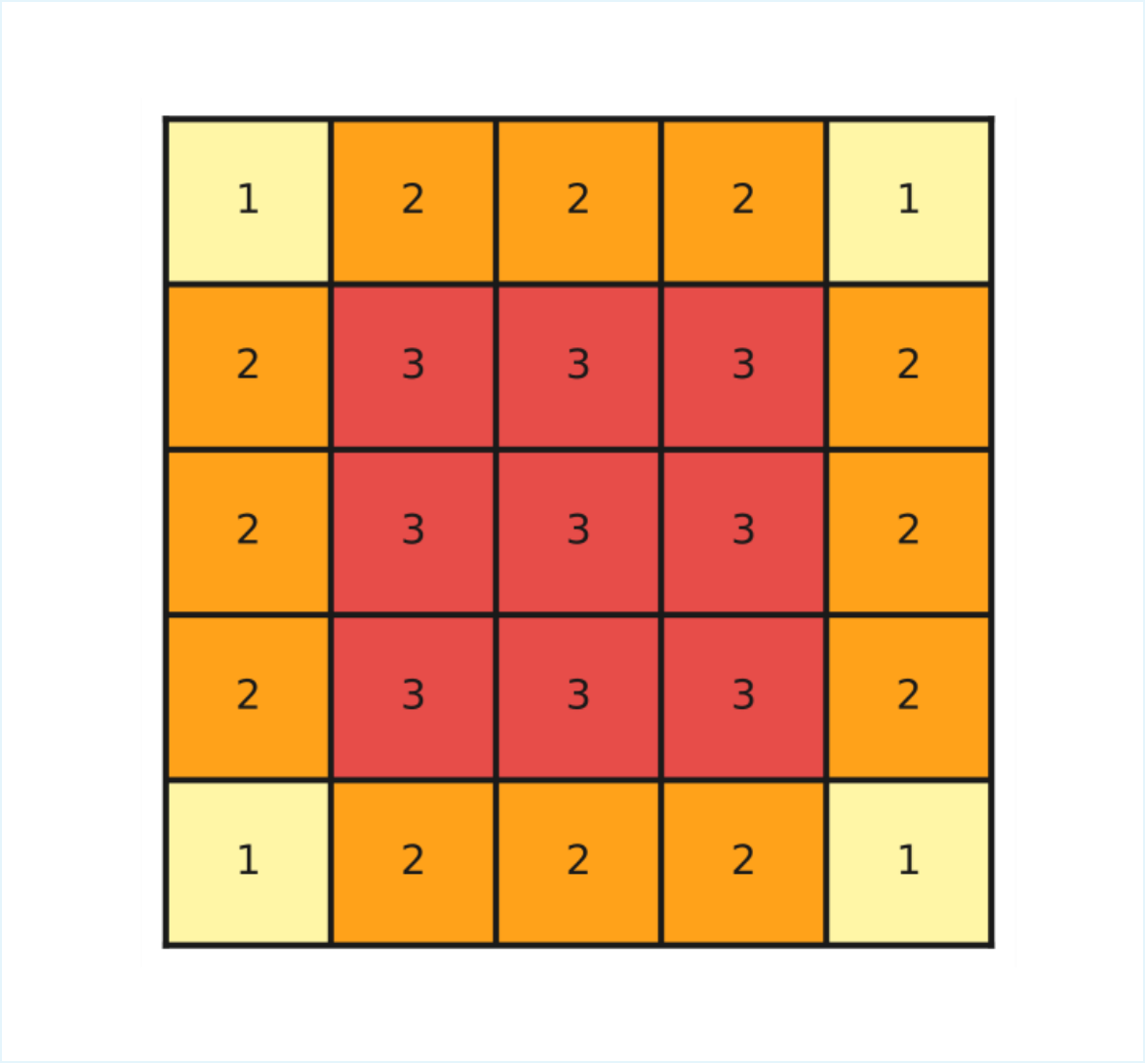}
        \caption[]%
        {{\small}}    
    \end{subfigure}
    \caption{Directed graph constructed in Algorithm 1 applied to a square generator (left) and sandpile configuration $\eta^{A^{(k)}}$ (right).} 
    \label{Figs_full_square}
\end{figure}

    Inserting the obtained size of the first wave $w_{A^{(k)}}(\eta)$ into expression (\ref{exp_av_size_rec_eq}), using the structure of $A_1^{(k)}$ and invoking the induction hypothesis now yields
    \begin{equation*}
        \mathbb{E}[X(\eta)|Y \in A^{(k)}] = k^2 + \dfrac{(k-2)^2}{k^2} \dfrac{3(k-2)^4+15(k-2)^3 + 20(k-2)^2 -8}{30(k-2)} = \dfrac{3k^4+15k^3+20k^2-8}{30k},
    \end{equation*}
    establishing expression (\ref{exp_av_size_square}) for $N = k$. It follows that expression (\ref{exp_av_size_square}) holds for all positive integers $N$. 
\end{proof}

\newpage
\begin{remark}[Local confinement of avalanches and upper bound]\label{remark}
Observe from the arguments above that an avalanche emanating from a square generator is confined to this region, since the surrounding boundary of noncritical vertices acts as an insurmountable barrier. Moreover, given two sandpile configurations $\eta, \eta' \in \mathcal{X}$ with $\eta(v) \geq \eta'(v)$ for all $v \in V$, note that $x(\eta, v) \geq x(\eta', v)$ for all $v \in V$. Consequently, expression (\ref{exp_av_size_square}) provides an upper bound for the expected avalanche size associated with any $N \times N$ square-shaped region enclosed by noncritical vertices. \qed
\end{remark}

\begin{theorem}\label{thm_depth_square}
Consider a generator $A^{(N)}$ in a sandpile $\eta \in \Omega$ that has the form of an $N\times N$ square. The depth of this generator is given by 
\begin{equation}\label{depth_square_expr}
    \rho(\eta, A^{(N)}) = \lceil N/2 \rceil.
\end{equation}
\end{theorem}
\begin{proof}
We prove the statement by induction over $N$, using the definition of depth given in expression (\ref{depth_def}). The result is obvious for $N = 1$. Consider $N = 2$. Recall from the proof of Theorem \ref{thm_exp_av_size_square} that $\{v \in A^{(2)}|\eta^{A^{(2)}}(v) = 3\} = \emptyset$ (see Figure \ref{Figs_full_square_small_even}). Hence, it follows from expression (\ref{depth_def}) that $\rho(\eta, A^{(2)}) = 1$, which is consistent with expression (\ref{depth_square_expr}). We proceed to assume that expression (\ref{depth_def}) is valid for $N = k-2$ for some $k = 3, 4, \ldots$ and show that this implies its correctness for $N = k$. It follows from the proof of Theorem \ref{thm_exp_av_size_square} that the set $\{v \in A^{(k)}|\eta^{A(k)}(v) = 3\}$ is a square connected component of size $k-2 \times k-2$ (see Figure \ref{Figs_full_square}). Hence, by expression (\ref{depth_def}) and the induction hypothesis, we obtain
\begin{equation}
    \rho(\eta, A^{(k)}) = 1 + \rho(\eta, A^{(k-2)}) = 1 + \lceil (k-2)/2 \rceil = \lceil k/2 \rceil,
\end{equation}
which establishes expression (\ref{depth_square_expr}) for $N = k$. 
\end{proof}

We proceed to characterize the set of cornerstone vertices for square-shaped generators. To this end, we first analyze the effect of removing the sand grains from critical vertices at different locations within such a generator. The results are collected in Lemmas \ref{lem_cornerstones_center} -- \ref{lem_cornerstones_C}. We include the proof of Lemma \ref{lem_cornerstones_center}. The proofs of the remaining lemmas follow similar lines and are deferred to the Appendix. 

\begin{lemma}\label{lem_cornerstones_center}
Consider a generator $A^{(N)}$ in a sandpile $\eta \in \Omega$ that has the form of an $N \times N$ square, where $N \geq 2$, and let $v_i \in R_1(A^{(N)})$. Then,
\begin{equation}\label{Exp_av_size_center}
\mathbb{E}[X(\gamma_i \eta)|Y \in A^{(N)}] = \begin{cases}
    \dfrac{(N^2-1)(N^2+5N+6)}{10N}, &\text{if } N \text{ is odd}, \\[1.2em]
    \dfrac{N^5+5N^4+5N^3-5N^2-6N-30}{10N^2}, &\text{if } N \text{ is even.}
    \end{cases}
\end{equation}
\end{lemma}

\begin{proof}
    Let $\tilde{A}_R^{N,1}$ denote the remainder of the generator $A^{(N)}$ in the sandpile configuration $\gamma_i \eta$, i.e., $\tilde{A}_R^{N,1} = A^{(N)} \setminus \{v_i\}$. Note that conditioning on the vertex $Y$ that the next sand grain lands upon yields
    \begin{equation*}
        \mathbb{E}[X(\gamma_i \eta)|Y \in A^{(N)}] = \dfrac{N^2-1}{N^2} \mathbb{E}[X(\gamma_i \eta)|Y \in \tilde{A}^{N,1}_R].
    \end{equation*}
    Hence, it suffices to prove that
    \begin{equation}\label{Exp_av_size_center_2}
\mathbb{E}[X(\gamma_i \eta)|Y \in \tilde{A}^{N,1}_R] = \begin{cases}
    \dfrac{N(N^2+5N+6)}{10}, &\text{if } N \text{ is odd}, \\[1.2em]
    \dfrac{N^5+5N^4+5N^3-5N^2-6N-30}{10(N^2-1)}, &\text{if } N \text{ is even.}
    \end{cases}
\end{equation}
    We prove the validity of expression (\ref{Exp_av_size_center_2}) by applying Algorithm \ref{Avalanche_analysis} to the generator $\tilde{A}^{N,1}_R$ in the configuration $\gamma_i \eta$ and performing induction over $N$. First, we show that expression (\ref{Exp_av_size_center_2}) holds for $N = 2$ and $N = 3$. Consider the case that $N = 2$. We initialize a directed graph $(V, E)$ with $E = \emptyset$ and augment $E$ with edges from $v$ to each of its neighbours for each $v \in \tilde{A}^{2,1}_R$. Now, observe that $(\gamma_i \eta)(v) + \text{indeg}(v) - \text{outdeg}(v) < 4$ for all $v \notin \tilde{A}_R^{2, 1}$. Thus, we obtain $w_{\tilde{A}_R^{2,1}}(\gamma_i \eta) = 3$. The directed graph and the sandpile configuration $(\gamma_i \eta)^{\tilde{A}_R^{2,1}}$ are shown in Figure \ref{Figs_center_small_even}. 

   \begin{figure}[H]
    \centering
    \captionsetup[subfigure]{justification=centering, labelformat=empty, singlelinecheck=false, width=0.6\linewidth}

    \begin{subfigure}[b]{0.45\textwidth}
        \centering
        \includegraphics[width=0.7\textwidth]{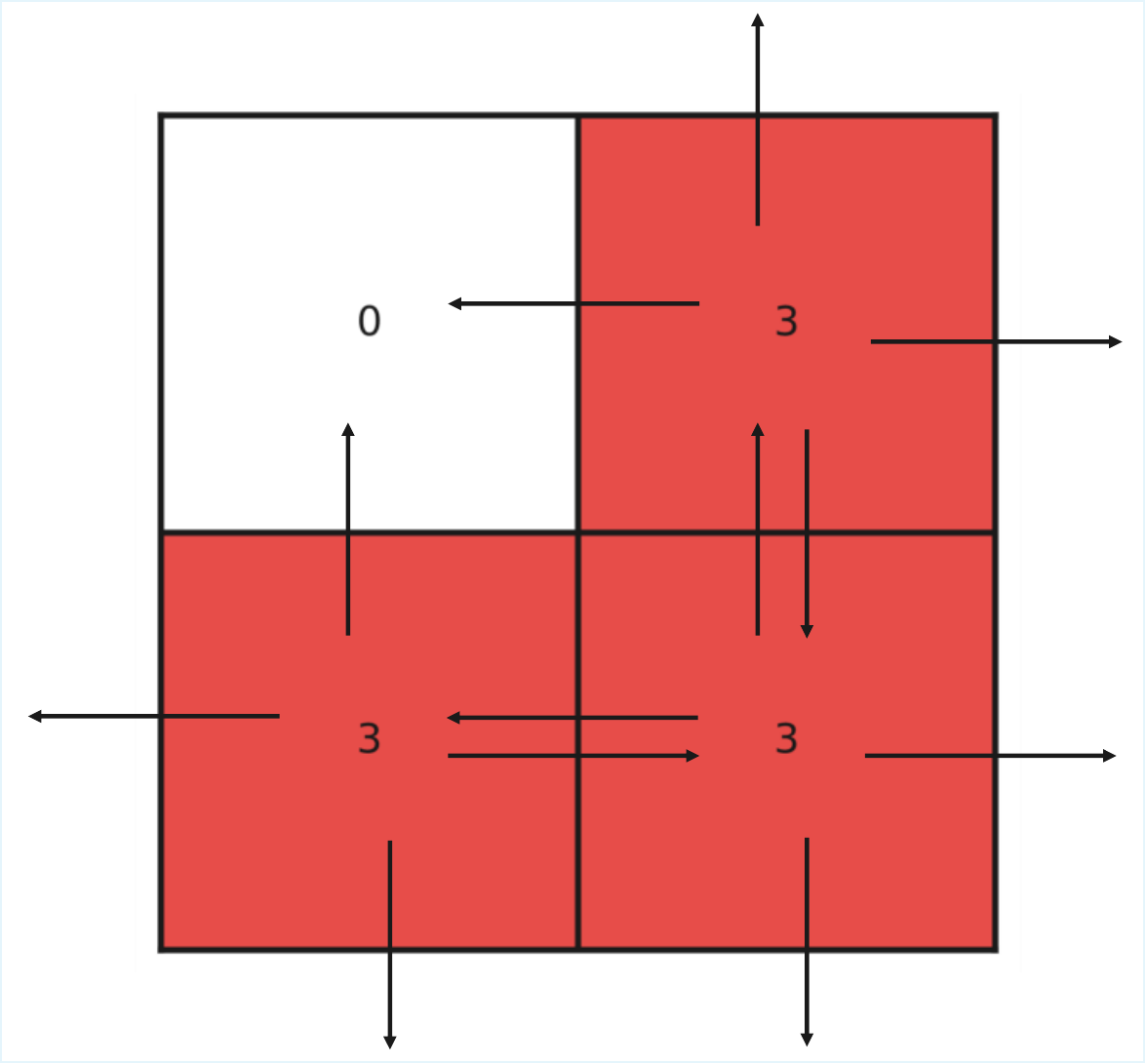}
        \caption[]%
        {{\small }}    
    \end{subfigure}
    \begin{subfigure}[b]{0.45\textwidth}  
        \centering 
        \includegraphics[width=0.7\textwidth]{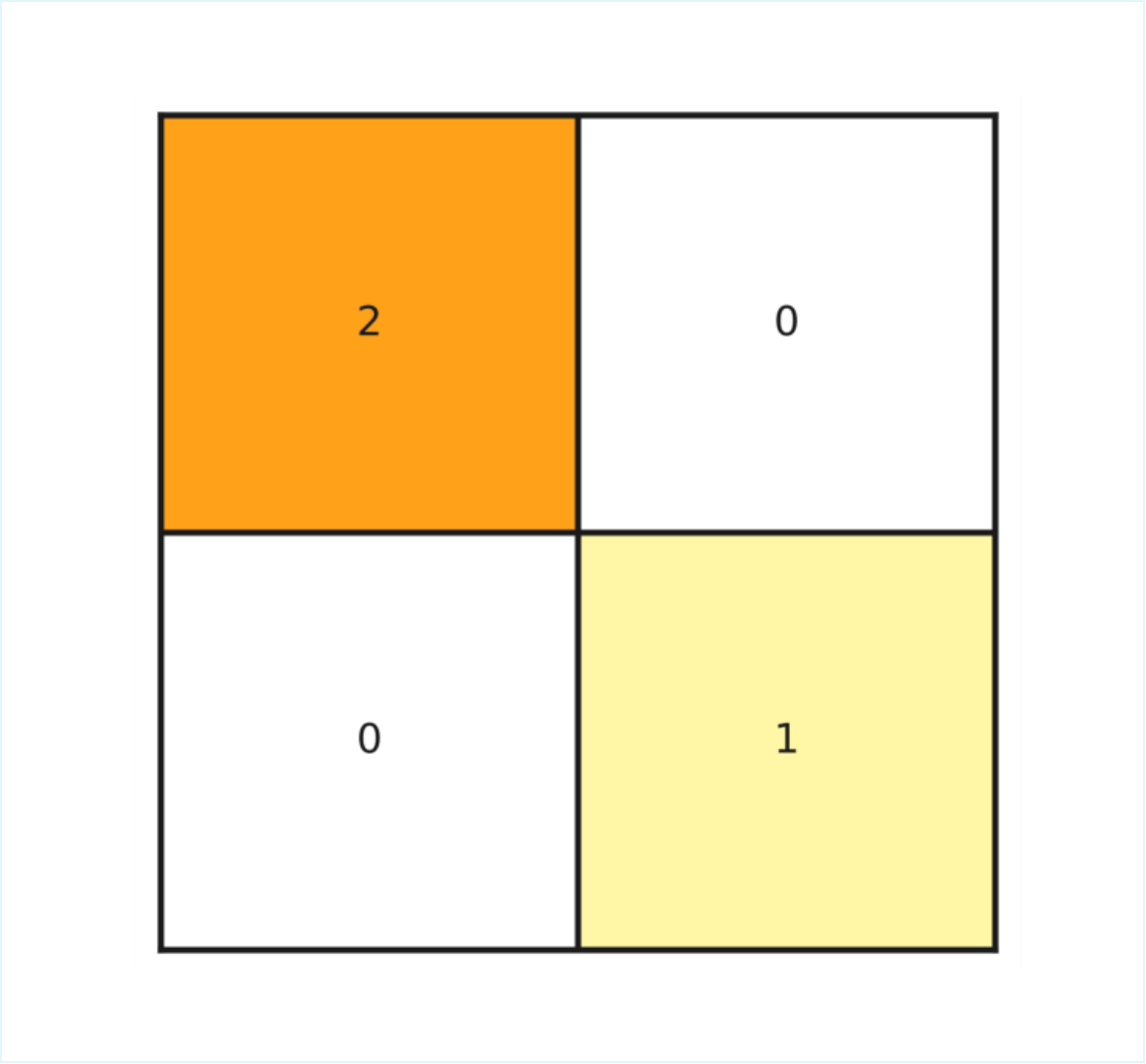}
        \caption[]%
        {{\small}}    
    \end{subfigure}
    \caption{Directed graph constructed in Algorithm 1 applied to generator $\tilde{A}_R^{2,1}$ in configuration $\gamma_i \eta$ (left) and sandpile configuration $(\gamma_i \eta)^{\tilde{A}_R^{2,1}}$ (right).} 
    \label{Figs_center_small_even}
\end{figure}
    It follows that
    \begin{equation*}
        \mathbb{E}[X(\gamma_i \eta)|Y \in \tilde{A}_R^{2,1}] = 3,
    \end{equation*}
    which agrees with expression (\ref{Exp_av_size_center_2}).

    Now, consider the case $N = 3$. We start by initializing the directed graph $(V, E)$ with $E = \emptyset$ and add edges from $v$ to each of its neighbours for each $v \in \tilde{A}_R^{3,1}$. At this point, the only vertex $v \notin \tilde{A}_R^{3,1}$ that satisfies $(\gamma_i \eta)(v) + \text{indeg}(v) - \text{outdeg}(v) \geq 4$ is the vertex $v_i$. Thus, we add edges to $E$ from $v_i$ to each of its neighbours. Since each neighbour of $v_i \in R_1(A^{(N)})$ is part of $\tilde{A}_R^{3,1}$, step 4 of Algorithm \ref{Avalanche_analysis} terminates at this point. We obtain $w_{\tilde{A}_R^{3,1}}(\gamma_i \eta) = 9$. Observe that for each $v \in \delta^{\text{in}}(A^{(N)})$, we have $\text{indeg}(v) = 3$. Hence, $(\gamma_i\eta)^{\tilde{A}_R^{3,1}}(v) = 2$. Also, for each $v \in C(A^{(N)})$, we have $\text{indeg}(v) = 2$. Therefore, $(\gamma_i\eta)^{\tilde{A}_R^{3,1}}(v) = 1$. Finally, we have $\text{indeg}(v_i) = \text{outdeg}(v_i) = 4$, and thus, $\eta^{\tilde{A}_R^{3,1}}(v_i) = 0$. The directed graph constructed in the algorithm and the sandpile configuration $(\gamma_i\eta)^{\tilde{A}_R^{3,1}}$ after the first wave are depicted in Figure \ref{Figs_center_small}. It follows that applying Algorithm \ref{Avalanche_analysis} in the case $N = 3$ yields
    \begin{equation*}
        \mathbb{E}[X(\gamma_i \eta)|Y \in \tilde{A}^{3,1}_R] = 9,
    \end{equation*}
    which is consistent with expression (\ref{Exp_av_size_center_2}). 

       \begin{figure}[H]
    \centering
    \captionsetup[subfigure]{justification=centering, labelformat=empty, singlelinecheck=false, width=0.6\linewidth}

    \begin{subfigure}[b]{0.45\textwidth}
        \centering
        \includegraphics[width=0.7\textwidth]{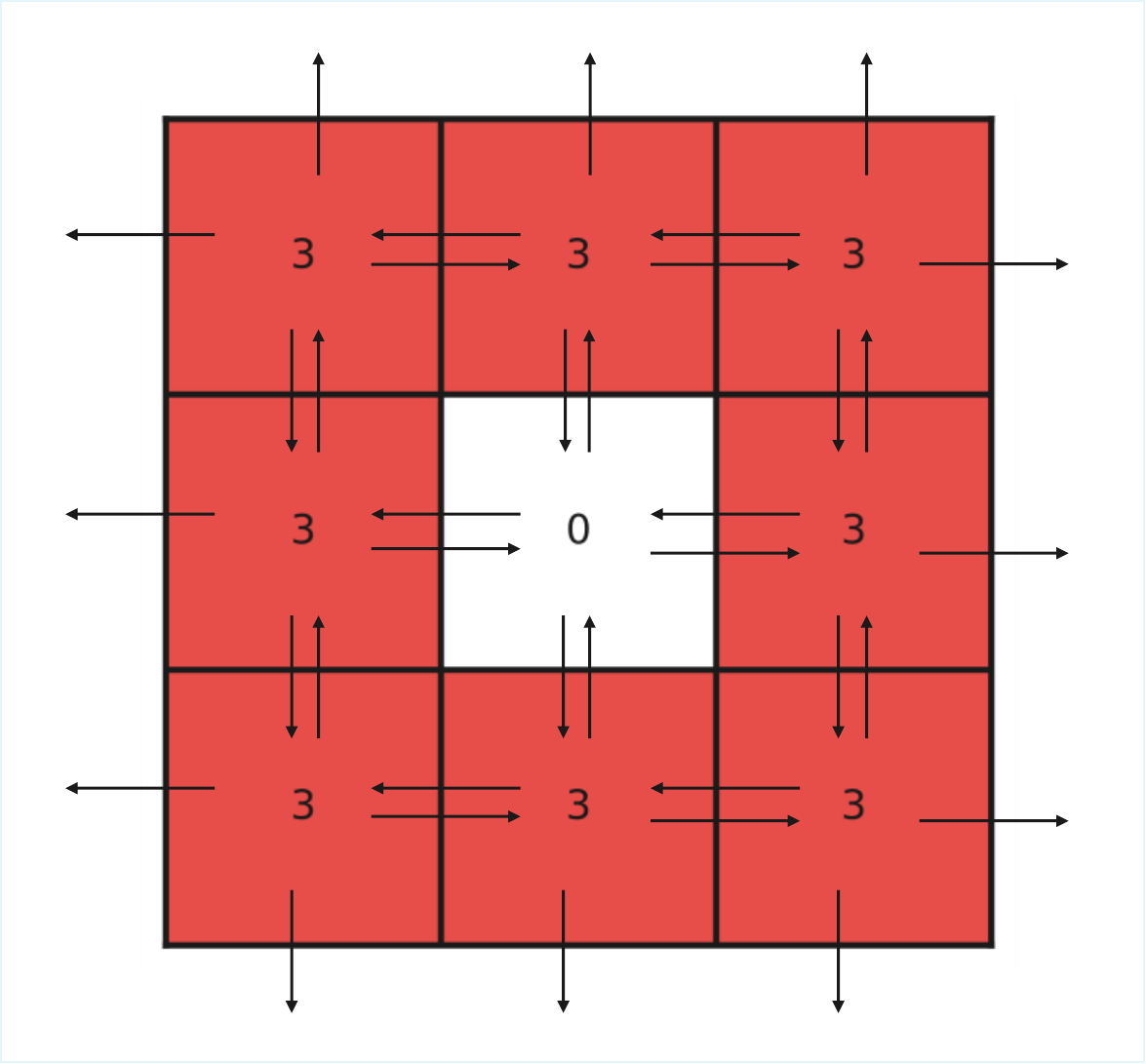}
        \caption[]%
        {{\small }}    
    \end{subfigure}
    \begin{subfigure}[b]{0.45\textwidth}  
        \centering 
        \includegraphics[width=0.7\textwidth]{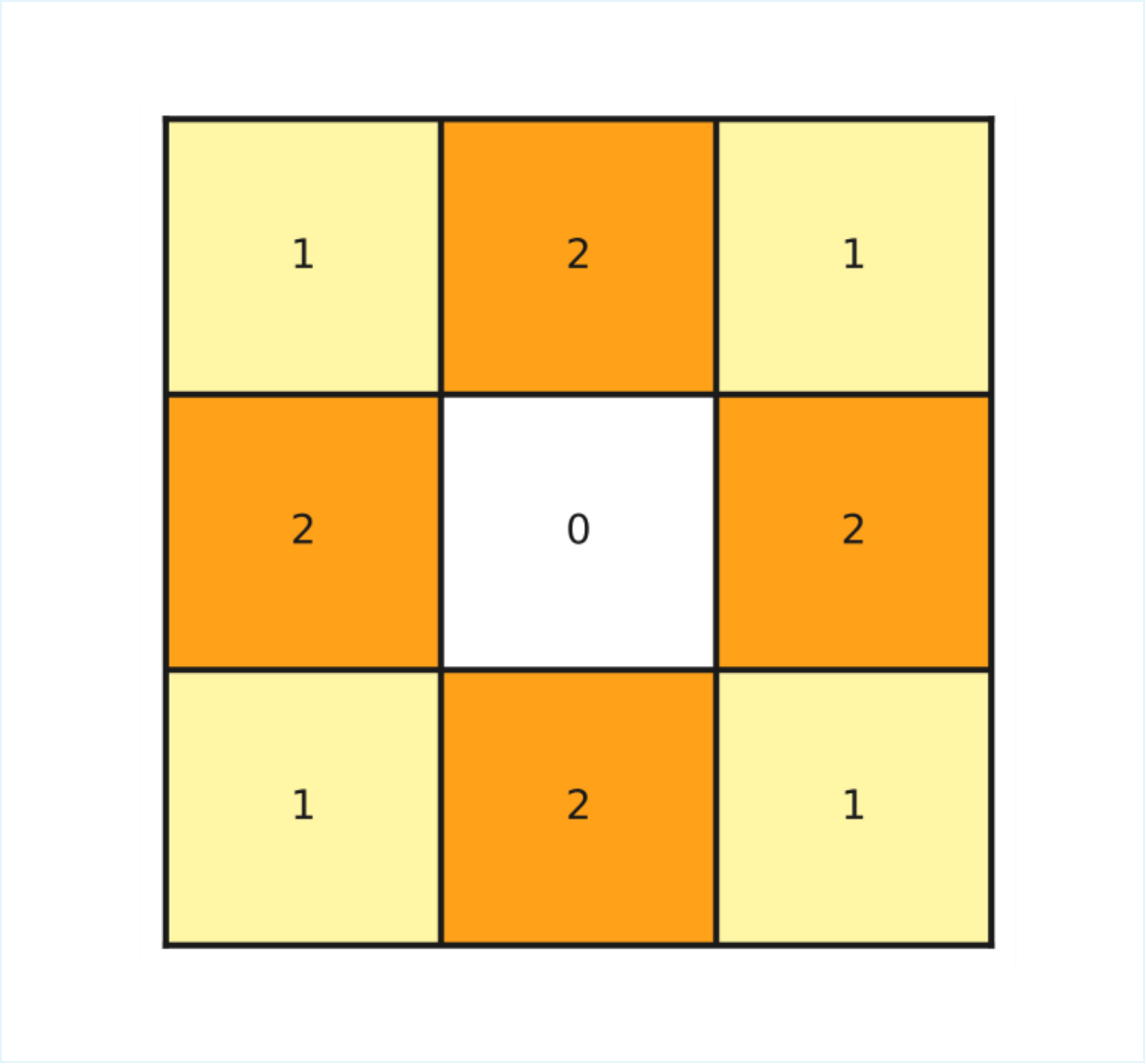}
        \caption[]%
        {{\small}}    
    \end{subfigure}
    \caption{Directed graph constructed in Algorithm 1 applied to generator $\tilde{A}_R^{3,1}$ in configuration $\gamma_i \eta$ (left) and sandpile configuration $(\gamma_i\eta)^{\tilde{A}_R^{3,1}}$ (right).} 
    \label{Figs_center_small}
\end{figure}
    Now, we assume that expression (\ref{Exp_av_size_center_2}) holds for $N = k-2$ for some $k = 3, 4, \ldots$. Consider the generator $\tilde{A}_R^{k,1}$ in the configuration $\gamma_i \eta$. The directed graph and the sandpile $(\gamma_i\eta)^{\tilde{A}_R^{k,1}}$ that remains after the first wave are provided in Figure~\ref{Figs_center} for $k = 5$. The algorithm evolves in a similar way for even $k$, although in this case, there are four options for the vertex $v_i$.

       \begin{figure}[H]
    \centering
    \captionsetup[subfigure]{justification=centering, labelformat=empty, singlelinecheck=false, width=0.6\linewidth}

    \begin{subfigure}[b]{0.45\textwidth}
        \centering
        \includegraphics[width=0.7\textwidth]{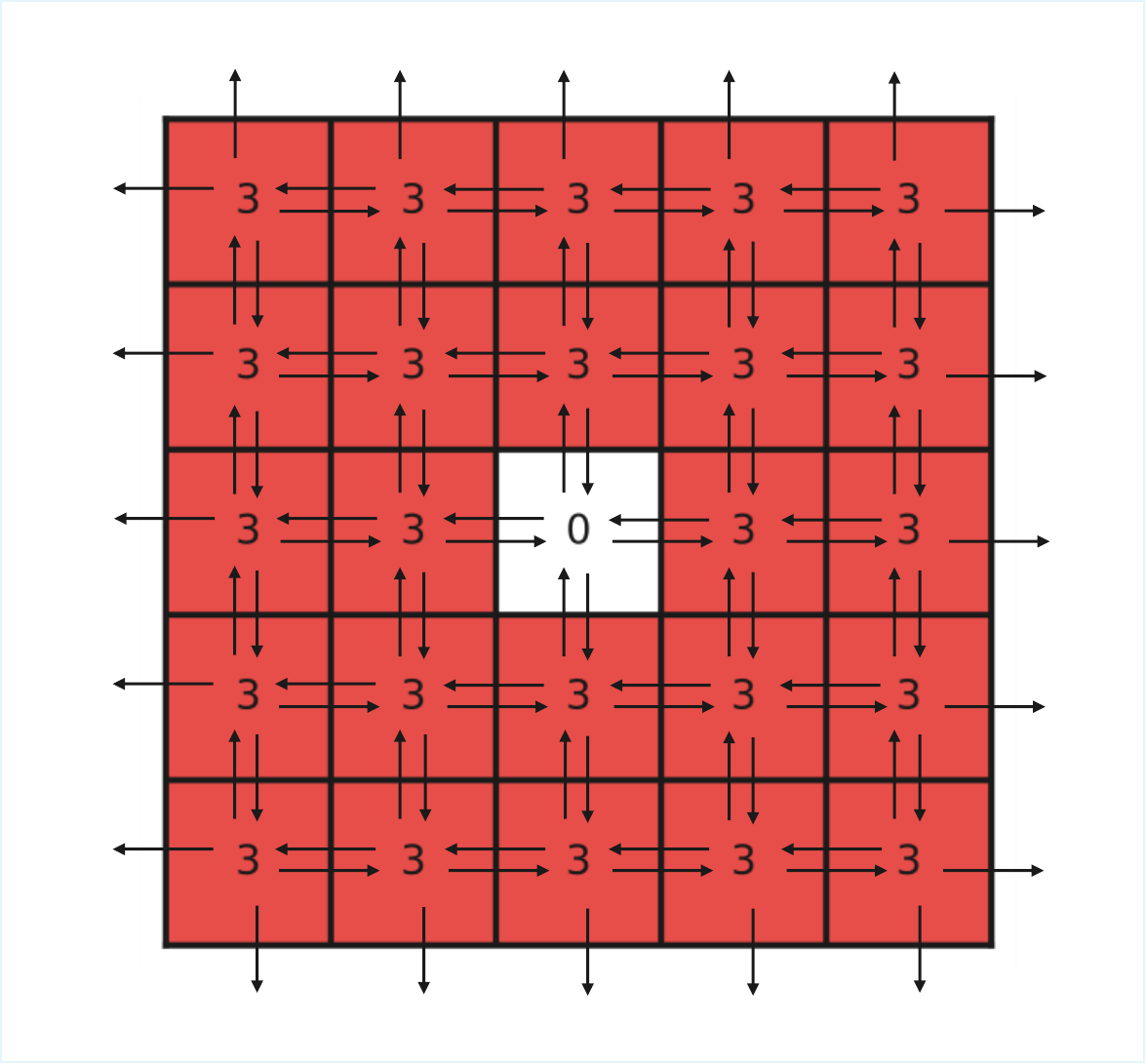}
        \caption[]%
        {{\small }}    
    \end{subfigure}
    \begin{subfigure}[b]{0.45\textwidth}  
        \centering 
        \includegraphics[width=0.7\textwidth]{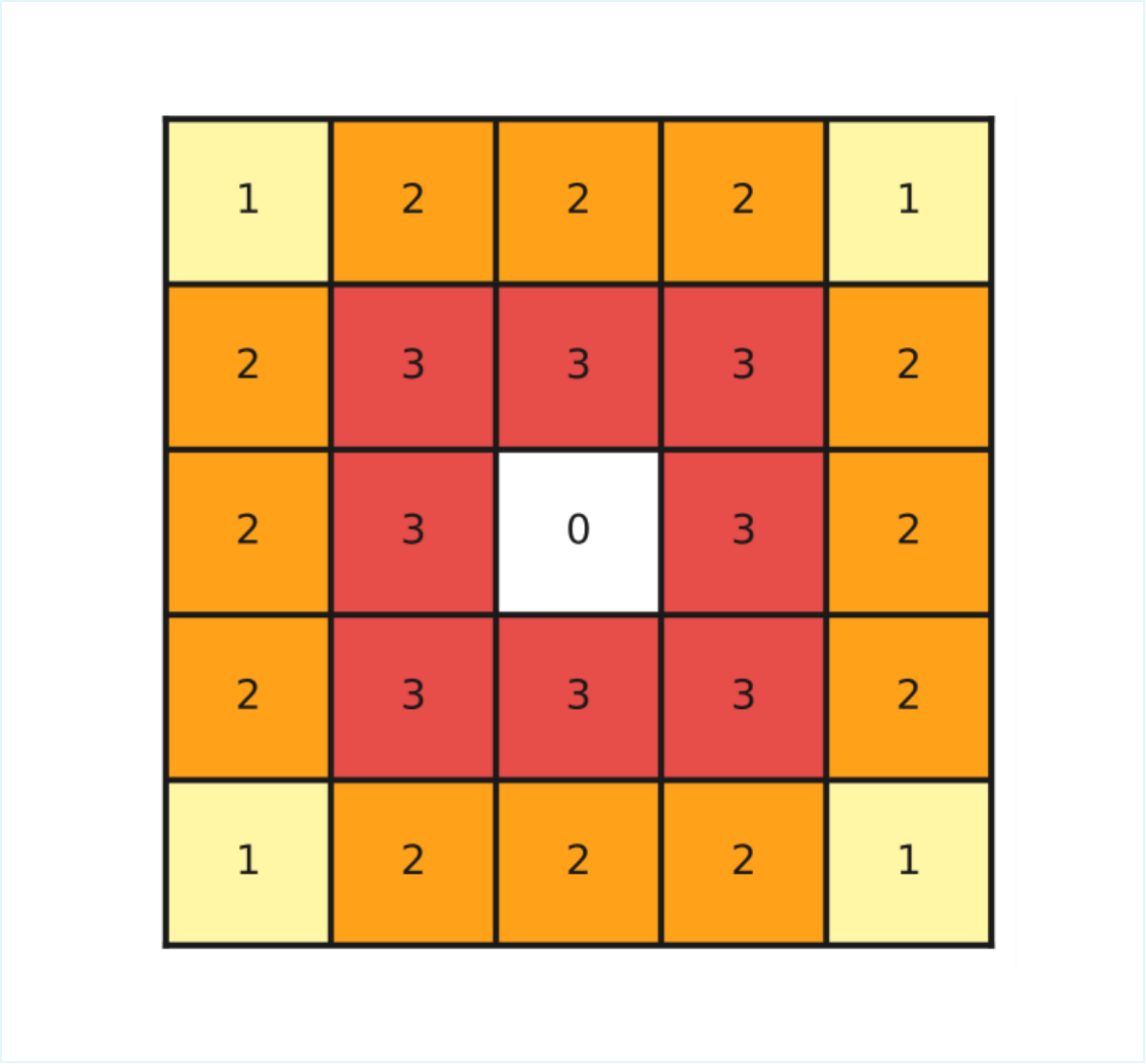}
        \caption[]%
        {{\small}}    
    \end{subfigure}
    \caption{Directed graph constructed in Algorithm 1 applied to generator $\tilde{A}_R^{k,1}$ in configuration $\gamma_i \eta$ (left) and sandpile configuration $(\gamma_i\eta)^{\tilde{A}_R^{k,1}}$ (right) for $k = 5$.} 
    \label{Figs_center}
\end{figure}

    It follows that Algorithm \ref{Avalanche_analysis} yields $w_{\tilde{A}_R^{k,1}}(\gamma_i \eta) = k^2$ and that the remainder of the generator $\tilde{A}_R^{k,1}$ after the first wave is a generator $\tilde{A}_R^{k-2,1}$ in the sandpile configuration $(\gamma_i \eta)^{\tilde{A}_R^{k,1}}$. Inserting this in expression (\ref{exp_av_size_rec_eq}) and using the induction hypothesis now yields
    \begin{equation*}
        \mathbb{E}[X(\gamma_i \eta)|Y \in \tilde{A}_R^{k,1}] = k^2 + \dfrac{(k-2)^2-1}{k^2-1} \dfrac{(k-2)((k-2)^2+5(k-2)+6)}{10} = \dfrac{k(k^2 + 5k+6)}{10}
    \end{equation*}
    for odd $k$ and 
    \begin{align*}
        \mathbb{E}[X(\gamma_i \eta)|Y \in \tilde{A}^{k,1}_R] &= k^2 + \dfrac{(k-2)^2-1}{k^2-1} \dfrac{(k-2)^5 + 5(k-2)^4 + 5(k-2)^3-5(k-2)^2-6(k-2)-30}{10((k-2)^2-1)}\\
        &= \dfrac{k^5 + 5k^4 + 5k^3 - 5k^2-6k -30}{10(k^2-1)}
    \end{align*}
    for even $k$. This establishes the validity of expression (\ref{Exp_av_size_center_2}) for all positive integers $N$.  
\end{proof}

\begin{lemma}\label{lem_cornerstones_R}
Consider a generator $A^{(N)}$ in a sandpile $\eta \in \Omega$ that has the form of an $N \times N$ square, where $N \geq 3$, and let $v_i \in R_k(A^{(N)}) \setminus C\big(\bigcup_{j=1}^k R_j(A^{(N)}) \big)$, $k = 2, \ldots, \lceil N/2 \rceil$.  Then,
\begin{equation}\label{Exp_av_size_R}
    \mathbb{E}[X(\gamma_i \eta)|Y \in A^{(N)}] = \begin{cases}
        \dfrac{3N^5+15N^4 + 15N^3-15N^2-18N+10(4k^3  - 24k^2 + 23k -6)}{30N^2}, &\text{if } N \text{ is odd,} \\[1.2em]
        \dfrac{3N^5 + 15N^4 + 15N^3 - 15N^2 - 18N + 20k(2k^2-9k+1)}{30N^2}, &\text{if } N \text{ is even.}
    \end{cases}
\end{equation}
    
\end{lemma}
\begin{proof}
    See Appendix \ref{Appendix}.
\end{proof}

\vspace{5mm}

\begin{lemma}\label{lem_cornerstones_C}
    Consider a generator $A^{(N)}$ in a sandpile $\eta \in \Omega$ that has the form of an $N \times N$ square, where $N \geq 3$, and let $v_i \in C\left(\bigcup_{j=1}^k R_j(A^{(N)}) \right), \quad k = 2, \ldots, \lceil N/2 \rceil$.  Then,
    \begin{equation}\label{Exp_av_size_C}
        \mathbb{E}[X(\gamma_i \eta)|Y \in A^{(N)}] = \begin{cases}
            \dfrac{3N^5 + 15N^4 + 15N^3-15N^2 - 18N+10(4k^3 - 24k^2+23k-3)}{30N^2}, &\text{if } N \text{ is odd,} \\[1.2em]
            \dfrac{3N^5 + 15N^4 + 15N^3 - 15N^2 - 18N + 10(4k^3 - 18k^2 + 2k + 3)}{30N^2}, &\text{if } N \text{ is even.}
        \end{cases}
    \end{equation}
\end{lemma}
\begin{proof}
    See Appendix \ref{Appendix}.
\end{proof}

Theorem \ref{thm_cornerstones_square_GOA} now provides a characterization of the set of cornerstone vertices for square-shaped generators, i.e., the vertices for which the minimum in expression (\ref{stability_level}) is attained.

\begin{theorem}\label{thm_cornerstones_square_GOA}
Consider a generator $A^{(N)}$ in a sandpile $\eta\in \Omega$ that has the form of an $N \times N$ square. The set of cornerstone vertices corresponding to this generator is given by
\begin{equation*}
    B^*_{A^{(N)}}(\eta) = \begin{cases}
        R_1(A^{(N)}), &\text{if } N = 1, 2, \\
        R_2(A^{(N)}) \setminus C(A^{(N)}), &\text{if } N = 3, 4, \\
        R_3(A^{(N)}) \setminus C(\bigcup_{i=1}^3 R_i(A^{(N)})), &\text{if } N \geq 5.
        \end{cases}
\end{equation*}
\end{theorem}
\begin{proof}
    The statement is obvious for $N = 1,2$. Now, consider $N = 3$. Let $v_i \in R_2(A^{(3)})\setminus C(A^{(3)})$, $v_j \in C(A^{(3)})$ and $v_k \in R_1(A^{(3)})$. Comparing expressions (\ref{Exp_av_size_center}), (\ref{Exp_av_size_R}) and (\ref{Exp_av_size_C}), we obtain
    \begin{equation*}
        \mathbb{E}[X(\gamma_i \eta)|Y \in A^{(3)}] < \mathbb{E}[X(\gamma_j \eta)|Y \in A^{(3)}] < \mathbb{E}[X(\gamma_k \eta))|Y \in A^{(3)}].
    \end{equation*}
    Hence, we have $B^*_{A^{(3)}}(\eta) = R_2(A^{(3)}) \setminus C(A^{(3)})$. 

    We proceed to consider $N = 4$. Let $v_i \in R_2(A^{(4)})\setminus C(A^{(4)})$, $v_j \in C(A^{(4)})$ and $v_k \in R_1(A^{(4)})$. Again comparing expressions (\ref{Exp_av_size_center}), (\ref{Exp_av_size_R}) and (\ref{Exp_av_size_C}) yields
    \begin{equation*}
         \mathbb{E}[X(\gamma_i \eta)|Y \in A^{(4)}] < \mathbb{E}[X(\gamma_j \eta)|Y \in A^{(4)}] < \mathbb{E}[X(\gamma_k \eta))|Y \in A^{(4)}].
    \end{equation*}
    It follows that $B^*_{A^{(4)}}(\eta) = R_2(A^{(4)}) \setminus C(A^{(4)})$. 

    Now, consider $N \geq 5$. Let $v_{i_k} \in R_k(A^{(N)}) \setminus C(\bigcup_{j=1}^k R_j(A^{(N)}))$, $k = 2, \ldots, \lceil N/2 \rceil$. Minimizing expression (\ref{Exp_av_size_R}) with respect to $k$, we obtain
    \begin{equation*}
        \argmin_{k = 2, \ldots, \lceil N/2 \rceil} \mathbb{E}[X(\gamma_{i_k} \eta)|Y \in A^{(N)}] = 3
    \end{equation*}
    for both odd and even $N$. 
    
    Now, let $v_{i_k} \in C\left(\bigcup_{j=1}^k R_j(A^{(N)})\right)$, $k = 2, \ldots, \lceil N/2 \rceil$. Minimizing expression (\ref{Exp_av_size_C}) with respect to $k$ yields
    \begin{equation*}
        \argmin_{k = 2, \ldots, \lceil N/2 \rceil} \mathbb{E}[X(\gamma_{i_k} \eta)|Y \in A^{(N)}] = 3
    \end{equation*}
    for both odd and even $N$. 

    Consider $v_i \in R_3(A^{(N)}) \setminus C\left(\bigcup_{j=1}^3 R_j(A^{(N)})\right)$, $v_j \in C\left(\bigcup_{j=1}^3 R_j(A^{(N)})\right)$ and $v_k \in R_1(A^{(N)})$. 
    Comparing expressions (\ref{Exp_av_size_center}) and expressions (\ref{Exp_av_size_R}) and (\ref{Exp_av_size_C}) for $k = 3$ now yields
    \begin{equation*}
        \mathbb{E}[X(\gamma_i \eta)|Y \in A^{(N)}] < \mathbb{E}[X(\gamma_j \eta)|Y \in A^{(N)}] < \mathbb{E}[X(\gamma_k \eta)|Y \in A^{(N)}],
    \end{equation*}
    which completes the proof.
\end{proof}
A surprising aspect of the result in Theorem \ref{thm_cornerstones_square_GOA} is that the location of the set of cornerstone vertices does not scale with the size of the square. Observe that the result strikes a balance between the impact of the intervention on the largest possible avalanches and the amount of avalanches that are affected by the intervention. After all, emptying a vertex near the center of the square may prevent a large avalanche, but it has no effect on the size of avalanches triggered when the new sand grain lands farther from the center. On the other hand, emptying a more peripheral vertex mitigates many of the potential avalanches, but the reduction in their size is smaller.

\begin{corollary}
    Consider a generator $A^{(N)}$ in a sandpile $\eta \in \Omega$ that has the form of an $N \times N$ square. The stability level of this generator is given by
    \begin{equation*}
        \lambda(\eta, A^{(N)}) = \begin{cases}
            0, &\text{if } N = 1, \\[1.2em]
            9/16, &\text{if } N = 2, \\[1.2em]
            32/41, &\text{if } N = 3, \\[1.2em]
            15/17, &\text{if } N = 4, \\[1.2em]
            \dfrac{3N^5+15N^4+15N^3-15N^2-18N-450}{N(3N^4+15N^3+20N^2-8)}, &\text{if } N \geq 5, \quad N \text{ odd,}\\[1.2em]
            \dfrac{3N^5+15N^4+15N^3-15N^2-18N-480}{N(3N^4+15N^3+20N^2-8)}, &\text{if } N \geq 5, \quad N \text{ even.}
        \end{cases}
    \end{equation*}
\end{corollary}
\begin{proof}
    The case $N = 1$ is evident from expression (\ref{stability_level}). The remaining cases follow immediately from Theorem \ref{thm_exp_av_size_square}, Lemmas \ref{lem_cornerstones_center} -- \ref{lem_cornerstones_C} and Theorem \ref{thm_cornerstones_square_GOA}.
\end{proof}

\section{Conclusion}
\label{Conclusion}
As a paradigmatic model of self-organized criticality, the Abelian sandpile model provides a theoretical framework for large-scale cascading phenomena, including forest fires, earthquakes or financial market crashes. Since large avalanches in such settings correspond to catastrophic events, understanding how to mitigate them is of considerable practical interest. This work took a first step towards a theoretical understanding of the impact of interventions in the sandpile model. Specifically, we investigated the effect of removing sand grains from connected components of critical vertices. To study this effect, we first extended the method proposed by \citet{Dorso} for computing the expected size of an avalanche emanating from a given generator, and we provided a formal justification for the resulting scheme. Using this algorithm, we performed a detailed analysis of the impact of removing sand grains at different locations in square-shaped generators. This class of generators admits an explicit local analysis and a characterization of avalanche structures. Our results reveal a class of optimal target vertices, which we refer to as cornerstone vertices, that balance the tradeoff between reducing the maximal avalanche size and increasing the number of avalanches that are mitigated when a new sand grain is added to the square. Interestingly, the locations of these optimal targets do not scale with the size of the square. 

This work opens several promising directions for future research. Our analysis of square-shaped generators provides insight into the impact of interventions on avalanche structure and highlights the tradeoff between restricting maximal avalanche sizes and increasing the number of mitigated avalanches. It would be interesting to extend this analysis to generators drawn from the stationary distribution of the sandpile Markov chain. Such generators may initiate avalanches that cover a large part of the lattice, precluding a purely local analysis and potentially leading to fundamentally different optimal intervention strategies. 

Furthermore, rather than focusing solely on the immediate effect of interventions, it would be valuable to investigate long-term mitigation strategies, for example within a Markov decision process framework. Such an approach has recently been explored for developing control strategies in the Ising model \citep{deJongh1, deJongh2}. In addition, it may be worthwhile to explore alternative intervention mechanisms beyond sand grain removal, possibly inspired by specific applications. Examples include artificially triggering small avalanches to release stress and reduce the likelihood of larger events, or increasing the toppling thresholds of selected vertices.

\bibliographystyle{plainnat}
\bibliography{sandpile_library} 

\newpage

\appendix

\section{Remaining proofs}
\label{Appendix}

\subsection*{Proof of Lemma \ref{lem_cornerstones_R}}

Let $\tilde{A}_R^{N,k}$ denote the remainder of $A^{(N)}$ in the sandpile configuration $\gamma_i \eta$, i.e., $\tilde{A}_R^{N,k} = A^{(N)} \setminus \{v_i\}$. Conditioning on the vertex $Y$ at which the next sand grain lands, we obtain
    \begin{equation*}
        \mathbb{E}[X(\gamma_i \eta)|Y \in A^{(N)}] = \dfrac{N^2-1}{N^2} \mathbb{E}[X(\gamma_i \eta)|Y \in \tilde{A}^{N,k}_R].
    \end{equation*}
    It follows that it is sufficient to prove that
\begin{equation}\label{Exp_av_size_R_2}
    \mathbb{E}[X(\gamma_i \eta)|Y \in \tilde{A}^{N,k}_R] = \begin{cases}
        \dfrac{3N^5+15N^4 + 15N^3-15N^2-18N+10(4k^3-24k^2+23k-6)}{30(N^2-1)}, &\text{if } N \text{ is odd,} \\[1.2em]
        \dfrac{3N^5 + 15N^4 + 15N^3 - 15N^2 - 18N + 20k(2k^2-9k+1)}{30(N^2-1)}, &\text{if } N \text{ is even.}
    \end{cases}
\end{equation}
    We again prove this statement by induction over $N$. We start by showing that expression (\ref{Exp_av_size_R_2}) holds for $N = 3$ and $k = 2$ and for $N = 4$ and $k = 2$. First, suppose that $N = 3$ and $k = 2$. Figure \ref{Figs_R_small} shows the directed graph obtained by applying Algorithm \ref{Avalanche_analysis} to the generator $\tilde{A}_R^{3, 2}$ in the sandpile configuration $\gamma_i \eta$.

       \begin{figure}[H]
    \centering
    \captionsetup[subfigure]{justification=centering, labelformat=empty, singlelinecheck=false, width=0.6\linewidth}

    \begin{subfigure}[b]{0.45\textwidth}
        \centering
        \includegraphics[width=0.7\textwidth]{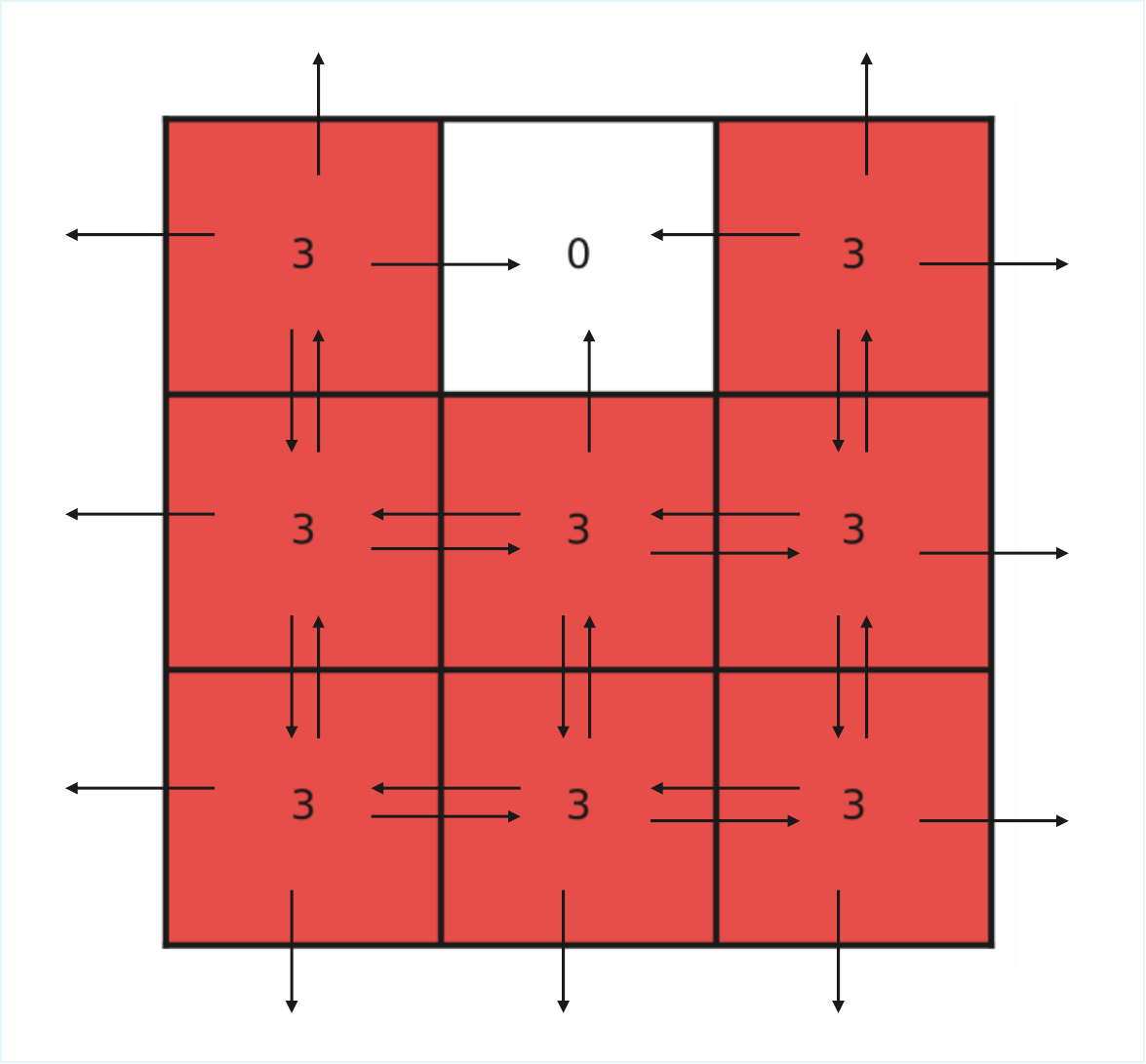}
        \caption[]%
        {{\small }}    
    \end{subfigure}
    \begin{subfigure}[b]{0.45\textwidth}  
        \centering 
        \includegraphics[width=0.7\textwidth]{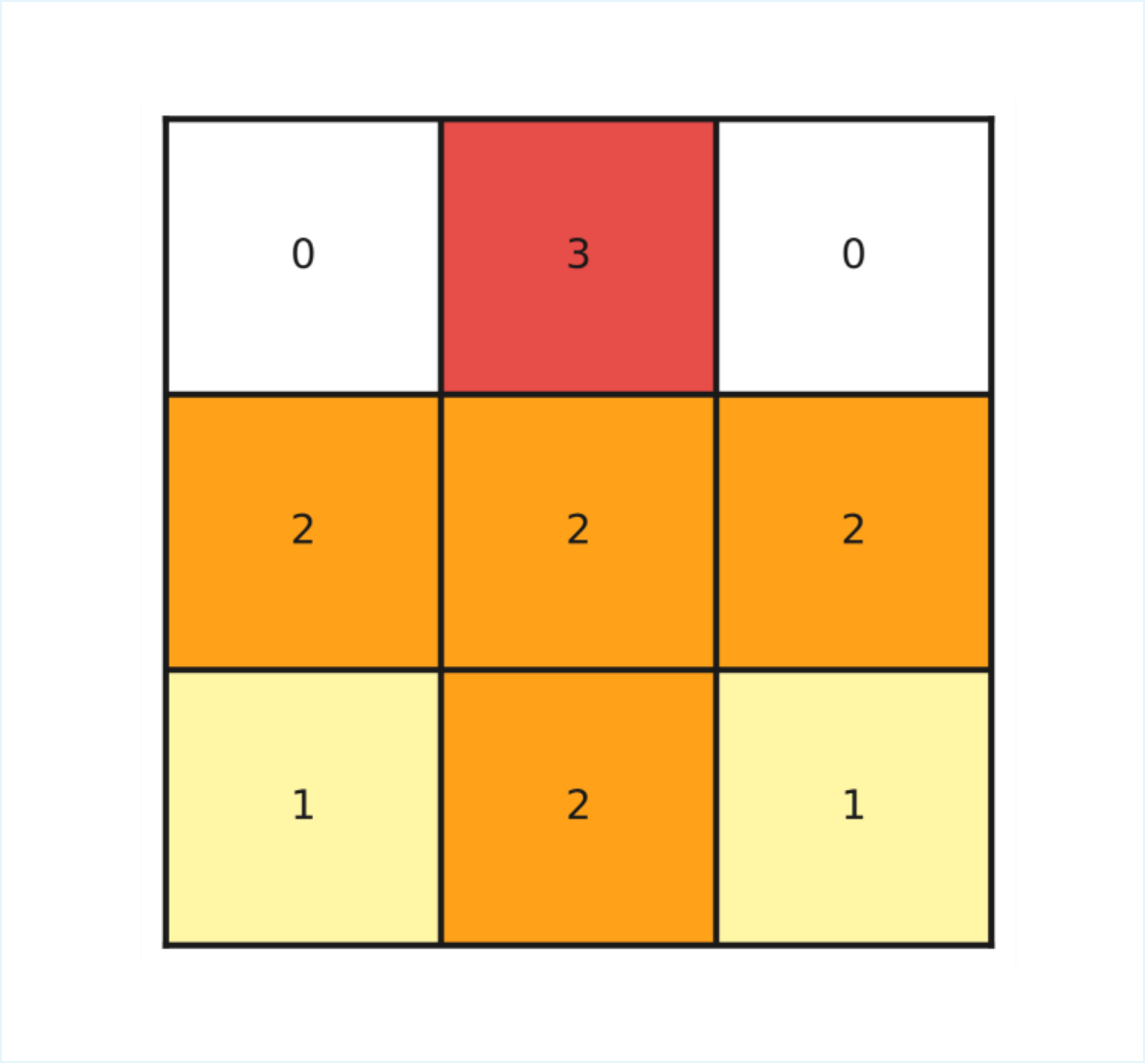}
        \caption[]%
        {{\small}}    
    \end{subfigure}
    \caption{Directed graph constructed in Algorithm 1 applied to generator $\tilde{A}_R^{3,2}$ in configuration $\gamma_i \eta$ (left) and sandpile configuration $(\gamma_i \eta)^{\tilde{A}_R^{3,2}}$ (right).} 
    \label{Figs_R_small}
\end{figure}

    Observe that $w_{\tilde{A}_R^{3, 2}}(\gamma_i \eta) = 8$ and $\{v \in \tilde{A}_R^{3, 2}| (\gamma_i \eta)^{\tilde{A}_R^{3, 2}}(v) = 3\} = \emptyset$. Thus, by expression (\ref{exp_av_size_rec_eq}), we have
    \begin{equation*}
        \mathbb{E}[X(\gamma_i \eta)|Y \in \tilde{A}^{3,2}_R] = 8,
    \end{equation*}
    which agrees with expression (\ref{Exp_av_size_R_2}). 

    Now, consider $N = 4$ and $k = 2$. The evolution of Algorithm \ref{Avalanche_analysis} for this case is depicted in Figure \ref{Figs_R_small_even}.

  \begin{figure}[H]
    \centering
    \captionsetup[subfigure]{justification=centering, labelformat=empty, singlelinecheck=false, width=0.6\linewidth}

    \begin{subfigure}[b]{0.32\textwidth}
        \centering
        \includegraphics[width=0.8\textwidth]{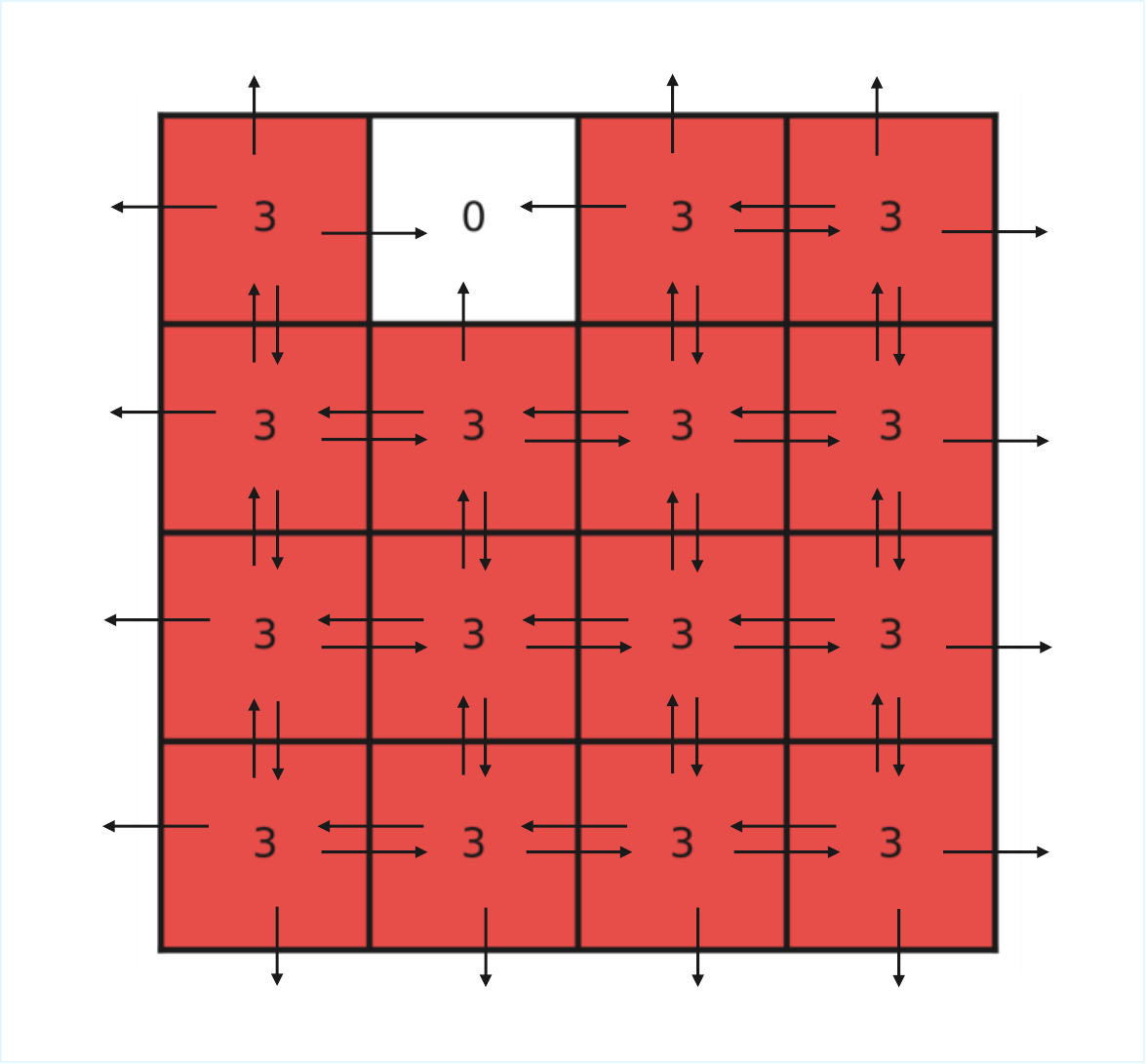}
        \caption[]%
        {{}}    
    \end{subfigure}
    \begin{subfigure}[b]{0.32\textwidth}  
        \centering 
        \includegraphics[width=0.8\textwidth]{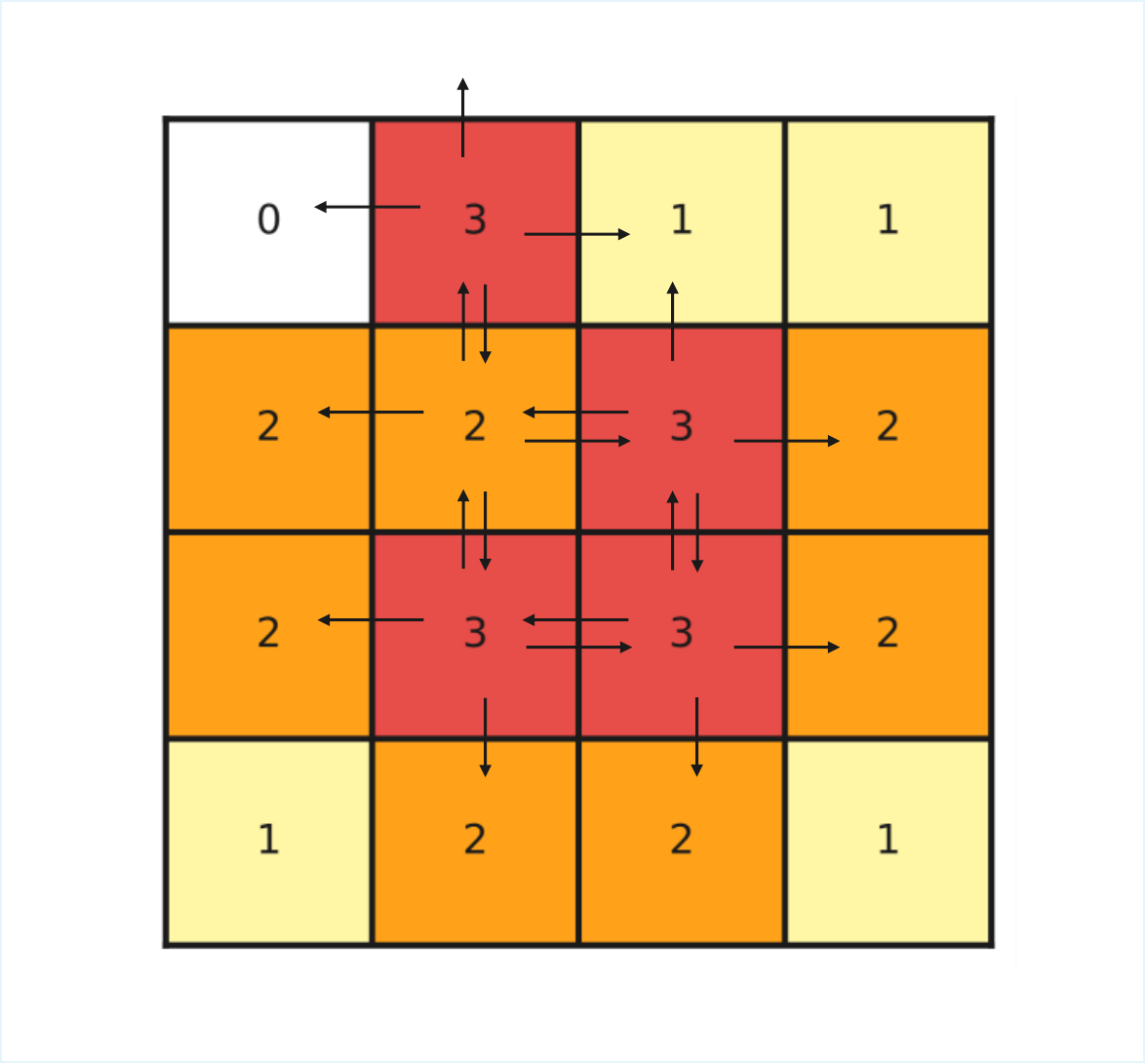}
        \caption[]%
        {{}}    
    \end{subfigure}
    \begin{subfigure}[b]{0.32\textwidth}  
        \centering 
        \includegraphics[width=0.8\textwidth]{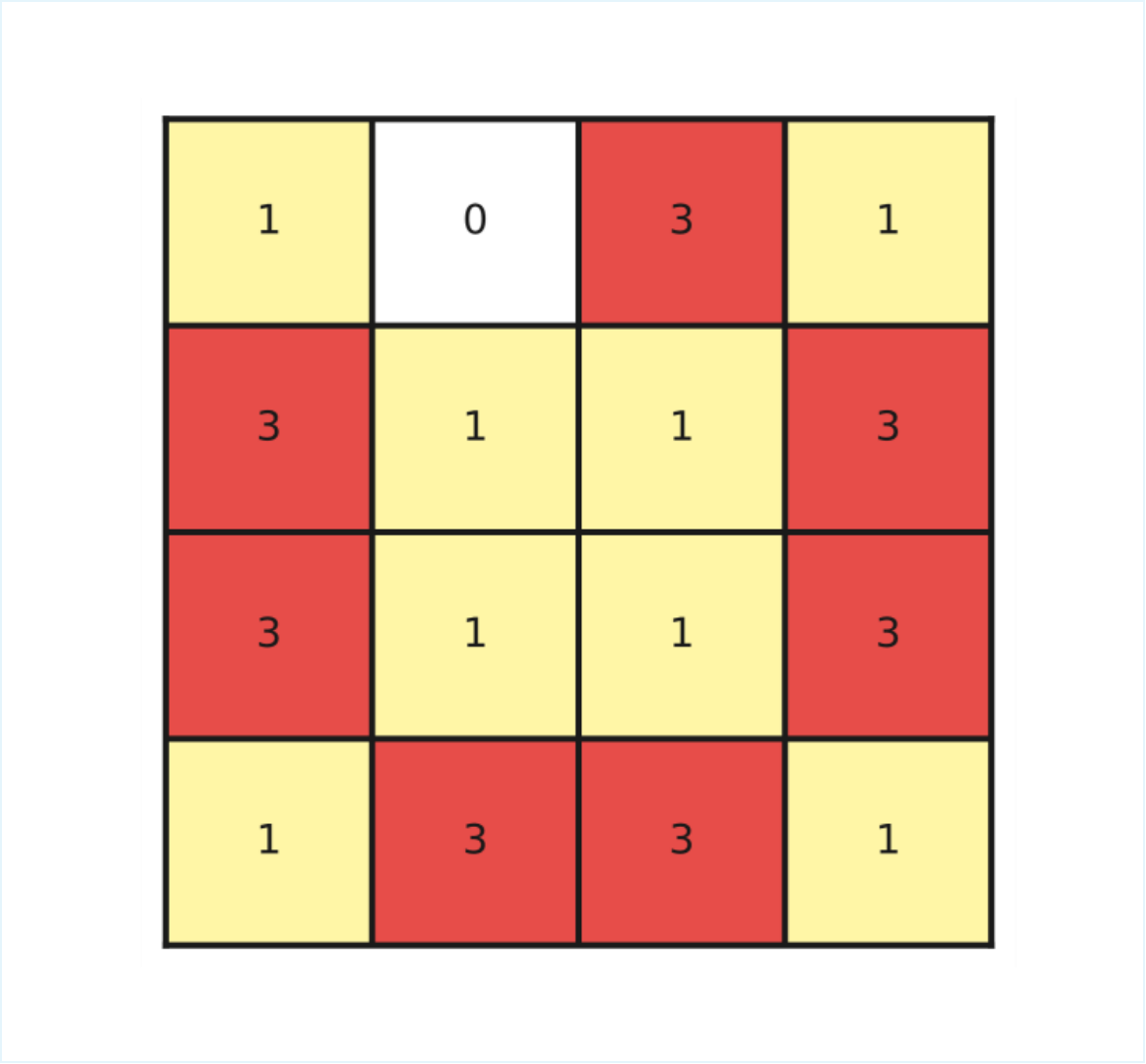}
        \caption[]%
        {{}}    
    \end{subfigure}
    \caption[ ]
    {\small Evolution of Algorithm \ref{Avalanche_analysis} applied to generator $\tilde{A}_R^{4,2}$ in sandpile configuration $\gamma_i \eta$. }
    \label{Figs_R_small_even}
\end{figure}

    Let $v_{i_1}$, $v_{i_2}$ and $v_{i_3}$ denote the neighbours of $v_i$ that are part of the corners, the inner boundary and the interior of $A^{(4)}$ respectively. We initialize a directed graph $(V, E)$ with $E = \emptyset$. For each $v \in \tilde{A}_R^{4,2}$, we now add edges to $E$ from $v$ to each of its neighbours. Now, observe that $(\gamma_i \eta)(v) + \text{indeg}(v) - \text{outdeg}(v) < 4$ for all $v \notin \tilde{A}_R^{4,2}$. Also, we obtain 
    \begin{equation*}
        (\gamma_i \eta)^{\tilde{A}_R^{4,2}}(v) = \begin{cases}
            0, &\text{if } v = v_{i_1}, \\
            1, &\text{if } v \in (C(A^{(4)}) \setminus \{v_{i_1}\}) \cup \{v_{i_2}\}, \\
            2, &\text{if } v \in (\delta^{\text{in}}(A^{(4)}) \setminus \{v_i, v_{i_2}\}) \cup \{v_{i_3}\}, \\
            3, &\text{otherwise.}
        \end{cases}
    \end{equation*}
    Hence, we have $w_{\tilde{A}_R^{4,2}} = 15$ and $\tilde{A}_{R,1}^{4,2} := \{v \in \tilde{A}_R^{4,2}|(\gamma_i \eta)^{\tilde{A}_R^{4,2}}(v) = 3\} = \text{int}(A^{(4)}) \setminus \{v_{i_3}\}$. Applying Algorithm \ref{Avalanche_analysis} to the generator $\tilde{A}_{R,1}^{4,2}$ in the sandpile configuration $(\gamma_i \eta)^{\tilde{A}_R^{4,2}}$ now yields $w_{\tilde{A}_{R,1}^{4,2}}((\gamma_i \eta)^{\tilde{A}_R^{4,2}}) = 5$ and $\{v \in \tilde{A}_{R,1}^{4,2}|((\gamma_i \eta)^{\tilde{A}_R^{4,2}})^{\tilde{A}_{R,1}^{4,2}}(v) = 3\} = \emptyset$. Thus, we obtain
    \begin{equation*}
        \mathbb{E}[X(\gamma_i \eta)|Y \in \tilde{A}^{4,2}_R] = 16,
    \end{equation*}
    which is consistent with expression (\ref{Exp_av_size_R_2}).

    Suppose now that expression (\ref{Exp_av_size_R_2}) holds for $N = n-2$ and all $k = 2, \ldots, \lceil (n-2)/2 \rceil$ for some $n = 5, 6, \ldots$. We show that this implies its validity for $N = n$ and all $k = 2, \ldots, \lceil n/2 \rceil$ by means of an embedded induction argument. We first consider the case $k = \lceil n/2 \rceil$. The evolution of Algorithm \ref{Avalanche_analysis} for $N = 7$ and $k = 4$ is depicted in Figure \ref{Figs_R}.

    \begin{figure}[H]
    \centering
    \captionsetup[subfigure]{justification=centering, labelformat=empty, singlelinecheck=false, width=0.6\linewidth}

    \begin{subfigure}[b]{0.32\textwidth}
        \centering
        \includegraphics[width=\textwidth]{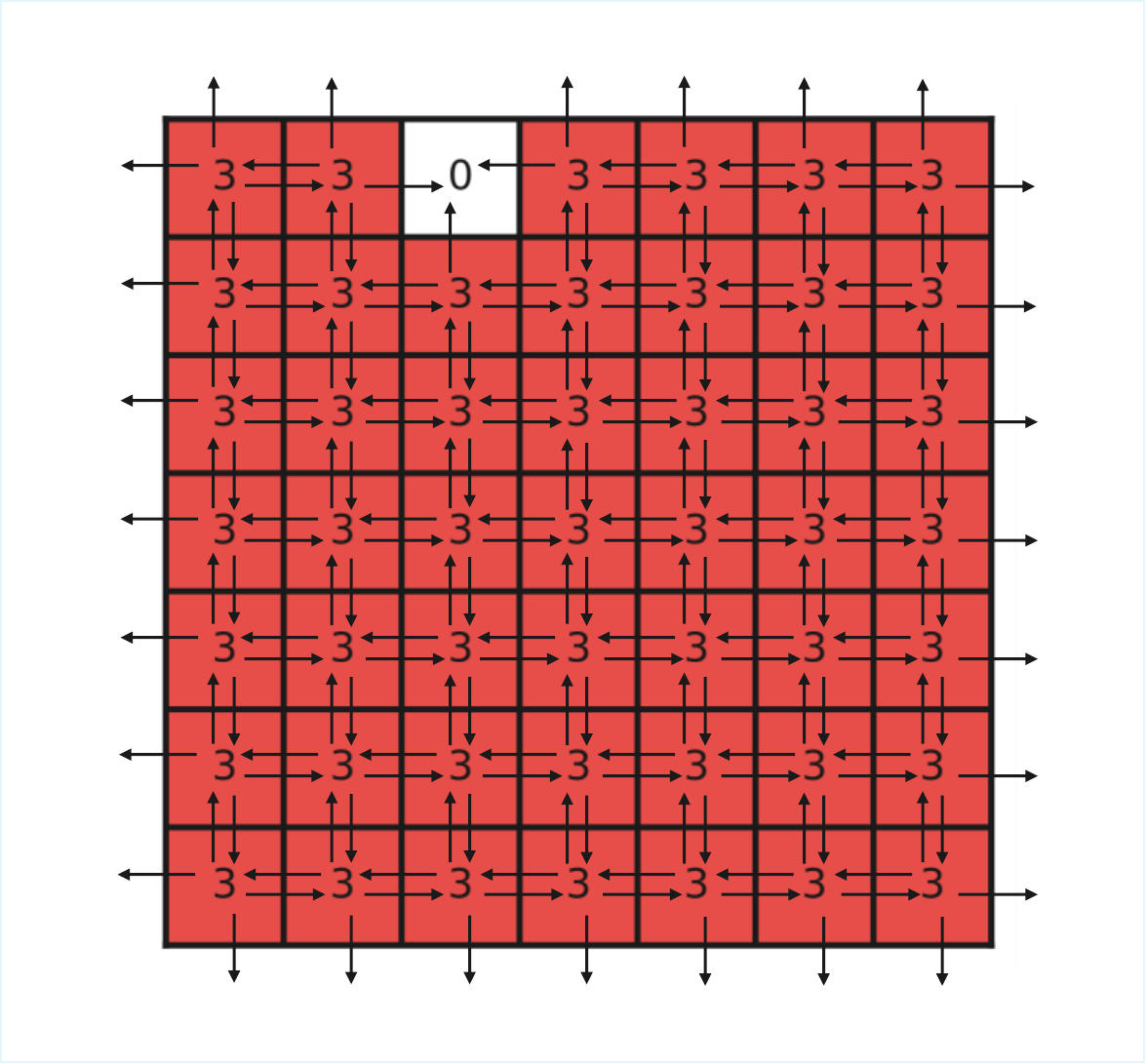}
        \caption[]%
        {{}}    
        \label{R}
    \end{subfigure}
    \begin{subfigure}[b]{0.32\textwidth}  
        \centering 
        \includegraphics[width=\textwidth]{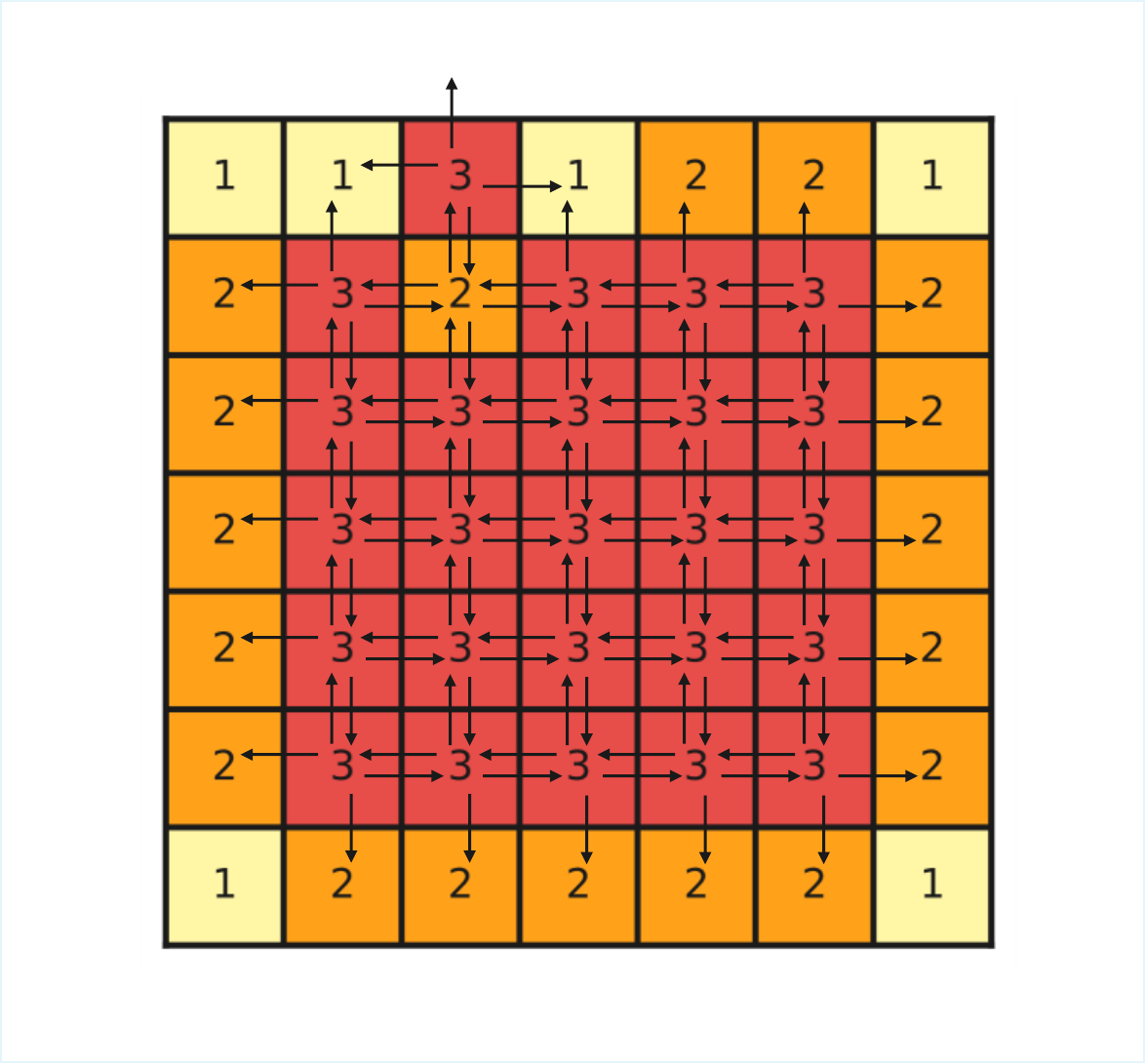}
        \caption[]%
        {{}}    
        \label{R_intermediate}
    \end{subfigure}
    \begin{subfigure}[b]{0.32\textwidth}  
        \centering 
        \includegraphics[width=\textwidth]{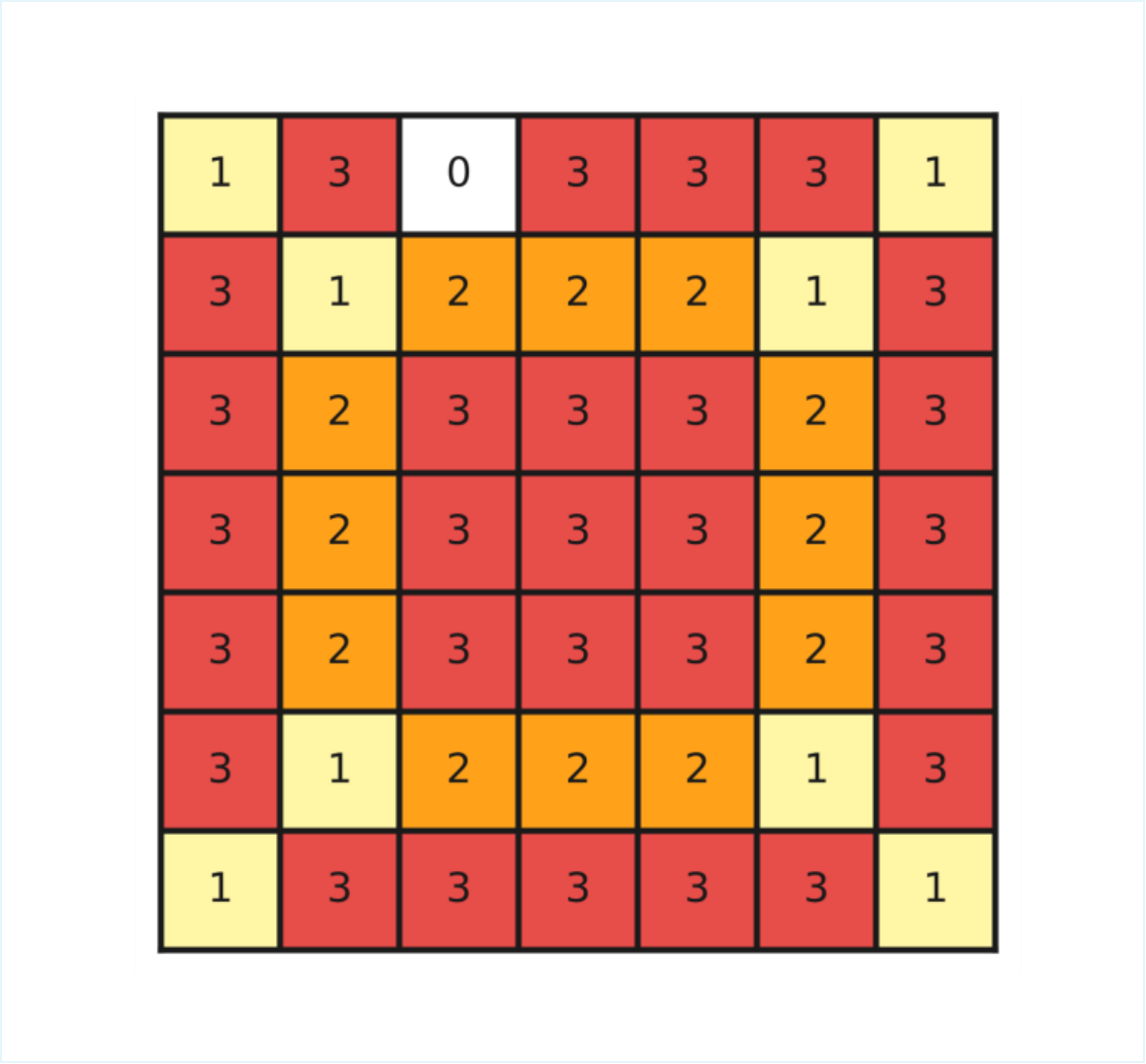}
        \caption[]%
        {{}}    
        \label{R_after}
    \end{subfigure}
    \caption[ ]
    {\small Evolution of Algorithm \ref{Avalanche_analysis} applied to generator $\tilde{A}^{7,4}$ in sandpile configuration $\gamma_i \eta$.} 
    \label{Figs_R}
\end{figure}
    Let $v_{i_1}$ and $v_{i_2}$ denote the neighbours of $v_i$ that are part of $\delta^{\text{in}}(A^{(N)}) \cup C(A^{(N)})$ and let $v_{i_3}$ denote the neighbour of $v_i$ that is part of $\text{int}(A^{(N)})$. Although Figure \ref{Figs_R} shows a case in which both $v_{i_1}$ and $v_{i_2}$ are part of $\delta^{\text{in}}(A^{(N)})$, observe that either $v_{i_1}$ or $v_{i_2}$ may be part of $C(A^{(N)})$. Following Algorithm \ref{Avalanche_analysis}, we start by initializing the directed graph $(V, E)$ with $E = \emptyset$ and add edges to $E$ from $v$ to each of its neighbours for each $v \in \tilde{A}_R^{n,\lceil n/2 \rceil}$. At this point, note that $(\gamma_i \eta)(v) + \text{indeg}(v) - \text{outdeg}(v) < 4$ for all $v \in V$. The obtained directed graph for the case $N = 7$, $k = 4$ is provided in Figure \ref{Figs_R} (left). Hence, we obtain $(\gamma_i \eta)^{\tilde{A}_R^{n,\lceil n/2 \rceil}}(v) = 3$ for all $v \in \{v_i\} \cup (\text{int}(A^{(N)}) \setminus \{v_{i_3}\})$ and $(\gamma_i \eta)^{\tilde{A}_R^{n,\lceil n/2 \rceil}}(v_{i_3}) = 2$. In addition, $(\gamma_i \eta)^{\tilde{A}_R^{n,\lceil n/2 \rceil}}(v) < 3$ for all $v \in (\delta^{\text{in}}(A^{(N)}) \cup C(A^{(N)})) \setminus \{v_i\}$. This implies that $w_{\tilde{A}_R^{n,\lceil n/2 \rceil}}(\gamma_i \eta) = n^2-1$ and the set $\tilde{A}^{n,\lceil n/2 \rceil}_{R,1} := \{v \in \tilde{A}_R^{n,\lceil n/2 \rceil}| (\gamma_i \eta)^{\tilde{A}_R^{n,\lceil n/2 \rceil}}(v) = 3\} = \text{int}(A^{(N)}) \setminus \{v_{i_3}\}$.

    We now apply Algorithm \ref{Avalanche_analysis} to the generator $\tilde{A}^{n,\lceil n/2 \rceil}_{R,1}$ in the sandpile configuration $(\gamma_i \eta)^{\tilde{A}_R^{n,\lceil n/2 \rceil}}$. Again, we start by initializing the directed graph $(V, E)$ with $E = \emptyset$ and add edges to $E$ from $v$ to each of its neighbours for each $v \in~\tilde{A}^{n,\lceil n/2 \rceil}_{R,1}$. At this point, note that the only vertex $v \notin \tilde{A}^{n,\lceil n/2 \rceil}_{R,1}$ that satisfies $(\gamma_i \eta)^{\tilde{A}_R^{n,\lceil n/2 \rceil}}(v) + \text{indeg}(v) - \text{outdeg}(v) \geq 4$ is $v = v_{i_3}$. Hence, we add edges to $E$ from $v_{i_3}$ to each of its neighbours. Now, the only vertex $v \notin \tilde{A}^{n,\lceil n/2 \rceil}_{R,1}$ for which $(\gamma_i \eta)^{\tilde{A}_R^{n,\lceil n/2 \rceil}}(v) + \text{indeg}(v) - \text{outdeg}(v) \geq 4$ is $v = v_i$. After adding edges from $v_i$ to each of its neigbours, step 4 of Algorithm \ref{Avalanche_analysis} terminates. The directed graph resulting from this iteration is illustrated in Figure \ref{Figs_R} (middle). Note that $w_{\tilde{A}^{n,\lceil n/2 \rceil}_{R,1}}((\gamma_i \eta)^{\tilde{A}_R^{n,\lceil n/2 \rceil}}) = (n-2)^2 +1$ and the set $\{v \in~\tilde{A}^{n,\lceil n/2 \rceil}_{R,1}|((\gamma_i \eta)^{\tilde{A}_R^{n,\lceil n/2 \rceil}})^{\tilde{A}^{n,\lceil n/2 \rceil}_{R,1}}(v) =~3\}$ is exactly a square of size $(n-4) \times (n-4)$, which we will denote by $A^{(n-4)}$. Evaluating expression (\ref{exp_av_size_rec_eq}) and using expression (\ref{exp_av_size_square}) now yields
    \begin{align*}
        &\mathbb{E}[X(\gamma_i \eta)|Y \in \tilde{A}^{n \lceil n/2 \rceil}_R] = n^2 -1 + \dfrac{(n-2)^2-1}{n^2-1} \mathbb{E}[X((\gamma_i \eta)^{\tilde{A}_R^{n, \lceil n/2 \rceil}})|Y \in \tilde{A}^{n,\lceil n/2 \rceil}_{R,1}] \\
        &= n^2 - 1 + \dfrac{(n-2)^2-1}{n^2-1}\left[(n-2)^2+1 + \dfrac{(n-4)^2}{(n-2)^2-1} \dfrac{3(n-4)^4+15(n-4)^3 + 20(n-4)^2-8}{30(n-4)}\right]\\
        &= \dfrac{n(3n^4+15n^3+20n^2-60n-8)}{30(n^2-1)},
    \end{align*}
    which is consistent with expression (\ref{Exp_av_size_R_2}) for $k = \lceil n/2 \rceil$. 

    Within the embedded induction argument, we now assume that expression (\ref{Exp_av_size_R_2}) holds for $N = n$ and $k = \ell + 1$ for some $\ell = 2, \ldots, \lceil N/2 \rceil -1$. We proceed to show that the statement holds for $k = \ell$. The directed graph resulting from Algorithm \ref{Avalanche_analysis} and the sandpile configuration $(\gamma_i \eta)^{\tilde{A}_R^{n, k}}$ for the case $n = 7$ and $k = 3$ are provided in Figure \ref{Figs_Rcase2}.

     \begin{figure}[H]
    \centering
    \captionsetup[subfigure]{justification=centering, labelformat=empty, singlelinecheck=false, width=0.6\linewidth}

    \begin{subfigure}[b]{0.45\textwidth}
        \centering
        \includegraphics[width=0.7\textwidth]{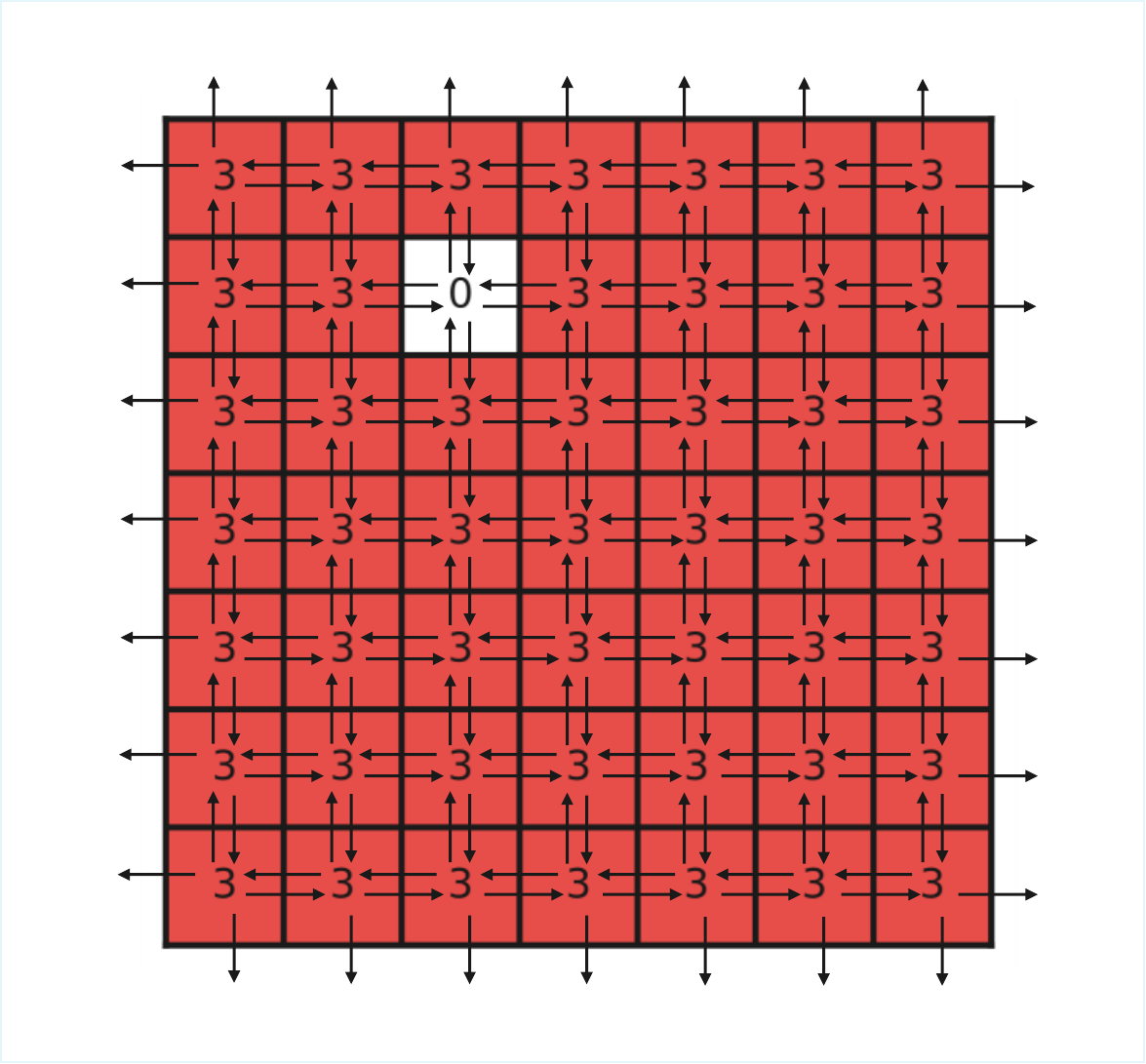}
        \caption[]%
        {{\small }}    
        \label{R_case2}
    \end{subfigure}
    \begin{subfigure}[b]{0.45\textwidth}  
        \centering 
        \includegraphics[width=0.7\textwidth]{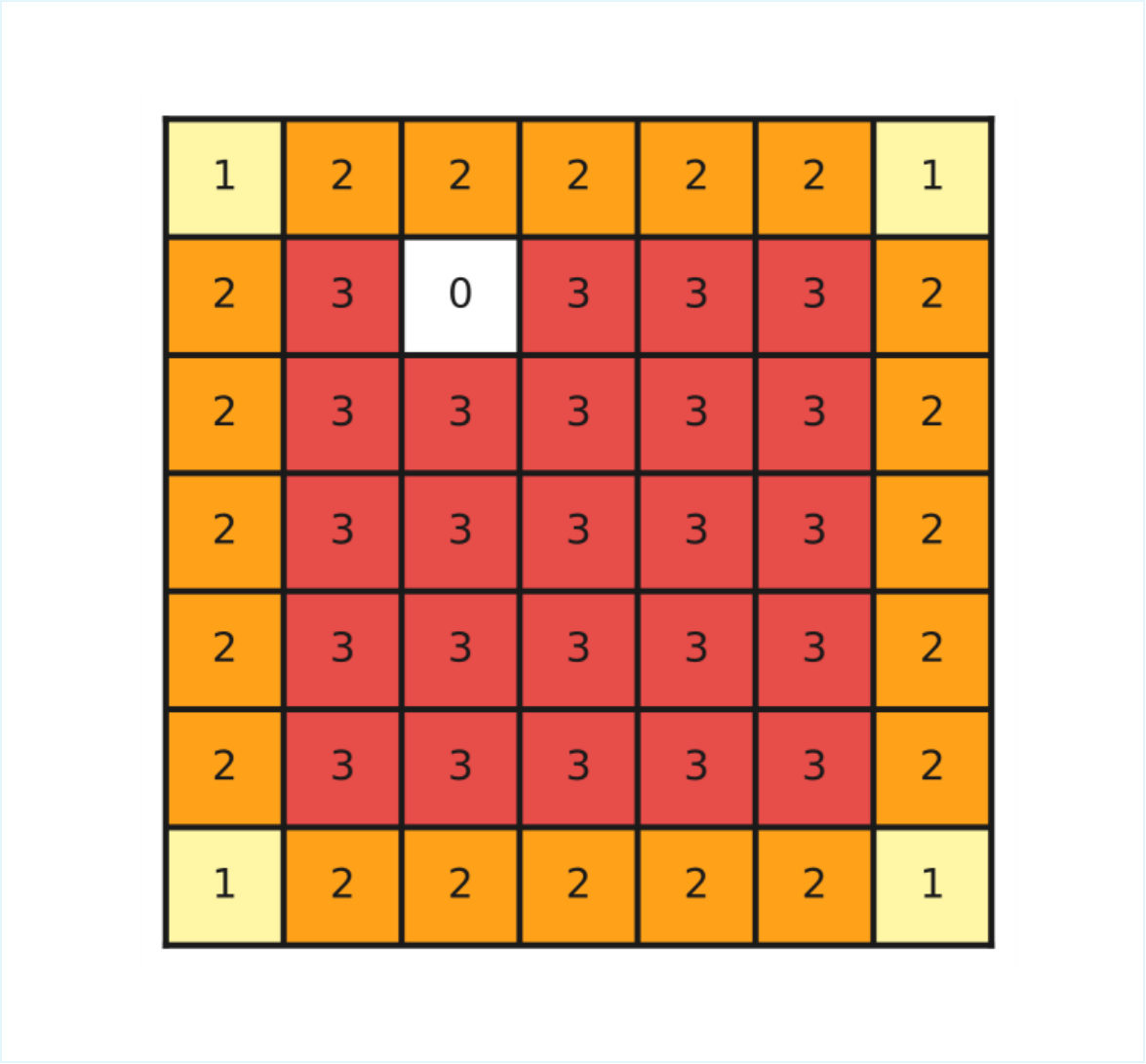}
        \caption[]%
        {{\small}}    
        \label{R_case2_after}
    \end{subfigure}
    \caption{Evolution of Algorithm \ref{Avalanche_analysis} applied to generator $\tilde{A}_R^{7,3}$ in sandpile configuration $\gamma_i \eta$.} 
    \label{Figs_Rcase2}
    \end{figure}
    Let $v_{i_1}$ and $v_{i_3}$ denote the neighbours of $v_i$ that are part of $R_{\ell}(A^{(N)})$, let $v_{i_2}$ denote the neighbour of $v_i$ that is part of $R_{\ell + 1}(A^{(N)})$ and let $v_{i_4}$ denote the neighbour of $v_i$ that is part of $R_{\ell - 1}(A^{(N)})$. Again, we start by initializing the directed graph $(V, E)$ with $E = \emptyset$. For each $v \in \tilde{A}_R^{n,\ell}$, we add edges to $E$ from $v$ to each of its neighbours. Note that at this point, we have $(\gamma_i \eta)(v) + \text{indeg}(v) - \text{outdeg}(v) = 4$ if and only if $v = v_i$. Hence, we continue to add edges to $E$ from $v_i$ to each of its neighbours. Now, we obtain $(\gamma_i \eta)(v) + \text{indeg}(v) - \text{outdeg}(v) < 4$ for all $v \in V$. The directed graph resulting from Algorithm \ref{Avalanche_analysis} is depicted in Figure \ref{Figs_Rcase2} (left). Observe now that
    \begin{equation*}
        (\gamma_i \eta)^{\tilde{A}_R^{n, \ell}}(v) = \begin{cases}
            0, &\text{if } v = v_i, \\
            1, &\text{if } v \in C(A^{(N)}), \\
            2, &\text{if } v \in R_{\lceil n/2 \rceil}(A^{(N)}) \setminus C(A^{(N)}), \\
            3, &\text{otherwise}.
        \end{cases}
    \end{equation*}
    It follows that $w_{\tilde{A}_R^{n,\ell}}(\gamma_i \eta) = n^2$ and that the set $\tilde{A}^{n,\ell}_{R,1} := \{v \in \tilde{A}_R^{n,\ell}| (\gamma_i \eta)^{\tilde{A}_R^{n,\ell}}(v) = 3\} = \text{int}(A^{(N)}) \setminus \{v_i\}$. We now use the induction hypothesis to evaluate expression (\ref{exp_av_size_rec_eq}):
    \begin{align*}
        &\mathbb{E}[X(\gamma_i \eta)|Y \in \tilde{A}^{n, \ell}_R] = n^2 + \dfrac{(n-2)^2-1}{n^2-1} \mathbb{E}[X((\gamma_i \eta)^{\tilde{A}_R^{n,\ell}})|Y \in \tilde{A}^{n, \ell}_{R,1}] \\
        &= n^2 + \dfrac{(n-2)^2-1}{n^2-1} \dfrac{3(n-2)^5 + 15(n-2)^4 + 15(n-2)^3 - 15(n-2)^2-18(n-2) + 10(4\ell^3-24\ell^2+23\ell-6)}{30((n-2)^2-1)} \\
        &= \dfrac{3n^5+15n^4 + 15n^3-15n^2-18n+10(4\ell^3-24\ell^2+23\ell-6)}{30(n^2-1)}
    \end{align*}
    if $n$ is odd and 
    \begin{align*}
        &\mathbb{E}[X(\gamma_i \eta)|Y \in \tilde{A}^{n, \ell}_R] = n^2 + \dfrac{(n-2)^2-1}{n^2-1} \mathbb{E}[X((\gamma_i \eta)^{\tilde{A}_R^{n, \ell}}|Y \in \tilde{A}^{n,\ell}_{R,1}] \\
        &= n^2 + \dfrac{(n-2)^2-1}{n^2-1} \dfrac{3(n-2)^5 + 15(n-2)^4+15(n-2)^3-15(n-2)^2-18(n-2)+20 \ell(2\ell^2 - 9\ell + 1)}{30((n-2)^2-1)} \\
        &= \dfrac{3n^5 + 15n^4 + 15n^3 - 15n^2 - 18n + 20\ell(2\ell^2-9\ell+1)}{30(n^2-1)}
    \end{align*}
    if $n$ is even. Thus, we established the validity of expression (\ref{Exp_av_size_R_2}) for $N = n$ and all $k = 2, \ldots, \lceil n/2 \rceil$. The complete induction argument now yields the statement for all $N \geq 3$ and all $k = 2, \ldots, \lceil n/2 \rceil$.

\subsection*{Proof of Lemma \ref{lem_cornerstones_C}}

Let $\tilde{A}_C^{N, k}$ denote the remainder of $A^{(N)}$ in the sandpile configuration $\gamma_i \eta$, i.e., $\tilde{A}_C^{N,k} = A^{(N)} \setminus \{v_i\}$. Note that
 \begin{equation*}
        \mathbb{E}[X(\gamma_i \eta)|Y \in A^{(N)}] = \dfrac{N^2-1}{N^2} \mathbb{E}[X(\gamma_i \eta)|Y \in \tilde{A}^{N,k}_C].
    \end{equation*}
Hence, it suffices to show that
\begin{equation}\label{Exp_av_size_C_2}
        \mathbb{E}[X(\gamma_i \eta)|Y \in \tilde{A}_C^{N,k}] = \begin{cases}
            \dfrac{3N^5 + 15N^4 + 15N^3-15N^2 - 18N+10(4k^3 - 24k^2+23k-3)}{30(N^2-1)}, &\text{if } N \text{ is odd,}\\[1.2em]
            \dfrac{3N^5 + 15N^4 + 15N^3 - 15N^2 - 18N + 10(4k^3 - 18k^2 + 2k + 3)}{30(N^2-1)}, &\text{if } N \text{ is even.}
        \end{cases}
    \end{equation}
    We prove the statement by induction over $N$. We first show that expression (\ref{Exp_av_size_C_2}) holds for $N = 3$ and $k = 2$ and for $N = 4$ and $k = 2$. Consider the case $N = 3$ and $k = 2$. The result of applying Algorithm \ref{Avalanche_analysis} to the generator $\tilde{A}_C^{3,2}$ in the sandpile configuration $\gamma_i \eta$ is depicted in Figure \ref{Figs_C_small}.

        \begin{figure}[H]
    \centering
    \captionsetup[subfigure]{justification=centering, labelformat=empty, singlelinecheck=false, width=0.5\linewidth}

    \begin{subfigure}[b]{0.3\textwidth}
        \centering
        \includegraphics[width=1.0\textwidth]{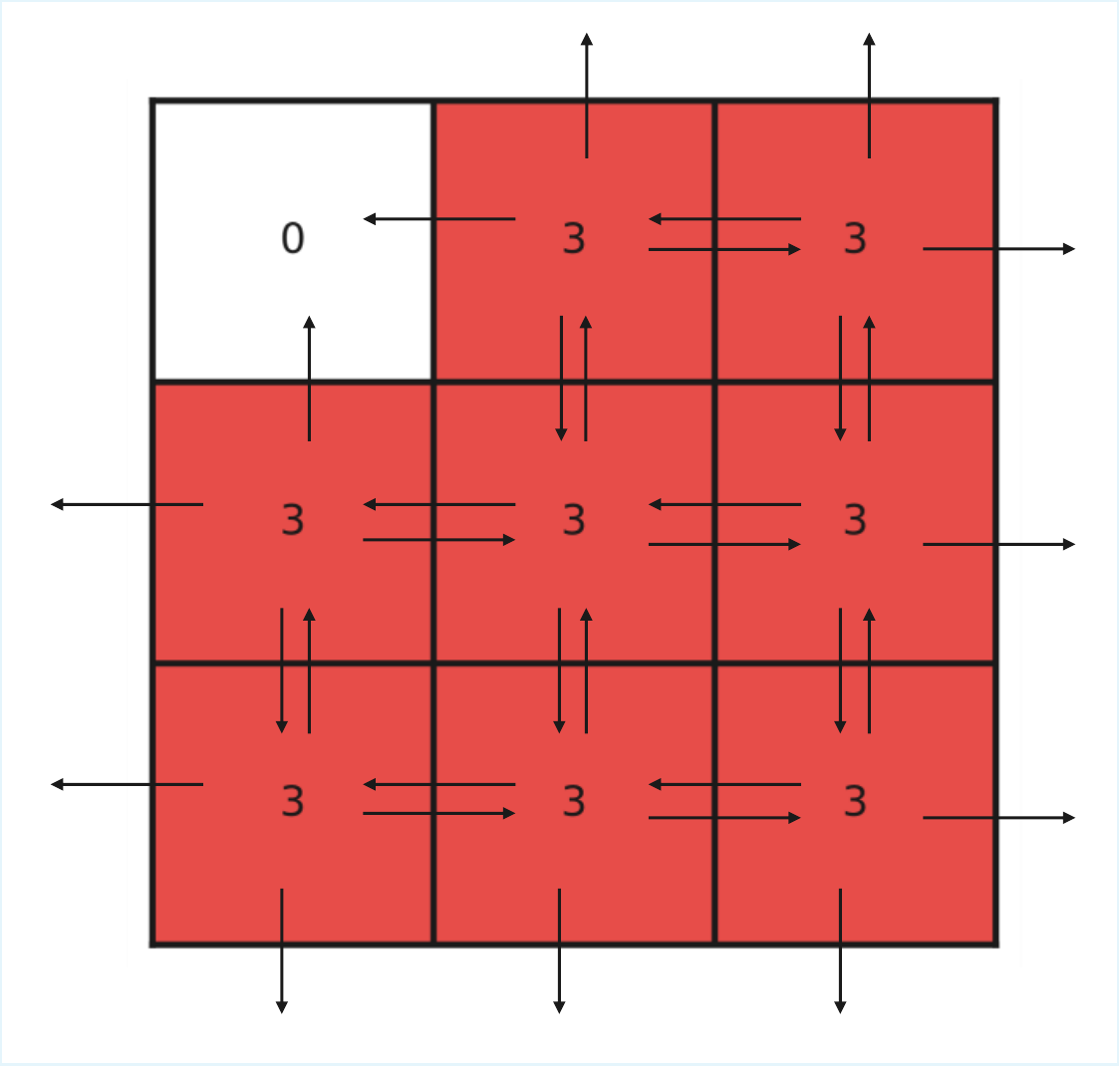}
        \caption[]%
        {{\small }}    
    \end{subfigure}
    \begin{subfigure}[b]{0.3\textwidth}  
        \centering 
        \includegraphics[width=1.0\textwidth]{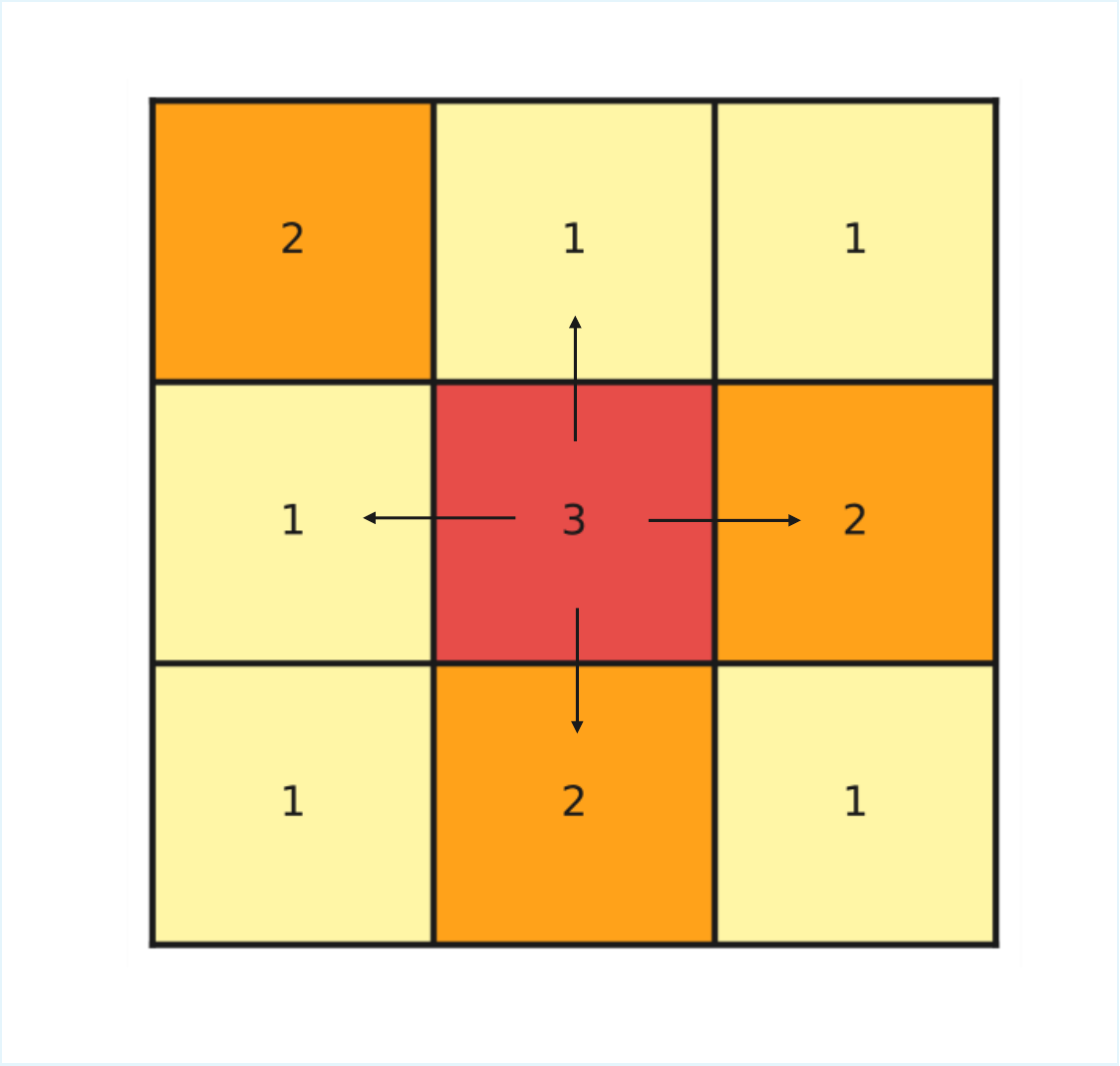}
        \caption[]%
        {{\small }}    
    \end{subfigure}
    \begin{subfigure}[b]{0.3\textwidth}  
        \centering 
        \includegraphics[width=1.0\textwidth]{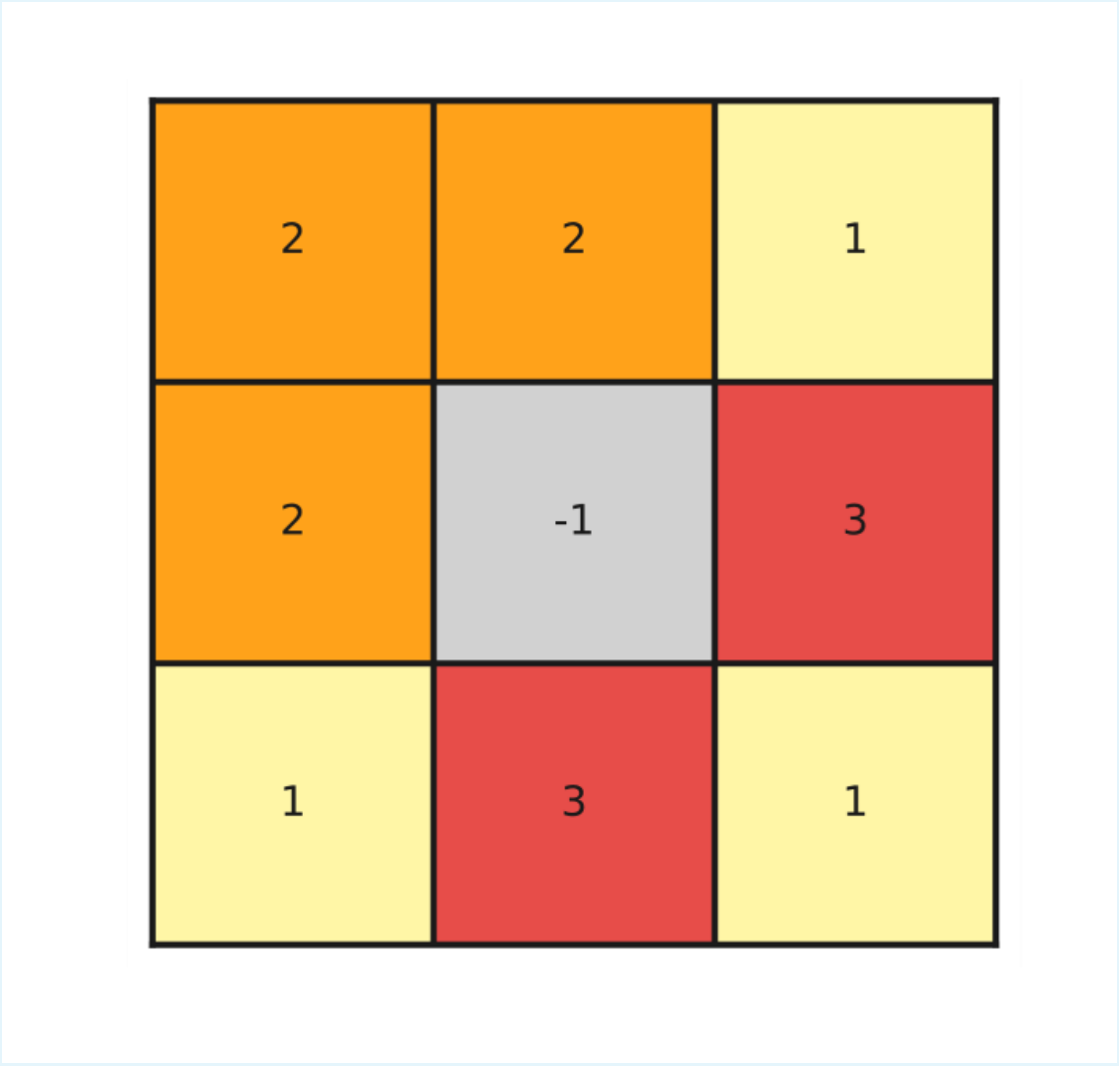}
        \caption[]%
        {{\small }}    
    \end{subfigure}
    \caption[ ]
    {\small Evolution of Algorithm \ref{Avalanche_analysis} applied to generator $\tilde{A}_C^{3, 2}$ in sandpile configuration $\gamma_i \eta$.} 
    \label{Figs_C_small}
\end{figure}
Let the vertex in the center of the generator $A^{(3)}$ be denoted by $\hat{v}$. Observe that $w_{\tilde{A}_C^{3,2}}(\gamma_i \eta) = 8$ and $\tilde{A}^{3,2}_{C,1} = \{v \in \tilde{A}_C^{3,2}|(\gamma_i \eta)^{\tilde{A}_C^{3,2}}(v) = 3\} = \{\hat{v}\}$. Also, note that $w_{\tilde{A}^{3,2}_{C,1}}((\gamma_i \eta)^{\tilde{A}_C^{3,2}}) = 1$ and $\{v \in \tilde{A}^{3,2}_{C,1} | ((\gamma_i \eta)^{\tilde{A}_C^{3,2}})^{\tilde{A}^{3,2}_{C,1}}(v) = 3\} = \emptyset$. Using expression (\ref{exp_av_size_rec_eq}), we obtain
\begin{equation*}
    \mathbb{E}[X(\gamma_i \eta)|Y \in \tilde{A}_C^{3,2}] = \dfrac{65}{8},
\end{equation*}
which is consistent with expression (\ref{Exp_av_size_C_2}).

Now, consider the case $N = 4$ and $k = 2$. The evolution of Algorithm \ref{Avalanche_analysis} is shown in Figure \ref{Figs_C_small_even}.

    \begin{figure}[H]
    \centering
    \captionsetup[subfigure]{justification=centering, labelformat=empty, singlelinecheck=false, width=0.5\linewidth}

    \begin{subfigure}[b]{0.3\textwidth}
        \centering
        \includegraphics[width=\textwidth]{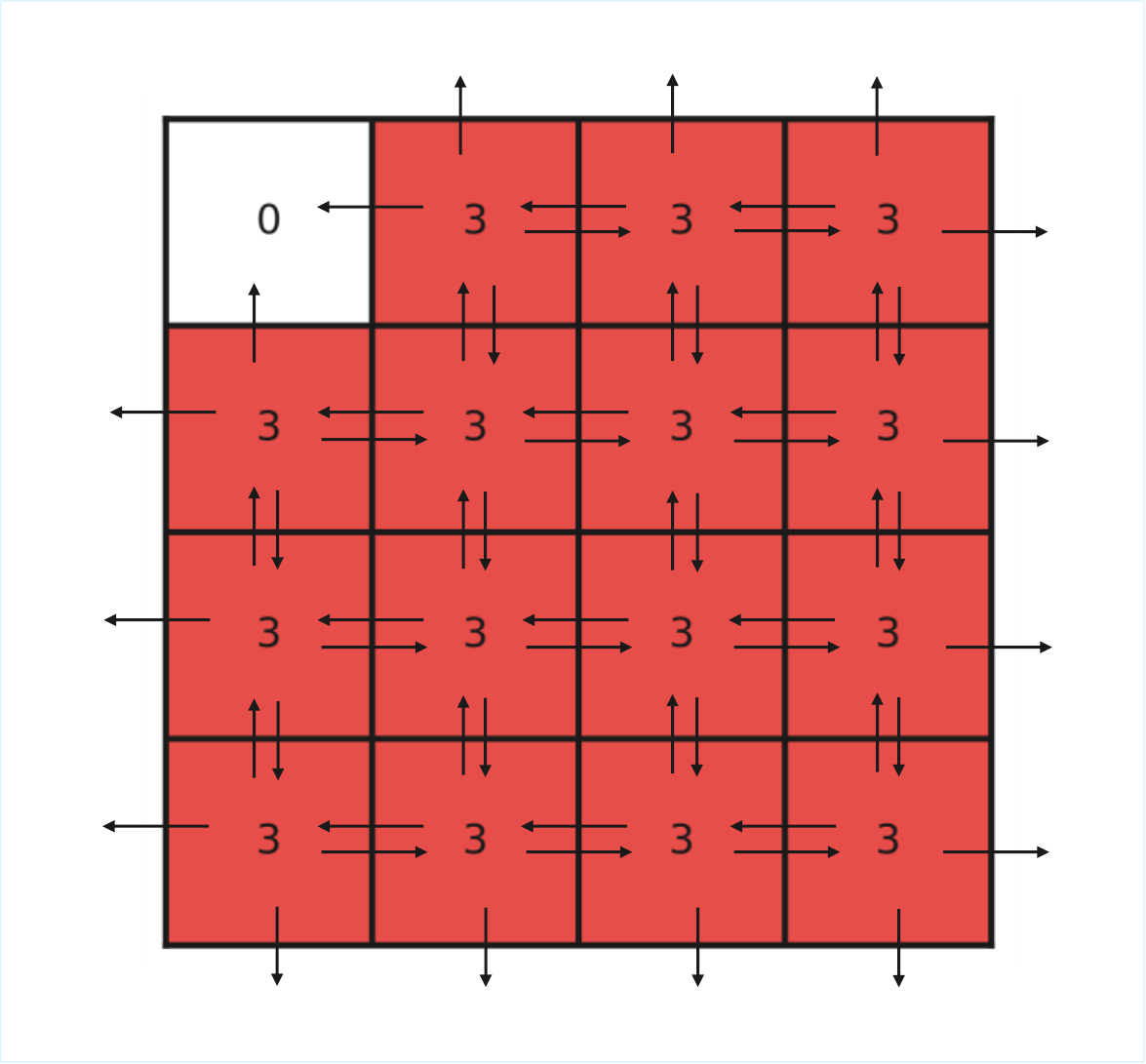}
        \caption[]%
        {{\small }}    
    \end{subfigure}
    \begin{subfigure}[b]{0.3\textwidth}  
        \centering 
        \includegraphics[width=\textwidth]{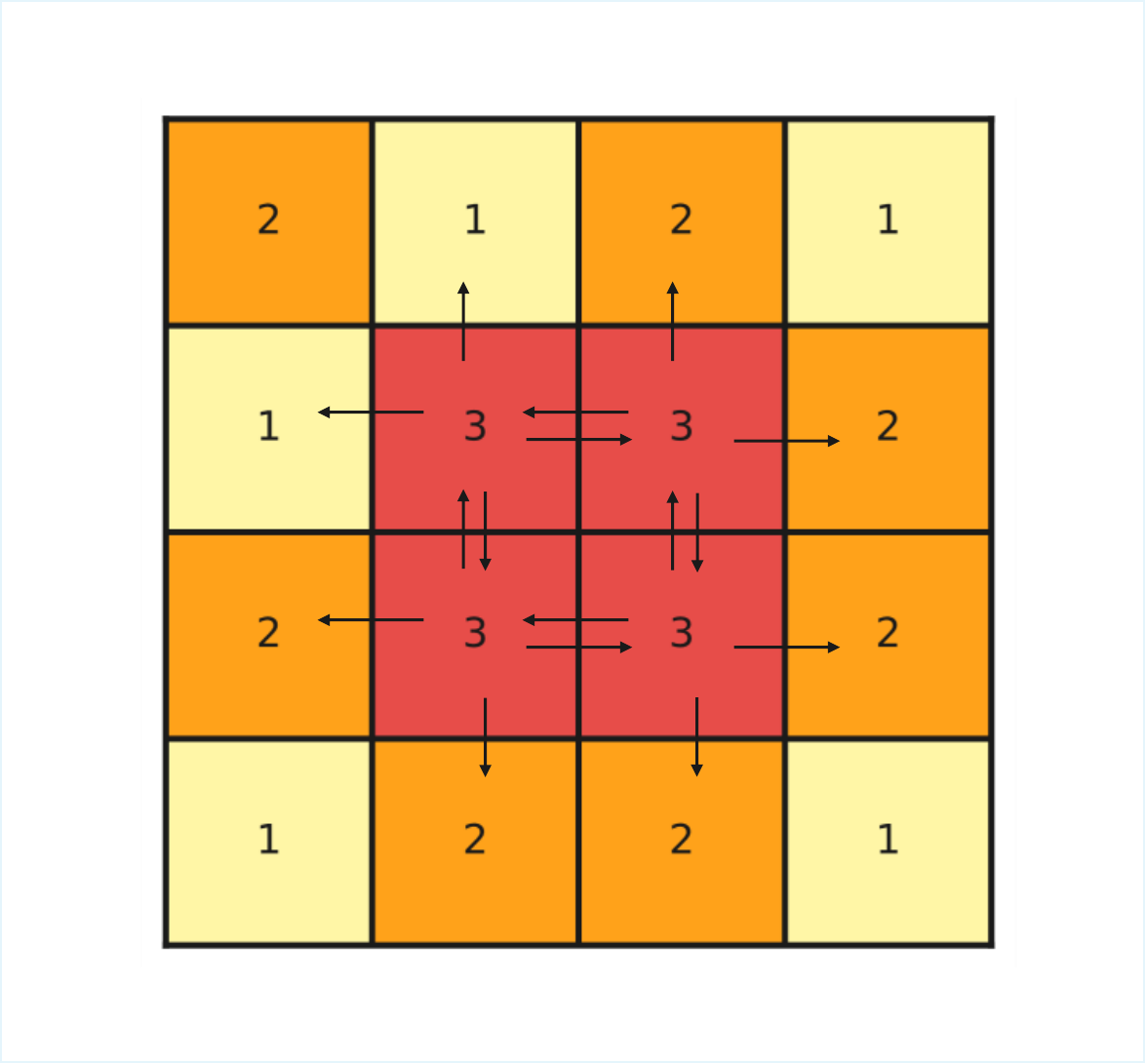}
        \caption[]%
        {{\small }}    
    \end{subfigure}
    \begin{subfigure}[b]{0.3\textwidth}  
        \centering 
        \includegraphics[width=\textwidth]{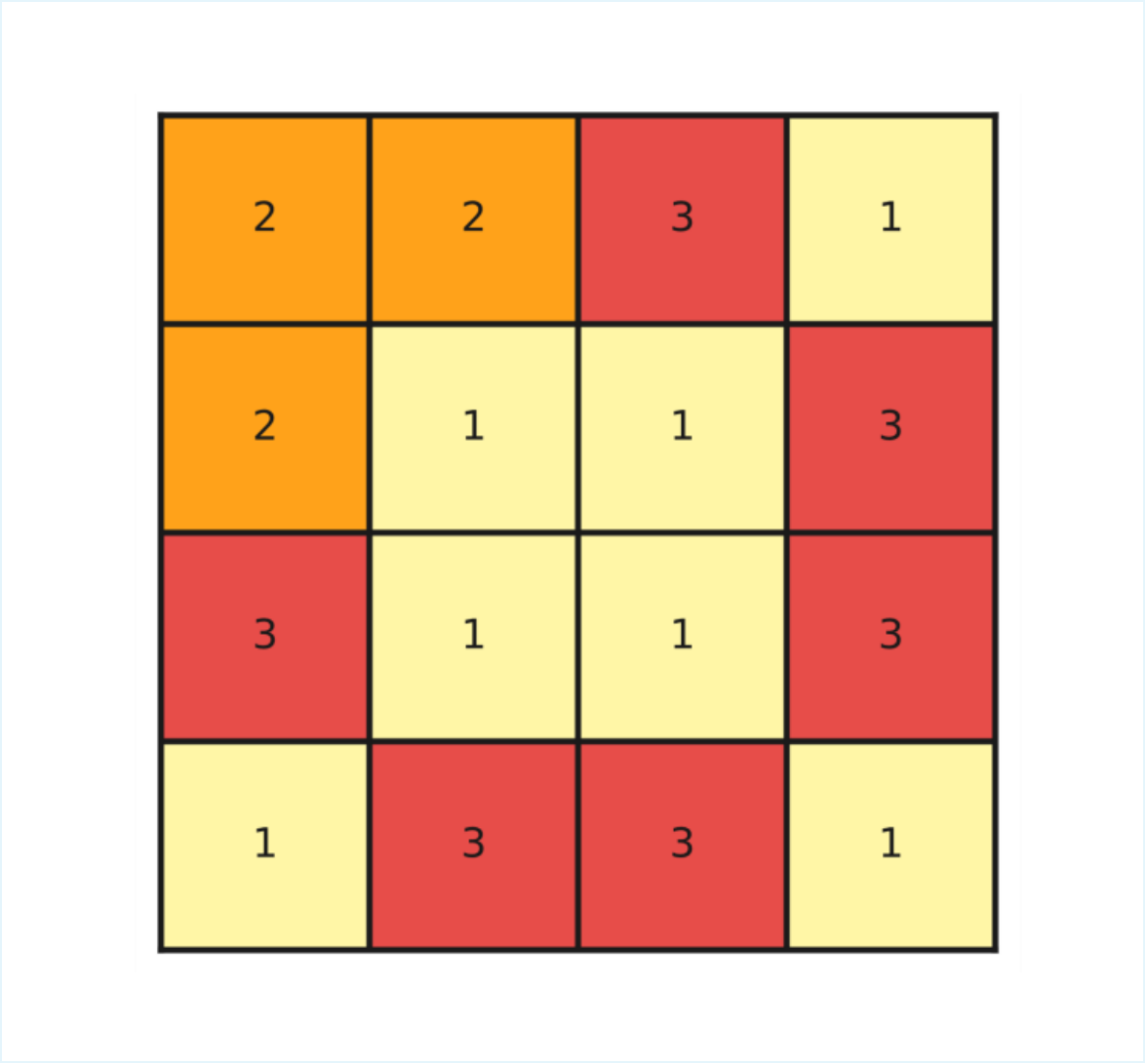}
        \caption[]%
        {{\small }}    
    \end{subfigure}
    \caption[ ]
    {\small Evolution of Algorithm \ref{Avalanche_analysis} applied to generator $\tilde{A}_C^{4, 2}$ in sandpile configuration $\gamma_i \eta$.} 
    \label{Figs_C_small_even}
\end{figure}
It follows that $w_{\tilde{A}_C^{4,2}}(\gamma_i \eta) = 15$ and $\tilde{A}_{C,1}^{4,2} = R_1(A^{(4)})$. From the second iteration of Algorithm \ref{Avalanche_analysis}, we obtain $w_{\tilde{A}_{C,1}^{4,2}}((\gamma_i \eta)^{\tilde{A}_C^{4,2}}) = 4$ and $\{v \in \tilde{A}^{4,2}_{C,1}|((\gamma_i \eta)^{\tilde{A}_C^{4,2}})^{\tilde{A}_{C,1}^{4,2}}(v) = 3\} = \emptyset$. Inserting this into expression (\ref{exp_av_size_rec_eq}) now yields
\begin{equation*}
    \mathbb{E}[X(\gamma_i \eta)|Y \in \tilde{A}_C^{4,2}] = \dfrac{241}{15},
\end{equation*}
which aligns with expression (\ref{Exp_av_size_C_2}).

We proceed to assume that expression (\ref{Exp_av_size_C_2}) is valid for $N = n-2$ and all $k = 2, \ldots, \lceil (n-2)/2\rceil$ for some $n = 5, 6, \ldots$. We show that this implies its correctness for $N = n$ and all $k = 2, \ldots, \lceil n/2 \rceil$ by means of an embedded induction argument over $k$. We start by considering the case $k = \lceil n/2 \rceil$. The result of applying Algorithm \ref{Avalanche_analysis} to the case $N = 7$ and $k = 4$ is illustrated in Figure \ref{Figs_C}.

 \begin{figure}[H]
    \centering
    \captionsetup[subfigure]{justification=centering, labelformat=empty, singlelinecheck=false, width=0.6\linewidth}

    \begin{subfigure}[b]{0.45\textwidth}
        \centering
        \includegraphics[width=0.7\textwidth]{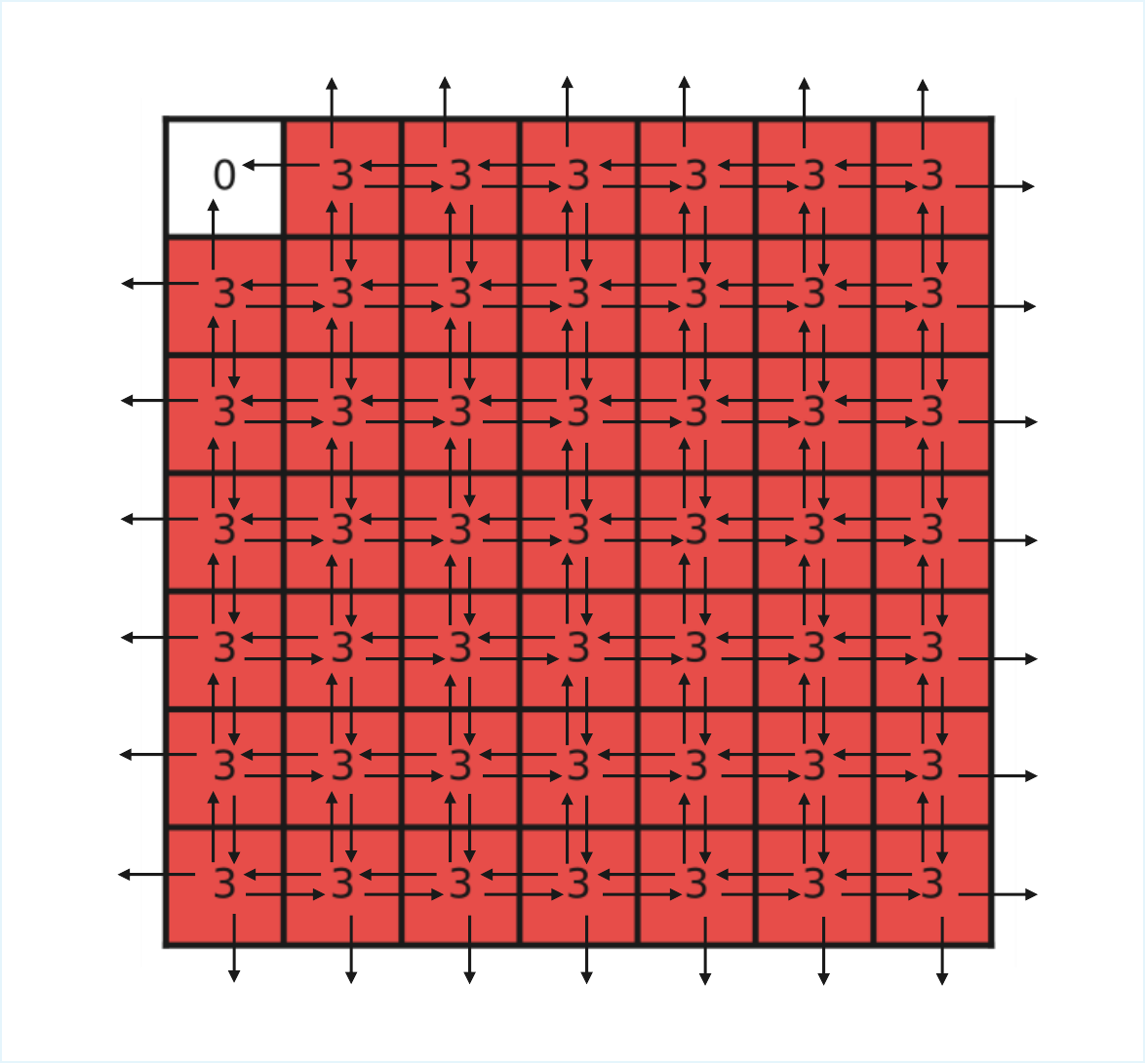}
        \caption[]%
        {{\small }}    
        \label{C}
    \end{subfigure}
    \begin{subfigure}[b]{0.45\textwidth}  
        \centering 
        \includegraphics[width=0.7\textwidth]{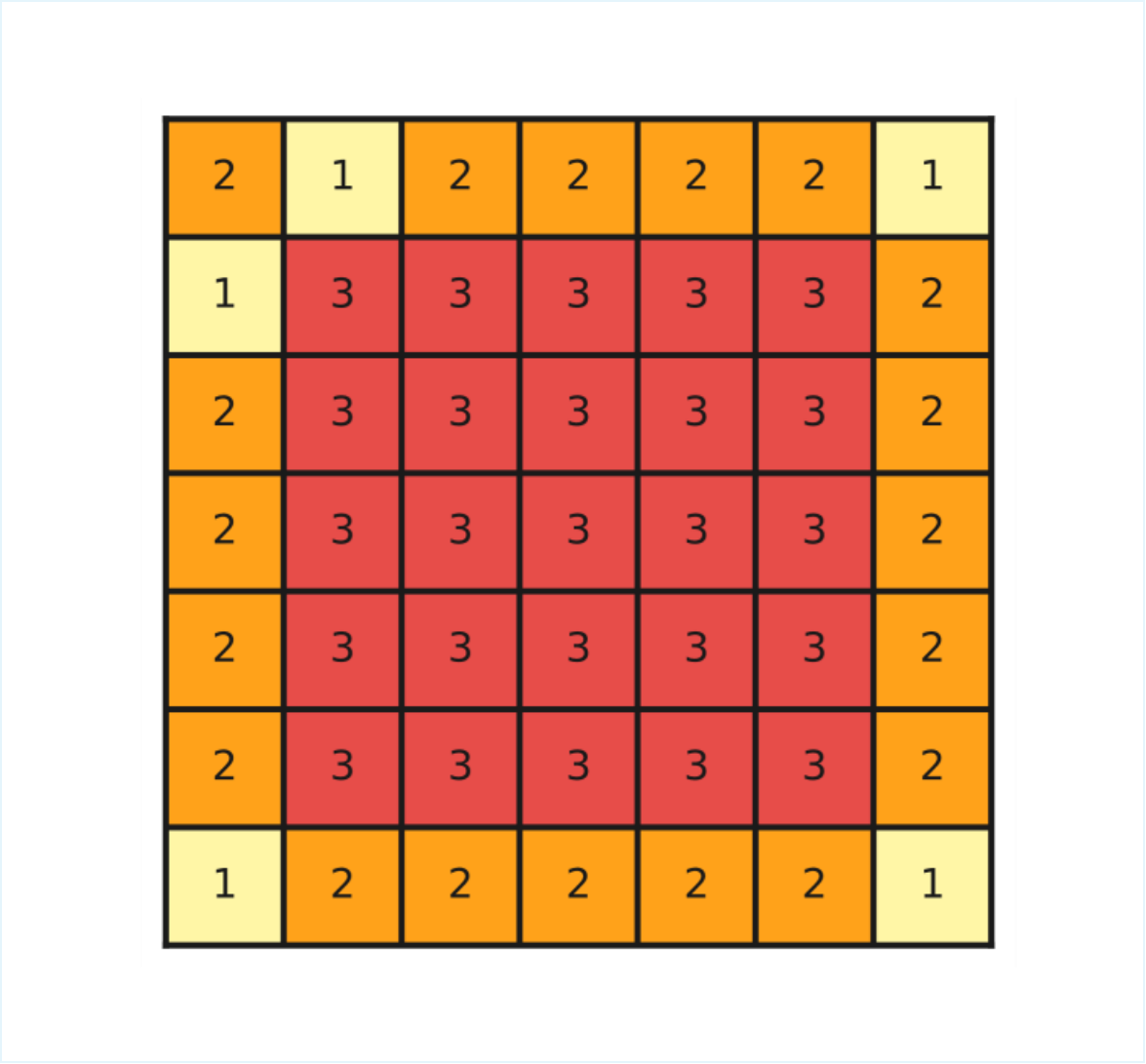}
        \caption[]%
        {{\small}}    
        \label{C_after}
    \end{subfigure}
    \caption{Evolution of Algorithm \ref{Avalanche_analysis} applied to generator $\tilde{A}_C^{7,4}$ in sandpile configuration $\gamma_i \eta$.} 
    \label{Figs_C}
    \end{figure}
Let the neighbours of $v_i$ be denoted by $v_{i_1}$ and $v_{i_2}$. In accordance with Algorithm \ref{Avalanche_analysis}, we initialize the directed graph $(V,E)$ with $E = \emptyset$ and add edges to $E$ from $v$ to each of its neighbours for each $v \in \tilde{A}_C^{n, \lceil n/2 \rceil}$. Observe that at this point, we have $(\gamma_i \eta)(v) + \text{indeg}(v) - \text{outdeg}(v) < 4$ for all $v \in V$. The resulting directed graph for the case $N = 7, k = 4$ is depicted in Figure \ref{Figs_C} (left). It follows that 
\begin{equation*}
    (\gamma_i \eta)^{\tilde{A}_C^{n, \lceil n/2 \rceil}}(v) = \begin{cases}
        3, &\text{if } v \in \text{int}(A^{(n)}), \\
        2, &\text{if } v \in (\delta^{\text{in}}(A^{(n)}) \setminus \{v_{i_1}, v_{i_2}\}) \cup \{v_i\}, \\
        1, &\text{if } v \in (C(A^{(n)}) \setminus \{v_i\}) \cup \{v_{i_1}, v_{i_2}\}.
    \end{cases}
\end{equation*}
This implies that $w_{\tilde{A}_C^{n, \lceil n/2 \rceil}}(\gamma_i \eta) = n^2-1$ and $\tilde{A}_{C,1}^{n, \lceil n/2 \rceil} := \{v \in \tilde{A}_C^{n, \lceil n/2 \rceil}|(\gamma_i \eta)^{\tilde{A}_C^{n, \lceil n/2 \rceil}}(v) = 3\} = \text{int}(A^{(N)})$. Using expression (\ref{exp_av_size_rec_eq}), the fact that $\text{int}(A^{(N)})$ is a square generator of size $(n-2) \times (n-2)$ and expression (\ref{exp_av_size_square}), we now obtain
\begin{align}
    \nonumber \mathbb{E}[X(\gamma_i \eta)|Y \in \tilde{A}^{n, \lceil n/2 \rceil}_C] &= n^2-1 + \dfrac{(n-2)^2}{n^2-1} \dfrac{3(n-2)^4+15(n-2)^3+20(n-2)^2-8}{30(n-2)} \\
    \label{Exp_av_size_C_int} &= \dfrac{3n^4 + 18n^3 + 38n^2 - 22n - 30}{30(n+1)}.
\end{align}
Note that for odd $n$, expression (\ref{Exp_av_size_C_int}) is equivalent to the first case in expression (\ref{Exp_av_size_C_2}) with $k = (n+1)/2$ and that for even $n$, it is equivalent to the second case in expression (\ref{Exp_av_size_C_2}) with $k = n/2$. Hence, expression (\ref{exp_av_size_square}) agrees with expression (\ref{Exp_av_size_C_2}) for $k = \lceil n/2 \rceil$. 

We now make the additional assumption that expression (\ref{Exp_av_size_C_2}) holds for $N = n$ and $k = \ell + 1$ for some $\ell~=~2, \ldots, \lceil N/2 \rceil -1$ and show that this implies its validity for $k = \ell$. Together with the fact that (\ref{Exp_av_size_C_2}) holds for $N = n$ and $k = \lceil n/2 \rceil$, this finalizes the embedded induction argument, establishing the statement for $N = n$ and all $k = 2, \ldots, \lceil n/2 \rceil$. The directed graph obtained from Algorithm \ref{Avalanche_analysis} and the sandpile configuration $(\gamma_i \eta)^{\tilde{A}_C^{n, \ell}}$ for the case $n = 7, \ell = 3$ are depicted in Figure \ref{Figs_C_case2}.

\begin{figure}[H]
   \begin{subfigure}[b]{0.45\textwidth}
        \centering
        \includegraphics[width=0.7\textwidth]{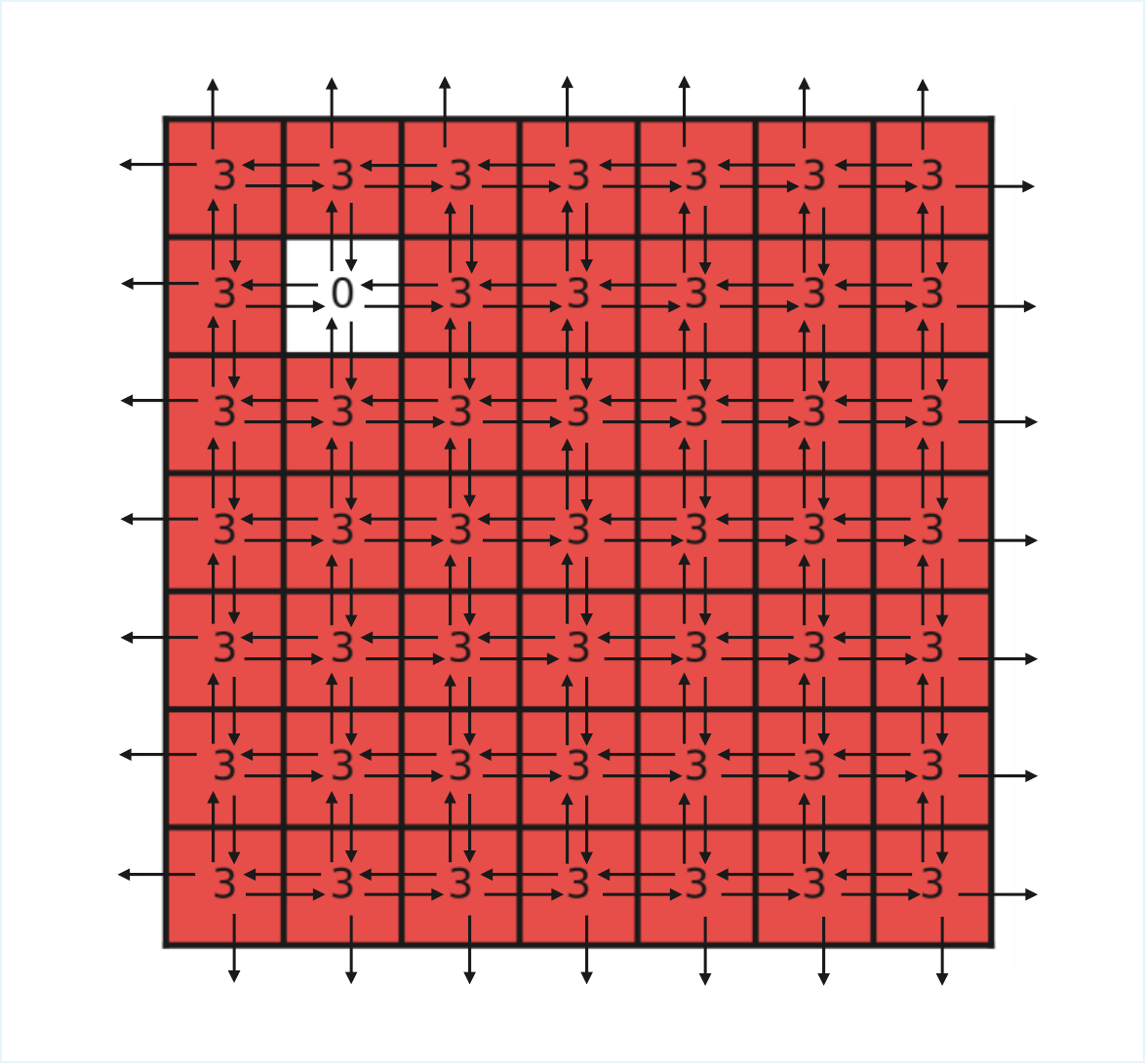}
        {{\small }}    
        \label{C_case2}
    \end{subfigure}
    \begin{subfigure}[b]{0.45\textwidth}  
        \centering 
        \includegraphics[width=0.7\textwidth]{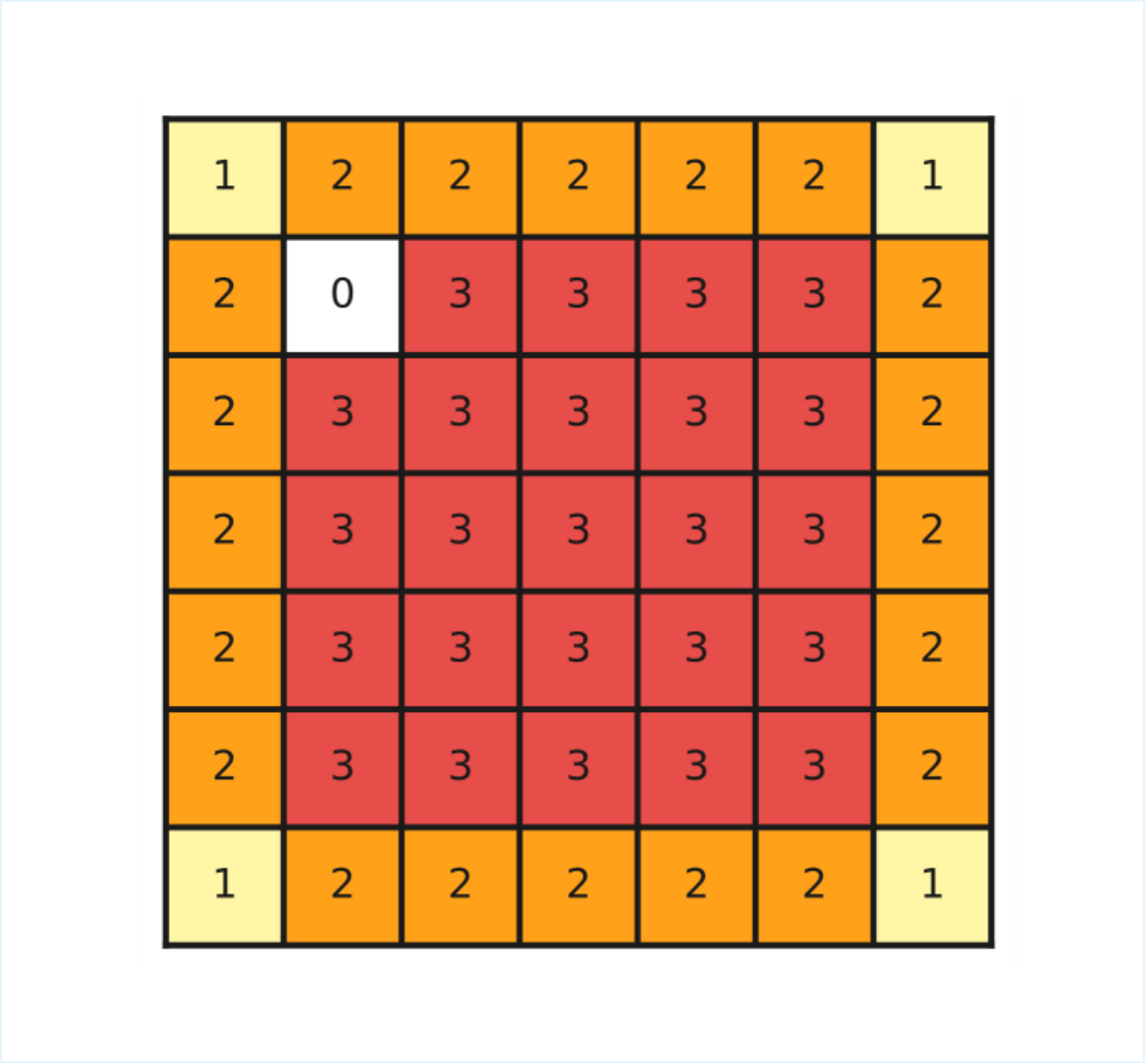}
        {{\small}}    
        \label{C_after_case2}
    \end{subfigure}
    \caption{Evolution of Algorithm \ref{Avalanche_analysis} applied to generator $\tilde{A}_C^{7,3}$ in sandpile configuration $\gamma_i \eta$.} 
    \label{Figs_C_case2}
    \end{figure}
After initializing the directed graph $(V, E)$ with $E = \emptyset$, we again add edges to $E$ from $v$ to each of its neighbours for each $v \in \tilde{A}_C^{n, \ell}$. Note that after this procedure, we have $(\gamma_i \eta)(v_i) + \text{indeg}(v_i) - \text{outdeg}(v_i) = 4$. Hence, we append edges to $E$ from $v_i$ to each of its neighbours. Now, we obtain $(\gamma_i \eta)(v) + \text{indeg}(v) - \text{outdeg}(v) < 4$ for all $v \in V$. The directed graph resulting from Algorithm \ref{Avalanche_analysis} is shown in Figure \ref{Figs_C_case2} (left). We derive that    
\begin{equation*}
    (\gamma_i \eta)^{\tilde{A}^{n, \ell}_C}(v) = \begin{cases}
        0, &\text{if } v = v_i, \\
        1, &\text{if } v \in C(A^{(N)}), \\
        2, &\text{if } v \in \delta^{\text{in}}(A^{(N)}), \\
        3, &\text{if } v \in \text{int}(A^{(N)}) \setminus \{v_i\}.
    \end{cases}
\end{equation*}
It follows that $w_{\tilde{A}_C^{n, \ell}}(\gamma_i \eta) = n^2$ and $\tilde{A}^{n, \ell}_{C,1}  := \{v \in \tilde{A}_C^{n, \ell}|(\gamma_i \eta)^{\tilde{A}_C^{n, \ell}}(v) = 3\} = \text{int}(A^{(N)}) \setminus \{v_i\}$. Inserting this in expression (\ref{exp_av_size_rec_eq}) and using the induction hypothesis, we now obtain
\begin{align*}
    & \mathbb{E}[X(\gamma_i \eta)|Y \in \tilde{A}_C^{n, \ell}] = n^2 + \dfrac{(n-2)^2-1}{n^2-1} \mathbb{E}[X((\gamma_i \eta)^{\tilde{A}_C^{n, \ell}})| Y \in \tilde{A}^{n, \ell}_{C,1}] \\
    &= n^2 + \dfrac{(n-2)^2-1}{n^2-1} \dfrac{3(n-2)^5+15(n-2)^4+15(n-2)^3-15(n-2)^2-18(n-2)+10(4\ell^3 - 24 \ell^2+23\ell - 3)}{30((n-2)^2-1)} \\
    &= \dfrac{3n^5 + 15n^4 + 15n^3 - 15n^2 - 18n + 10(4\ell^3 - 24 \ell^2 + 23 \ell - 3)}{30(n^2-1)}
\end{align*}
if $n$ is odd and
\begin{align*}
    & \mathbb{E}[X(\gamma_i \eta)|Y \in \tilde{A}_C^{n, \ell}] = n^2 + \dfrac{(n-2)^2-1}{n^2-1} \mathbb{E}[X((\gamma_i \eta)^{\tilde{A}_C^{n,\ell}})|Y \in \tilde{A}^{n, \ell}_{C,1}] \\
    &= n^2 + \dfrac{(n-2)^2-1}{n^2-1} \dfrac{3(n-2)^5 + 15(n-2)^4 + 15(n-2)^3 - 15(n-2)^2 - 18(n-2) + 10(4 \ell^3 - 18 \ell^2 + 2 \ell + 3)}{30((n-2)^2-1)} \\
    &= \dfrac{3n^5 + 15n^4 + 15n^3 - 15n^2 - 18n + 10(4\ell^3 - 18\ell^2 + 2\ell + 3)}{30(n^2-1)}
\end{align*}
if $n$ is even. This yields the correctness of expression (\ref{Exp_av_size_C_2}) for $N = n$ and all $\ell = 2, \ldots, \lceil n/2 \rceil$. The full induction argument now establishes the validity of expression (\ref{Exp_av_size_C_2}), and thus of expression (\ref{Exp_av_size_C}), for all $N \geq 3$ and all $k~=~2, \ldots, \lceil n/2 \rceil$.

\end{document}